\documentclass[reqno]{amsart}
\usepackage{graphicx} % Required for inserting images
\usepackage{subfiles}
\usepackage{enumitem}
\usepackage[T1]{fontenc}
\usepackage{dsfont}
\newtheorem{theorem}{Theorem}[section]
\newtheorem{lemma}{Lemma}[section]
\newtheorem{proposition}{Proposition}[section]
\newtheorem{corollary}{Corollary}[section]
\newtheorem{notation}{Notation}[section]
\newtheorem{definition}{Definition}[section]
\theoremstyle{remark}
\newtheorem{remark}{Remark}[section]
\usepackage{physics}
\usepackage{mathtools}
\usepackage{amsmath, amssymb}
\def\undertilde#1{\mathord{\vtop{\ialign{##\crcr
$\hfil\displaystyle{#1}\hfil$\crcr\noalign{\kern1.5pt\nointerlineskip}
$\hfil\widetilde{}\hfil$\crcr\noalign{\kern1.5pt}}}}}
\setlist[enumerate,2]{label=\arabic*)} \newcommand*\diff{\mathop{}\!\mathrm{d}}
\newcommand{\R}{\mathbb{R}}
\newcommand{\eps}{\epsilon}
\newcommand{\N}{\mathbb{N}}
\newcommand{\ind}{\mathds{1}}
\newcommand{\phitan}{\phi^{\natural}}
\newcommand{\phitaninverse}{\phi^{-\natural}}
\newcommand{\phidiss}{\phi_0^\flat}
\newcommand{\phicom}{\phi_{\phitan(u_L)}^\sharp}
\newcommand{\statespace}{\mathcal{I}_0}
\newcommand{\weakone}{\mathcal{S}_{\text{weak},u_L, \eps}}
\newcommand{\weaktwo}{\mathcal{S}_{\text{weak}}}
\newcommand{\Pismall}{\Pi_{C,s_0}}
\newcommand{\Qsmall}{Q_{C,s_0}}
\newcommand{\Rsmall}{R_{C,s_0}}
\newcommand{\Rsmallbd}{R_{C,s_0}^{bd}}
\newcommand{\Rsmallnot}{R_{C,s_0}^{0}}
\newcommand{\Rsmallplus}{R_{C,s_0}^{+}}
\newcommand{\Rsmallminus}{R_{C,s_0}^{-}}
\newcommand{\maximalspeedlarge}{\hat{\lambda}_1}
\newcommand{\maximalspeedsmall}{\hat{\lambda}_2}
\newcommand{\aplim}{\textnormal{ap }\lim}
\DeclarePairedDelimiter{\ceil}{\lceil}{\rceil}
\title[$L^2$-stability for Scalar Conservation Laws]{$L^2$-stability \& Minimal Entropy Conditions for Scalar Conservation Laws with Concave-Convex Fluxes}

\author{Jeffrey Cheng}
\address{Department of Mathematics, The University of Texas at Austin, Austin, TX 78712.}
\email{jeffrey.cheng@utexas.edu}

\date{\today}
\thanks{\textbf{Acknowledgements}: The author would like to thank his graduate advisor, Alexis Vasseur, for providing helpful conversations and support throughout. This work was partially supported by NSF-DMS Grant 1840314.}

\begin{document}

\begin{abstract}
In this paper, we study stability properties of solutions to scalar conservation laws with a class of non-convex fluxes. Using the theory of $a$-contraction with shifts, we show $L^2$-stability for shocks among a class of large perturbations, and give estimates on the weight coefficient $a$ in regimes where the shock amplitude is both large and small. Then, we use these estimates as a building block to show a uniqueness theorem under minimal entropy conditions for weak solutions to the conservation law via a modified front tracking algorithm. The proof is inspired by an analogous program carried out in the $2 \times 2$ system setting by Chen, Golding, Krupa, \& Vasseur.
\end{abstract}
\maketitle 

\tableofcontents

\section{Introduction \& main results}
We consider a $1-$d scalar conservation law:
\begin{equation}\label{cl}
u_t+(f(u))_x=0 \indent t > 0, x \in \R,
\end{equation}
where $(t,x) \in \R^+ \times \R$ are time and space and $u \in \statespace \subset \R$ is the unknown. The state space $\statespace$ is assumed to be convex and bounded. Without loss of generality, we may assume $\statespace=[-M,M]$ for some $M >0$. We assume that $f \in C^4(\statespace)$. In this paper, we will consider the case where $f$ is concave-convex. This means $f$ verifies the following properties:
\begin{equation}\label{concaveconvex}
\begin{cases}
uf''(u) > 0 & (u \neq 0), \\
f'''(0) \neq 0, \\
\lim_{|u| \to +\infty}f'(u)=+\infty.
\end{cases}
\end{equation}
In particular, $f$ has only one inflection point, normalized to be $f''(0)=0$. Our guiding example is $f(u)=u^3$. We will consider classes of solutions to \eqref{cl} that are entropic for one (or more) convex entropies $\eta$, i.e. solutions which verify additionally:
\begin{equation}\label{entropic}
(\eta(u))_t+(q(u))_x \leq 0 \indent t > 0, x \in \R,
\end{equation}
in the sense of distributions. More specifically, we ask that for all $\psi \in C_0^\infty(\R^+ \times \R)$ verifying $\psi \geq 0$:
\[
\int_0^\infty \int_{-\infty}^\infty \bigl(\psi_t(t,x)\eta(u(t,x))+\psi_x(t,x)q(u(t,x))\bigr)\diff x \diff t \geq 0.
\]
\begin{notation}
Throughout, we work with a single fixed $C^3$ strictly convex entropy, denoted $\eta$, and a few Kru\v{z}kov entropies, denoted $\eta_k$, where $\eta_k(u):=|u-k|$. A weak solution to \eqref{cl} will be denoted ``$\eta$-entropic'' if it is entropic for the entropy $\eta$, and will be denoted ``$\eta_k-$entropic'' if it is entropic for $\eta_k$.
\end{notation}
It is well known that weak solutions to \eqref{cl} need not be unique. Kru\v{z}kov showed that solutions which are $\eta_k$-entropic for all $k$ are contractive against one other in $L^1$, and thus one recovers uniqueness if all the entropy conditions are imposed \cite{MR267257}. The main question we explore in this work is whether uniqueness can be recovered by imposing less entropy conditions. 
\par Dafermos \cite{MR546634} and DiPerna \cite{MR523630} showed a weak/strong stability result for Lipschitz solutions to \eqref{cl} by studying the evolution of the relative entropy:
\[
\eta(v|w)=\eta(v)-\eta(w)-\eta'(w)(v-w),
\]
where $w$ is the strong solution, and $v$ is a $\eta$-entropic weak solution. Note that in this case, $v$ is assumed to be entropic for merely one entropy. Notice that if $u$ is a weak solution of \eqref{cl}, \eqref{entropic}, then $u$ verifies a full family of entropy inequalities. Indeed, we define $q(a;b)$, the relative entropy flux, as:
\[
q(a;b)=q(a)-q(b)-q'(b)(f(a)-f(b)).
\]
Then, for any $b \in \statespace$ constant, each $\eta(\cdot|b)$ is an entropy for \eqref{cl}, and each $\eta$-entropic $u$ satisfies:
\begin{equation}\label{relativeentropic}
(\eta(u|b))_t+(q(u;b))_x \leq 0,
\end{equation}
in the sense of distributions. Similar to Kru\v{z}kov's theory \cite{MR267257}, \eqref{relativeentropic} provides a family of entropies measuring the distance of the solution to any constant state. Thus, the relative entropy method may be considered an analog of Kru\v{z}kov's theory applicable to systems of conservation laws (although we treat the scalar case here, where Kru\v{z}kov's methods are available). The main difference is that the relative entropies are $L^2$-based rather than $L^1$-based. 

\par Recently, the study of the time evolution of the relative entropy has been extended to the setting in which $w$ is a shock, under a genuine nonlinearity condition. In the scalar case, this corresponds to $f$ being strictly convex. Leger showed $L^2$-contraction for shocks among a class of solutions verifying merely one entropy condition \cite{MR2771666}. More recently, a weighted $L^2$-stability property for shocks, the so-called ``a-contraction'' property, has been shown, which can be seen as an extension of the result of Leger to systems (see \cite{MR3537479} and \cite{MR3519973}. See also \cite{MR4176349} for extensions to balance laws). We also briefly mention the extension of the $a$-contraction theory to systems of conservation laws with viscosity, which was recently used to prove the uniqueness of solutions to the isentropic Euler equations in the class of inviscid limits from Navier-Stokes \cite{MR4195742}, \cite{NSVS}. Another result of note is the recent work of Huang, Wang, \& Zhang, who prove the $a$-contraction property for the cubic flux in the viscous setting, which is the chief example in the class of fluxes we address in this article \cite{Wangcubic}.
\par In this work, we adapt the so-called ``relative entropy'' program to two novel scenarios. Primarily, we consider for the first time at the hyperbolic level a flux which is not genuinely nonlinear. Secondly, we construct the modified front tracking algorithm in a large-BV setting. The first main result is to show the $a$-contraction property for shocks via the relative entropy method in the case where the flux $f$ is concave-convex, among a class of weak solutions verifying minimal entropy conditions.

\par We restrict our study to solutions verifying the so-called Strong Trace property, which is now classical for the $a-$contraction theory:
\begin{definition}\label{strongtrace}
Let $u \in L^\infty(\R^+ \times \R)$. We say that $u$ verifies the Strong Trace property if for any Lipschitzian curve $t \mapsto h(t)$, there exists two bounded functions $u_-, u_+ \in L^\infty(\R^+)$ such that for any $T > 0$:
\begin{align*}
&\lim_{n\to \infty}\int_0^T\sup_{y \in (0,\frac{1}{n})}|u(t,h(t)+y)-u_+(t)|\diff t \\
&=\lim_{n\to \infty}\int_0^T\sup_{y \in (-\frac{1}{n},0)}|u(t,h(t)+y)-u_-(t)|\diff t=0.
\end{align*}
\end{definition}
\begin{notation}
For convenience, we will use later the notation $u_+(t)=u(t,h(t)+)$, and $u_-(t)=u(t,h(t)-)$. 
\end{notation}

\par Now, we can define the wildest spaces of solutions that we consider in the paper. The space of weak solutions that a fixed shock $(u_L, u_R)$ has stability in will depend on $u_L$. To this end, for $\eps > 0$, consider the class:
\begin{gather*}
\weakone=\{u \in L^\infty(\R^+ \times \R:\statespace)|\text{weak solution to } \eqref{cl}, \eta,\eta_{\phitan(u_L)}-\text{entropic, } \\
\text{ verifying Definition \ref{strongtrace}, and } (u_-(t),u_+(t)) \not \in B(\phidiss(u_L),\eps) \times B(u_L, \eps) \\
\text{ for almost every $t > 0$ for all Lipschitz curves $h(t)$}\}.
\end{gather*}
For definitions of the auxiliary functions $\phitan$, $\phidiss$, see \eqref{phitan}, \eqref{phidiss}.

\par Now, we may state our first main theorem. It gives a contraction (up to shift) for a suitable pseudo-distance between a fixed Ole\u{\i}nik shock and a wild solution.
\begin{theorem}\label{dissipationfnllargeshock}
Consider \eqref{cl}, with $f$ verifying \eqref{concaveconvex}. Let $u_L > 0$ and $u_R \in \mathcal{V}_0$ such that $(u_L, u_R)$ verifies the Ole\u{\i}nik condition. Then, there exists $0 < a^* < 1$ such that the following holds: let $a_1, a_2 > 0$ with $\frac{a_2}{a_1} < a^*$. Then, for any $u \in \weakone$, any $\overline{t} \in [0,\infty)$, and any $x_0 \in \R$, there exists a Lipschitz path $h:[\overline{t},\infty) \to \R$, with $h(\overline{t})=x_0$, such that the following dissipation functional verifies:
\begin{equation}\label{sufficientlarge}
a_1(\dot{h}(t)\eta(u_-(t)|u_L)-q(u_-(t);u_L))-a_2(\dot{h}(t)\eta(u_+(t)|u_R)-q(u_+(t);u_R)) \leq 0,
\end{equation}
for almost every $t \in [\overline{t},\infty)$. Here, $a^*$ depends only on the system, and continuously on $\eps, s_0=|u_L-u_R|$, $d(u_L,0)=|u_L-0|$, and $d(u_R, \phitan(u_L))$ if $d(u_R, \phitan(u_L))>0$.
\end{theorem}
\begin{notation}
In the above, and in all that follows, ``dependence on the system'' will be interpreted as a dependence on $f, \eta$, and $\statespace$, i.e. there will be implicit dependence on $M$, where $\statespace=[-M,M]$.
\end{notation}
As a corollary, we obtain the following contraction property:
\begin{corollary}\label{thmlargeshock}
Consider \eqref{cl}, with $f$ verifying \eqref{concaveconvex}. Let $u_L > 0, u_R \in \mathcal{V}_0$ such that $(u_L, u_R)$ verifies the Ole\u{\i}nik condition. Then, there exists $0 < a^* < 1, C > 0$ such that the following holds: let $a_1, a_2 > 0$ with $\frac{a_2}{a_1} < a^*$. For any $u \in \weakone$, there exists a Lipschitz path $h(t)$, with $h(0)=0$, such that for any $t > 0$, the pseudo-norm:
\[
E_t(u)=a_1 \int_{-\infty}^0\eta(u(t,x+h(t))|u_L)\diff x+a_2 \int_0^\infty \eta(u(t,x+h(t))|u_R)\diff x,
\]
is non-increasing in time. Especially, for every $t > 0$:
\[
||u(t,\cdot+h(t))-S||_{L^2(\R)} \leq C||u_0-S||_{L^2(\R)},
\]
where $u_0=u(t=0)$ and $S(x)=u_L$ for $x < 0$ and $S(x)=u_R$ for $x > 0$.
\end{corollary}
Assuming the contraction property holds, the stability in $L^2$ is a simple consequence of the following lemma:
\begin{lemma}[from \cite{MR2807139}, \cite{MR2508169}]\label{quadraticentropy}
For any fixed compact set $V \subseteq \statespace$, there exists $c^*, c^{**} > 0$ such that for all $(u,v) \in \statespace \times V$:
\begin{equation*}
    c^*|u-v|^2 \leq \eta(u|v)\leq c^{**}|u-v|^2
\end{equation*}
The constants $c^*, c^{**}$ depend on bounds on the second derivative of $\eta$ in $V$, and on the continuity of $\eta$ on $\statespace$.
\end{lemma}
By keeping track of the dependence of $a$ on the aforementioned quantities carefully, we will obtain as a second corollary that for $u_L$ contained in a compact set localized to the region where $f$ is convex, we will be able to take $a=a^* > 0$ uniformly for all shocks, although we will also need to restrict our class of weak solutions slightly more. Roughly speaking, as $\weakone$ depends on $u_L$, we need to consider a class that is contained in the intersection of all $\weakone$ over $u_L$ ranging over some compact set. Fix a closed interval $[\underline{b},\overline{b}]$, with $0 < \underline{b} < \overline{b}$. Then, we define:
\begin{gather*}
\weaktwo=\{u \in L^\infty(\R^+ \times \R:\statespace)|\text{weak solution to } \eqref{cl}, \eta,\eta_{\phitan(\underline{b})}-\text{entropic, } \\
\text{verifying Definition \ref{strongtrace}, and }
u_0 \geq \phitan(\underline{b}) \}.
\end{gather*}
Now, we can state the sharpened result:
\begin{corollary}\label{uniformalarge}
Consider \eqref{cl} with $f$ verifying \eqref{concaveconvex}. Fix a closed interval $[\underline{b},\overline{b}]$, with $0 < \underline{b} < \overline{b}$, and $\tilde{s_0} > 0$. Then there exists $0 < a^* < 1$ such that the following holds. Let $a_1, a_2 > 0$ with $\frac{a_2}{a_1} \leq a^*$. Consider any Ole\u{\i}nik shock $(u_L,u_R)$ with $u_L,u_R \in [\underline{b},\overline{b}], |u_L-u_R| \geq \tilde{s_0}$, any $u \in \weaktwo$, any $\overline{t} \in [0,\infty)$, and any $x_0 \in \R$. Then, there exists a Lipschitz path $h:[\overline{t},\infty) \to \R$, with $h(\overline{t})=x_0$, such that the following dissipation functional verifies:
\begin{equation*}
a_1(\dot{h}(t)\eta(u_-(t)|u_L)-q(u_-(t);u_L))-a_2(\dot{h}(t)\eta(u_+(t)|u_R)-q(u_+(t);u_R)) \leq 0,
\end{equation*}
for almost every $t \in [\overline{t},\infty)$. Here, $a^*$ depends only on the system, $\underline{b},\overline{b}$, and $\tilde{s_0}$. Moreover, there is a maximal shock speed $\maximalspeedlarge$, i.e. $|\dot{h}(t)| \leq \maximalspeedlarge$ for almost every $t \in [\overline{t},\infty)$.
\end{corollary}
For use in a modified front tracking scheme later, we will need sharper estimates on $a^*$ in the small shock regime, i.e. the setting where $|u_L-u_R| < \tilde{s_0}$ (where $\tilde{s_0}$ will be determined by the following theorem). The need for this was first realized by Chen, Krupa \& Vasseur \cite{MR4487515} and established by Golding, Krupa, \& Vasseur \cite{MR4667839}. More precisely, we need $|a^*-1|$ to be proportional to the size of the shock. This leads us to our second main theorem.
\begin{theorem}\label{theoremsmallshock}
Consider \eqref{cl} with $f$ verifying \eqref{concaveconvex}. Fix a closed interval $[\underline{b}, \overline{b}]$, with $0 < \underline{b} < \overline{b}$. Then, there exist constants $C_0, \tilde{s_0}>0$ such that the following is true: consider any Ole\u{\i}nik shock $(u_L, u_R)$ with $u_L, u_R \in [\underline{b}, \overline{b}]$, and $|u_L-u_R|=s_0 < \tilde{s_0}$. Then, for any $a_1, a_2 > 0$ satisfying:
\begin{equation}\label{weightestimate}
1+\frac{C_0}{2}s_0 \leq \frac{a_1}{a_2} \leq 1+2C_0s_0,
\end{equation}
and any $u \in \weaktwo$, any $\overline{t} \in [0,\infty)$, and any $x_0 \in \R$, there exists a Lipschitz shift function $h: [\overline{t} \to \infty) \to \R$, with $h(\overline{t})=x_0$, such that the following dissipation functional verifies:
\begin{equation*}
a_1(\dot{h}(t)\eta(u_-(t)|u_L)-q(u_-(t);u_L))-a_2(\dot{h}(t)\eta(u_+(t)|u_R)-q(u_+(t);u_R)) \leq 0,
\end{equation*}
for almost every $t \in [\overline{t},\infty)$. Moreover, there is a maximal shock speed $\maximalspeedsmall$, i.e. $|\dot{h}(t)| \leq \maximalspeedsmall$ for almost every $t \in [\overline{t},\infty)$.
\end{theorem}
As an application of these estimates, we prove a uniqueness theorem for solutions to \eqref{cl} verifying minimal entropy conditions. In the case where the flux is convex, it is well known that if a weak solution $u$ is entropic for one $C^1$ strictly convex entropy, then it is in fact the Kru\v{z}kov solution. This was first proven by Panov \cite{MR1296003}, and later by De Lellis, Otto, \& Westdickenberg \cite{MR2104269} using the correspondence between entropy solutions of the conservation law and viscosity solutions of the integrated equation (see also \cite{MR4623215}, in which the results of the aforementioned articles were recently improved). Later, it was again proven by Krupa \& Vasseur \cite{MR3954680}, who introduced the ideology that $L^2$-stability for shock waves can be extended to stability results for more general solutions, via a strategy of showing that a $\eta$-entropic weak solution $u$ verifies Ole\u{\i}nik's ``condition E'', which is known to imply that $u$ is then in fact the Kru\v{z}kov solution \cite{MR94541}, \cite{MR1707279}. 
\par Much less is known when the flux is not convex. In some cases, traveling waves have been shown to converge in the diffusive-dispersive limit to ``non-classical'' or ``undercompressive'' shock waves, which are not the Ole\u{\i}nik shock but are nevertheless entropic for at least one entropy \cite{MR1475777}, \cite{MR1920633}. Uniqueness of Riemann solutions in the class of piecewise smooth functions verifying a single entropy condition and an appropriate kinetic relation was shown by LeFloch \cite{MR1927887}, while uniqueness of weak solutions in a class of piecewise smooth solutions verifying a certain regularity property was shown by Ballou \cite{MR435615}. Regularity of the Kru\v{z}kov solution was studied by Cheng, Dafermos, \& Hoff \cite{MR785530}, \cite{MR818862}, \cite{MR688972}. However, to the best of our knowledge, there are no uniqueness or regularity results under minimal entropy conditions or for less regular solutions. This is what we aim to address in our last theorem.
\par The strategy will be to use the $L^2$-stability result of Theorem \ref{theoremsmallshock} to control the $L^2$ distance between a weak solution and a BV function which is to be determined. This allows for the transfer of desirable properties of the BV function onto the weak solution. The discretization scheme of Krupa \& Vasseur used in \cite{MR3954680} fails for the non-convex case, so we use a modified front tracking algorithm introduced in the $2 \times 2$ system setting by Chen, Krupa, \& Vasseur \cite{MR4487515}. The key difference is that as we are in the scalar case, the front tracking algorithm is valid for large-$BV$ initial data \cite{MR303068}. One of the main contributions of this work is to construct the weight $a$ in the modified front tracking algorithm in a large-BV setting. The last main result we prove is the following stability theorem:
\begin{theorem}\label{theoremstability}
Consider \eqref{cl} where $f$ verifies \eqref{concaveconvex}. Let $u^0 \in BV(\R)$ and \newline
$Range(u^0) \subset [\underline{b}, \overline{b}]$, where $0 < \underline{b} < \overline{b}$. Let $u$ be the Kru\v{z}kov solution to \eqref{cl} with initial data $u^0$. Let $u_n \in \weaktwo$ be a sequence of wild solutions, uniformly bounded in $L^\infty(\R^+ \times \R)$, with initial values $u^0_n \in L^\infty(\R)$. If $u^0_n$ converges to $u^0$ in $L^2(\R)$, then for every $T > 0, R> 0$, $u_n$ converges to $u$ in $L^\infty([0,T];L^2(-R,R))$. In particular, $u$ is unique in the class $\weaktwo$. 
\end{theorem}
As any $BV$ function verifies the Strong Trace property (Definition \ref{strongtrace}), we obtain the following theorem as a consequence:
\begin{theorem}\label{niceuniqueness}
Consider \eqref{cl} where $f$ verifies \eqref{concaveconvex}. Let $u^0 \in BV(\R)$ and \newline
$Range(u^0) \subset [\underline{b}, \overline{b}]$, where $0 < \underline{b} < \overline{b}$. Then, any solution with initial value $u^0$ is unique in the class:
\[
\{v \in L^\infty([0,T];BV(\R)) \text{ for all } T>0 \text{ and } \eta,\eta_{\phitan(\underline{b})}-entropic\}
\]
with the same initial data.
\end{theorem}
This shows that in certain cases, less than the complete set of entropy conditions need to be imposed in order to obtain uniqueness, even in the case when the flux is not convex. This is sharp in the sense that one $C^3$ convex entropy is not enough to grant uniqueness for concave-convex fluxes, even in the case where the initial data is localized to the convex region (see Appendix \ref{app1}). Finally, we note that the methods in this article can be used to prove analogous results for Ole\u{\i}nik shocks with left-hand states in the concave region of the flux, and solutions with initial data localized to the concave region, although we do not address that here. 
\par The article is structured as follows. In Section \ref{sectionlarge}, we prove the most general result on $L^2$-stability for shock waves, Theorem \ref{dissipationfnllargeshock}. As a corollary, we obtain that the weight $a$ may be picked uniformly among shocks localized to a compact set in the convex region of the flux, the first half of the finer estimates needed in Section \ref{sectionstability}. In Section \ref{sectionsmall}, we prove Theorem \ref{theoremsmallshock}, providing the other half of the finer estimates needed. Finally, in Section \ref{sectionstability} we prove Theorem \ref{theoremstability} via a modified front tracking algorithm.

\section{Dissipation rates in the large shock regime}\label{sectionlarge}
\subsection{Outline of the proof of Theorem \ref{dissipationfnllargeshock}}
We fix a shock of interest $(u_L, u_R)$ with size $s_0=|u_L-u_R|>0$, which we wish to establish stability for. We desire to show \eqref{sufficientlarge}:
\begin{equation*}
\dot{h}(t)(a_1\eta(u_-(t)|u_L)-a_2\eta(u_+(t)|u_R))-a_1q(u_-(t);u_L)+a_2q(u_+(t);u_R) \leq 0,
\end{equation*}
for some appropriate shift $h$. An important subset of state space will be the following:
\begin{equation}\label{pia}
\Pi_a:=\{u \in \statespace|\frac{a_1}{a_2}\eta(u|u_L)-\eta(u|u_R) \leq 0\}.
\end{equation}
If the weak solution $u$ does not have any discontinuity at $h(t)$ and $u(t,h(t)) \in \partial \Pi_a$, then the term multiplied by $\dot{h}(t)$ in \eqref{sufficientlarge} vanishes. Thus, in order to have the non-positivity, we need that for every $u \in \partial \Pi_a$: 
\begin{equation}\label{Dcont}
D_{cont}(u):=(q(u;u_R)-\lambda(u)\eta(v|u_R))-\frac{a_1}{a_2}(q(u;u_L)-\lambda(u)\eta(u|u_L)) \leq 0,
\end{equation}
where $\lambda(u)=f'(u)$. In the large shock regime, we will denote $a:=\frac{a_2}{a_1}$ and show that for $a$ sufficiently small, $D_{cont} \leq 0$ for $u \in \Pi_a$ (not just on the boundary). This brings us to the first key proposition:
\begin{proposition}\label{Dcontlarge}
Consider \eqref{cl}, where $f$ verifies \eqref{concaveconvex}. Let $u_L > 0, u_R \in \statespace$ such that $(u_L, u_R)$ verifies the Ole\u{\i}nik condition. There exists $a^* > 0$ such that for any $a < a^*$, and any $u \in \Pi_a$:
\[
D_{cont}(u) \leq 0.
\]
Finally, $a^*$ depends only on the system and continuously on $s_0$ and $d(u_L,0)$, where $d(u_L,0)=|u_L-0|$.
\end{proposition}
In the above, it is assumed that $u$ has no discontinuity at $x=h(t)$. However, situations arise where $x=h(t)$ corresponds locally to a Rankine-Hugoniot discontinuity curve of $u$. At this time, $(u(t,h(t)-),u(t,h(t)+,\dot{h}(t))$ corresponds to a shock $(u_-, u_+, \sigma_{\pm})$. As $u \in \weakone$ is $\eta,\eta_{\phitan(u_L)}$-entropic, we may assume the shock will be as well. The next two propositions ensure that the $a$-contraction remains valid in these cases. For any such shock $(u_-, u_+, \sigma_{\pm})$, we define:
\begin{equation}\label{DRH}
D_{RH}(u_-,u_+,\sigma_\pm)=q(u_+;u_R)-\sigma_\pm \eta(u_+|u_R)-\frac{1}{a}(q(u_-;u_L)-\sigma_\pm \eta(u_-|u_L)).
\end{equation}
Notice that $D_{RH}(u_-, u_+, \sigma_\pm)$ reduces to $D_{cont}(u)$ in the case when $u_-=u_+=u$ and $\sigma_\pm=\lambda(u)$. The first proposition will control this quantity when $u_- \in \Pi_a$:
\begin{proposition}\label{DRHlarge1}
Consider \eqref{cl}, where $f$ verifies \eqref{concaveconvex}. Let $u_L > 0$ and $u_R \in \statespace$ such that $(u_L, u_R)$ verifies the Ole\u{\i}nik condition. There exists $a^* > 0$ such that for any $a < a^*$, and any $\eta,\eta_{\phitan(u_L)}$-entropic shock $(u_-, u_+, \sigma_\pm)$ with $u_- \in \Pi_a$:
\[
D_{RH}(u_-,u_+,\sigma_\pm) \leq 0.
\]
Here, $a^*$ depends only on the system and continuously on $s_0$, $d(u_L,0)$, and \newline 
$d(u_R, \phitan(u_L))$ if $d(u_R, \phitan(u_L))$ > 0.
\end{proposition}
The last proposition will control this quantity when $u_+ \in \Pi_a$, and $u_-$ is far away from $u_+$:
\begin{proposition}\label{DRHlarge2}
Consider \eqref{cl}, where $f$ verifies \eqref{concaveconvex}. Let $u_L > 0$ and $u_R \in \statespace$ such that $(u_L, u_R)$ verifies the Ole\u{\i}nik condition. Let $\eps > 0$. There exists $a^* > 0$ such that for any $a < a^*$ and any shock $(u_-, u_+, \sigma_\pm)$ such that $u_- \leq \phidiss(u_L)-\eps, u_+ \in \Pi_a$:
\[
D_{RH}(u_-,u_+,\sigma_\pm) \leq 0.
\]
Here, $a^*$ depends only on the system and continuously on $\eps,s_0$, and $d(u_L,0)$. 
\end{proposition}

\subsection{Preliminaries}
\subsubsection{Auxiliary functions}
In this section, we introduce some important functions related to entropy dissipation for the concave-convex flux $f$. A detailed exposition of most of these functions may be found in \cite{MR1927887}, so we omit the proofs of results that may be found therein. 
\begin{definition}
We define the tangent function $\phitan$ as follows:
\begin{equation}\label{phitan}
\phitan(u)=\begin{cases} 
v \text{ such that } v \neq u, f'(v)=\frac{f(u)-f(v)}{u-v} & u \neq 0, \\
0 & u=0.
\end{cases}
\end{equation}
\end{definition}
The function $\phitan$ is uniquely defined owing to the concave-convex nature of the flux, and $\phitan \in C^1({\statespace})$ because $f \in C^4(\statespace)$, via the implicit function theorem. It is also monotone decreasing and thus invertible. We denote the inverse $\phitaninverse$.  
\par Recall that for a convex entropy $\eta$, a shock wave $(u_-, u_+)$ is $\eta$-entropic if and only if $E_\eta(u_-, u_+) \leq 0$, where $E_\eta$ is the following entropy dissipation function:
\begin{equation}\label{entropydissipation}
E_\eta(u_-, u_+):=-\sigma_\pm(\eta(u_+)-\eta(u_-))+q(u_+)-q(u_-),
\end{equation}
where $\sigma_\pm=\frac{f(u_-)-f(u_+)}{u_--u_+}$ and $q$ is the entropy flux associated to $\eta$. We reference \cite{MR1927887} for the following theorem regarding $E_\eta$:
\begin{lemma}[Theorem 3.1 in \cite{MR1927887}]\label{entropydissipationlemma}
Let $\eta$ be $C^1$. For any left state $u_- > 0$, the function $E_\eta(u_-,\cdot)$ is monotone decreasing in $(-\infty,\phitan(u_-)]$ and monotone increasing in $[\phitan(u_-),\infty)$. Further, $E_\eta(u_-,u_-)=0$.
\end{lemma}
Notice that for a $C^1$ entropy $\eta$, this implies there are exactly two points where $E_\eta(u, \cdot)$ equals zero, one of which equals $u$. We will use this to implicitly define a zero entropy dissipation function:
\begin{definition} 
We define the zero entropy dissipation function as follows:
\begin{equation}\label{phidiss}
\phidiss(u)=\begin{cases}
v \text{ such that } v \neq u, \text{ and } E_\eta(u,v)=0 & u \neq 0, \\
0 & u=0.
\end{cases}
\end{equation}
\end{definition}
Again, $\phidiss$ is monotone decreasing and $\phidiss \in C^1({\statespace})$. 
\par Finally, we define the companion function to a fixed point of reference:
\begin{definition}
Define the companion function to $k$ as follows:
\[
\phi_k^\sharp=\begin{cases}
v \text{ such that } v \neq u, k, \frac{f(v)-f(u)}{v-u}=\frac{f(k)-f(u)}{k-u} & u \neq \phitaninverse(k), \\
k & u=\phitaninverse(k).
\end{cases}
\]
\end{definition}
\begin{remark}
We will most often use the companion function $\phicom$, for some $u_L > 0$:
\[
\phicom(u)=\begin{cases}
v \text{ such that } v \neq u, \phitan(u_L), \text{ and } \frac{f(v)-f(u)}{v-u}=\frac{f(\phitan(u_L))-f(u)}{\phitan(u_L)-u} & u \neq u_L, \\
\phitan(u_L) & u=u_L.
\end{cases}
\]
\end{remark}
For later purposes, we need a Lipschitz bound on $\phicom$ at $u_L$.
\begin{lemma}\label{companionlemma}
Let $u_L > 0$. Then, there exists $C> 0$ such that for any $u \in \statespace$ with $u \geq u_L$ and $\phicom(u) \in \statespace$:
\[
|\phicom(u)-\phicom(u_L)|\leq C|u-u_L|.
\]
Here, $C$ depends on the system and continuously on $d(u_L,0)$.
\end{lemma}
\begin{proof}
Fix $k:=\frac{\phitan(u_L)}{2} < 0$. By \eqref{concaveconvex}, there exists $c < 0$ such that  $f'' \leq c$ on $[-M, k]$ (note that $c:=\max_{x \in [-M,k]}f''(x)$ depends on $d(u_L,0)$). So, for any $x \in [-M, k]$, we obtain:
\[
f(x) \leq \frac{c}{2}(x-\phitan(u_L))^2+f'(\phitan(u_L))(x-\phitan(u_L))+f(\phitan(u_L))=:p(x).
\]
Next, for $u > u_L$, define $h$ as the companion function to $\phitan(u_L)$ associated to the barrier function $p$:
\[
h(u)=\begin{cases}
v \text{ such that } v \leq \phitan(u_L), \text{ and } \frac{p(v)-f(u)}{v-u}=\frac{f(\phitan(u_L))-f(u_L)}{\phitan(u_L)-u_L} & u \neq u_L, \\
\phitan(u_L) & u=u_L.
\end{cases}
\]
Note that for $u > u_L$, $|\phicom(u)-\phitan(u_L)| \leq |h(u)-\phitan(u_L)|$, so it suffices to show the bound $|h(u)-\phitan(u_L)| \leq C|u-u_L|$. We can compute $h(u)$ explicitly using the quadratic formula and find:
\[
h(u)=\phitan(u_L)-\frac{2|\sigma(u,\phitan(u_L))-\sigma(u_L,\phitan(u_L))|}{|c|},
\]
where $\sigma(u,v)=\frac{f(u)-f(v)}{u-v}$. So, we obtain:
\[
|h(u)-\phitan(u_L)| \leq \frac{2}{|c|}|\sigma(u,\phitan(u_L))-\sigma(u_L,\phitan(u_L))| \leq C|u-u_L|,
\]
where $C$ may again be taken to depend continuously on $d(u_L,0)$ by property \ref{enum:7} (see Section \ref{system properties}).  
\end{proof}
\subsubsection{Entropy dissipation for Kru\v{z}kov entropies}
In this section, we give useful necessary and sufficient conditions for a shock $(u_-, u_+, \sigma_\pm)$ to be entropic for a Kru\v{z}kov entropy $\eta_k(u)=|u-k|$. Lemma \ref{entropydissipationlemma} holds only for $C^1$ entropies, which the Kru\v{z}kov entropies fail to verify. Recall that the associated entropy-flux is $q_k(u)=sgn(u-k)(f(u)-f(k))$.
\begin{lemma}\label{Kru\v{z}kov1}
Let $f$ verify \eqref{concaveconvex}. Let $(u_-, u_+)$ be a shock with $u_- > 0,u_+$. Let $k \leq \phitan(u_-)$. Then, $(u_-, u_+)$ is $\eta_k$-entropic if and only if $u_+ \geq k$. 
\end{lemma}
\begin{proof}
Let $u_+ \geq k$. Then, one may compute:
\[
E_{\eta_k}(u_-, u_+)=0,
\]
so $(u_-, u_+)$ is $\eta_k-$entropic. 
\par On the other hand, let $k > u_+$. It suffices to show that $E_{\eta_k}(u_-, u_+) >0$. We compute:
\[
E_{\eta_k}(u_-, u_+)=\sigma_\pm(u_+-k)+f(k)-f(u_+)+\sigma_\pm(u_--k)+f(k)-f(u_-).
\]
Now, $\sigma_\pm(u_+-k)+f(k)-f(u_+) > 0 \iff \sigma_\pm < \sigma(k, u_+)$ However, this is clear geometrically as both $k, u_+ \leq \phitan(u_-)$. For the other term, we see $\sigma_\pm(u_--k)+f(k)-f(u_-) > 0 \iff \sigma_\pm > \sigma(u_-,k)$. This is also clear as both $k, u_+ \leq \phitan(u_-)$. So, both terms are positive and we see $E_{\eta_k}(u_-, u_+) > 0$. 
\end{proof}
The situation is more delicate when $k > \phitan(u_-)$. The following proof is a modification of the proof of Theorem 3.1 in \cite{MR1927887}.
\begin{lemma}\label{Kru\v{z}kov2}
Let $f$ be concave-convex. Let $(u_-, u_+)$ be a shock with $u_- > 0,u_+$. Let $k \in (\phitan(u_-),u_-]$. Then, $(u_-, u_+)$ is $\eta_k$-entropic if and only if $u_+ \geq \phi_{k}^\sharp(u_-)$.
\end{lemma}
\begin{proof}
Firstly, by the Ole\u{\i}nik condition, $(u_-, u_+)$ is entropic for any entropy for $u_+ \geq k$, so it suffices to check what happens for $u_+ < k$. In this region, $\eta_k$ is smooth, so we differentiate the entropy dissipation with respect to $u_+$ and find:
\[
\partial_{u_+}E_{\eta_k}(u_-,u_+)=b(u_-u_+)\partial_{u_+}\sigma_\pm,
\]
where:
\begin{align*}
b(u_-,u_+)&:=\eta_k(u_-)-\eta_k(u_+)-\eta'(u_+)(u_--u_+)=2(u_--k) > 0, \\
\partial_{u_+}\sigma_\pm&=\frac{f'(u_+)-\sigma_\pm}{u_+-u_-}.
\end{align*}
As the flux is concave-convex, it is clear that:
\begin{align*}
\partial_{u_+}\sigma_\pm &< 0 \text { for } u_+ < \phitan(u_-), \\
\partial_{u_+}\sigma_\pm &> 0 \text { for } u_+ > \phitan(u_-). \\
\end{align*}
Thus, we see that $u_+ \mapsto E(u_-, u_+)$ attains exactly one zero for $u_+ < k$ and will be positive for all $u_+$ less than this value. It remains to determine what this decisive value of $u_+$ is. In the region $u_+ < k$, we compute:
\begin{align*}
E_{\eta_k}(u_-,u_+)&=-\sigma_\pm(2k-u_--u_+)+2f(k)-f(u_+)-f(u_-) \\
&=(k-u_+)(\sigma(k,u_+)-\sigma_\pm)+(k-u_-)(\sigma(k,u_-)-\sigma_\pm).
\end{align*}
We see that $E_{\eta_k}(u_-,u_+)=0$ for $u_+$ such that $\sigma_\pm=\sigma(u_+,k)=\sigma(u_-,k)$, or exactly $u_+= \phi_{k}^\sharp(u_-)$.
\end{proof}

\subsubsection{Properties of the system}\label{system properties}
In this section, we list some abstract properties satisfied by the system. \newline
\textbf{Properties 1.} Recall that we assume that the set of states $\statespace$ is bounded. Further, we have the following (all these properties are easily verified for $C^4$ concave-convex fluxes):
\begin{enumerate}[label=(\alph*)]
\item \label{enum:1} For any $0 \neq u \in \statespace, f''(u) \neq 0$.
\item \label{enum:2} For $u_L \in \statespace$, we denote $s \mapsto S_{u_L}(s)$ the shock curve through $u_L$, and $s \mapsto \sigma_{u_L}(s)$ the associated shock speed, defined for $s \geq 0$. We choose the parameterization such that $s=|u_L-S_{u_L}(s)|$, i.e. $(u_L, S_{u_L}(s))$ is the shock with left state $u_L$ and strength $s$. For every $u_L \in \statespace, 0 \leq s < \infty$, we have:
\[
f(S_{u_L}(s))-f(u_L)=\sigma_{u_L}(s)(S_{u_L}(s)-u_L).
\]
Further, $\sigma_{u_L}(0)=\lambda(u_L)$. We assume that these curves are defined globally in $\statespace$ for every $u_L \in \statespace$. 
\item \label{enum:3} There exist functions $u_L \mapsto s^{u_L}_{\phitan}, u_L \mapsto s^{u_L}_{\phidiss}$ defined as follows:
\begin{enumerate}
\item For $0 \leq s \leq s^{u_L}_{\phitan}$, $(u_L, S_{u_L}(s))$ is an Ole\u{\i}nik shock. Further, $\sigma_{u_L}'(s) < 0$ for $0 \leq s < s^{u_L}_{\phitan}$, $\sigma_{u_L}'(s)=0$ for $s=s^{u_L}_{\phitan}$, and $\sigma_{u_L}'(s)=0$ for $s=s^{u_L}_{\phitan}$, and $\sigma_{u_L}'(s) > 0$ for $s > s^{u_L}_{\phitan}$.
\item For $0 \leq s \leq s^{u_L}_{\phidiss}$, $(u_L, S_{u_L}(s))$ is a $\eta-$entropic Rankine-Hugoniot discontinuity. Further, if $s > s^{u_L}_{\phidiss}$, then $(u_L, S_{u_L}(s))$ fails to be $\eta$-entropic. 
\end{enumerate}
\item \label{enum:4} If $(u,v)$ is a $\eta$-entropic Rankine-Hugoniot discontinuity with velocity $\sigma$ such that $u >v > 0$, then $\sigma > \lambda(v)$. More generally, for any $\eta$-entropic Rankine-Hugoniot discontinuity, we have $\sigma \in I(\lambda(u_L),\lambda(u_R))$, where $I(a,b)$ is the closed interval with endpoints $a$ and $b$. 
\item \label{enum:5} If $(u,v)$ is a $\eta$-entropic Rankine-Hugoniot discontinuity, then $v=S_{u}(s)$ for some $0 \leq s \leq s^u_{\phidiss}$. 
\item \label{enum:6} For $u_L \in \statespace, s \geq 0$, $\dv{}{s}\eta(u_L|S_{u_L}(s)) > 0$ (the shock strengthens with $s$). 
\item \label{enum:7} We have the following regularity:
\begin{enumerate}
    \item $u_L \mapsto r(u_L)$ is smooth on $\statespace \setminus \{0\}$. 
    \item $u_L \mapsto s^{u_L}_{\phitan}, s^{u_L}_{\phidiss}$ are continuous on $\statespace$. 
    \item $(s,u_L) \mapsto S_{u_L}(s), (s,u_L) \mapsto \sigma_{u_L}(s)$ are $C^1$ on $\{(s,u):u \in \statespace \setminus \{0\}, s \in [0,\infty)\}$.
\end{enumerate}
\end{enumerate}
\begin{notation}
For $u$ fixed, we will denote $s^u_k$ to be such that $S_u(s^u_k)=k$. If the subscript is a function $g$ instead of a point $k$, it is to be understood as $s^u_g=s^u_{g(u)}$. This is consistent with the notation introduced in property \ref{enum:3}.
\end{notation}
\begin{remark}
As we are in the scalar setting, the shock curves are simple and explicit. Let $u_L > 0$. Then, $S_{u_L}(s)=u_L-s$, and $\sigma_{u_L}(s)=\frac{f(u_L-s)-f(u_L)}{-s}$. Computing the derivative, we see:
\[
\sigma'_{u_L}(s) \approx -\frac{1}{2}f''(u_L-s)^2,
\]
for $s$ small. In other words, the negativity of $\sigma'_{u_L}$ degenerates for $u_L$ close to zero. As a result of this, any bounds in terms of $\sigma'_{u_L}$ (crucially, in the proofs of Propositions \ref{Dcontlarge}, \ref{DRHlarge1}, and \ref{DRHlarge2}) will depend on $d(u_L,0)$.
\end{remark}

\subsubsection{Study of the set $\Pi_a$}
An important part of the analysis in both the large and small shock regimes is to study the structure of the set $\Pi_a$ (defined in \eqref{pia}). Note that it is a closed interval, as for $a < 1$, the function $\frac{1}{a}\eta(u|u_L)-\eta(u|u_R)$ is strictly convex, and the function is clearly positive near both $\pm \infty$ and negative at $u_L$. Here, we give a lemma that illuminates the structure in the large shock regime. It is comparable to Lemma 3.2 in \cite{MR4176349}, although the proof is based on Lemma 3.4 in \cite{MR4667839}. It gives precise asymptotics on the diameter of the set $\Pi_a$ as $a \to 0$.
\begin{lemma}\label{pistructurelarge}
For $a$ sufficiently small, $\Pi_a$ is compactly contained within $\statespace$ and $diam(\Pi_a) \leq C\sqrt{a}$, where the implicit constant depends only on the system and continuously on $s_0$.
\end{lemma}
\begin{proof}
Let $u \in \Pi_a$. Then,
\begin{align*}
\eta(u|u_L)&\leq a\left(\eta(u|u_L)+\int_{u_L}^{u_R}\partial_v\eta(u|v)\diff v\right) \\
&\leq a(\eta(u|u_L)+s_0(|u-u_L|+|u-u_R|)).
\end{align*}
where $\partial_v\eta(u|v)=-\eta''(u)(u-v)$. By Lemma \ref{quadraticentropy}, this implies:
\begin{align*}
\frac{(1-a)}{a}|u-u_L|^2-Cs_0|u-u_L|-Cs_0^2 \leq 0,
\end{align*}
where $C$ depends only on the system. Using Young's inequality, we find that $|u-u_L| \leq C\sqrt{a}$ for $a$ sufficiently small, where $C$ now also depends on $s_0$ continuously. 
\end{proof}

\subsection{Proof of Theorem \ref{dissipationfnllargeshock}}\label{shiftconstruction}
In this section, we prove that Propositions \ref{Dcontlarge}, \ref{DRHlarge1}, and \ref{DRHlarge2} imply Theorem \ref{dissipationfnllargeshock}. We follow the construction of the shift function developed by Krupa \cite{MR4176349} and later used by Golding, Krupa, \& Vasseur \cite{MR4667839}, letting the shift be the generalized characteristic of the weak solution minus a correction. Because we will use the exact same shift in Section \ref{shiftconstructionsmall}, we introduce the notation used there. Define:
\begin{align*}
\tilde{\eta}(u)&=\frac{1}{a}\eta(u|u_L)-\eta(u|u_R), \\
\tilde{q}(u)&=\frac{1}{a}q(u;u_L)-q(u;u_R). \\
\end{align*}
These are identical definitions to \eqref{etatildedef}, \eqref{qtildedef}. Note that $\tilde{\eta}, \tilde{q}$ define an entropy, entropy-flux pair for the system (cf. Lemma \ref{pistructuresmall}). Then, we have:
\[
\Pi_a=\{u \in \statespace|\tilde{\eta}(u) \leq 0\}.
\]
\underline{\textbf{Step 1: Fixing the constant $a^*$.}}
Fix $a^*$ as follows. Firstly, pick $a^*$ satisfying the hypotheses of Propositions \ref{Dcontlarge} \ref{DRHlarge1}, and \ref{DRHlarge2}. Next, note that $\phidiss(\phidiss(u_L))=u_L$, and $\phidiss$ is monotone decreasing, so by taking $a^*$ smaller if needed, we may have 
\begin{equation}\label{midshockruleout}
\phidiss(v) < u \ \forall \ u \in \Pi_a \indent \text{and} \indent \Pi_a \subset B(u_L,\eps),
\end{equation} 
for all $v \geq \phidiss(u_L)+\eps$, where $\eps$ is the constant used to define $\weakone$. Note that as $\phidiss$ is continuous, $a^*$ depends continuously on $\eps$. Now, once and for all, fix $a < a^*$, which we will use to define the shift that gives the contraction. 
\par We need to control various terms in the dissipation for states outside of $\Pi_a$. This is contained in the following lemma:
\begin{lemma}\label{qcontrol}
For $u_L, u_R$ fixed, and any $a < a^*$, there exists constants $C_1=C_1(u_L, s_0, a)$ and $C_2=C_2(u_L,s_0,a)$ such that:
\begin{align*}
    \tilde{\eta}(u) &\geq C_1d(u, \partial \Pi_a) \text{ for any } u \in \statespace \setminus \Pi_a, \\
    |\tilde{q}(u)| &\leq C_2|\tilde{\eta}(u)| \text{ whenever } u \in \statespace \setminus \Pi_a \text{ verifies } \tilde{q}(u) \leq 0.
\end{align*}
\end{lemma}
\begin{proof}
Firstly we show that there exists a constant $C_1$ (depending on $a$) such that:
\begin{align}\label{eta'bd}
|\tilde{\eta}'(\overline{u})| \geq C_1, \text{ for } \overline{u} \in \partial \Pi_a.
\end{align}
Indeed, let $u_1 \in \partial \Pi_a$, and let $u_2$ be the other point on the boundary. Then, by Lemma \ref{quadraticentropy} :
\[
c^*|u_1-u_2|^2 \leq \tilde{\eta}(u_1|u_2)=-\tilde{\eta}'(u_2)(u_1-u_2) \leq |\tilde{\eta}'(u_2)||u_1-u_2|,
\]
where the equality follows because $\tilde{\eta}(u_1)=\tilde{\eta}(u_2)=0$. As this argument is symmetric in $u_1, u_2$, taking $C_1=c^*diam(\Pi_a)$ gives \eqref{eta'bd}. Note that here, as $\tilde{\eta}$ depends continuously on $u_L, s_0$, and $a$, the constant $c^*$ will depend continuously on those quantities as well (for a precise statement of the asymptotics of $c^*$, see Appendix A in \cite{MR2807139}). Now, let $u \in \statespace \setminus \Pi_a$. Without loss of generality, assume $u > u_2 > u_1$. Then, define $\phi(t)=\tilde{\eta}(u_2+t(u-u_2))$, for $t \in [0,1]$. Then:
\[
\phi'(0)=(u-u_2)\tilde{\eta}'(u_2) \geq C_1d(u,\partial \Pi_a),
\]
by \eqref{eta'bd}. Further, $\phi$ is convex, so for all $0 \leq t \leq 1$:
\[
\phi(t) \geq \phi(0)+t\phi'(0) \geq \tilde{\eta}(u_2)+tC_1d(u,\partial \Pi_a).
\]
At $t=1$, we see:
\[
\tilde{\eta}(u)=\phi(1) \geq C_1d(u,\partial \Pi_a).
\]
Finally, let $u \in \statespace \setminus \Pi_a$, with $\tilde{q}(u) \leq 0$. Again, without loss of generality, let $u > u_2 > u_1$. By Proposition \ref{Dcontlarge}, %(Proposition \ref{Dcontsmall} can be used here in the small shock regime)% 
$\tilde{q}(u_2) \geq 0$. So, we have:
\begin{align*}
|\tilde{q}(u)|=-\tilde{q}(u) &\leq \tilde{q}(u_2)-\tilde{q}(u)=-\int_0^1\tilde{q}'((1-t)u_2+tu)(u-u_2)dt \\
&=-\int_0^1\tilde{\eta}'((1-t)u_2+tu)f'((1-t)u_2+tu)(u-u_2)dt \\
&\leq C_3d(u, \partial \Pi_a),
\end{align*}
where $C_3=\sup_{u \in \statespace} |\tilde{\eta}'(u)f'(u)|$ depends continuously on $u_L, s_0$, and $a$ as well. Taking $C_2=\frac{C_3}{C_1}$ gives the result. 
\end{proof}
\underline{\textbf{Step 2: Construction of the shift.}} Consider the set $\Pi_a^c$, which is an open set. Therefore, the indicator function of this set is lower semi-continuous on $\statespace$, and $-\ind_{\Pi_a^c}$ is upper semi-continuous. Now, we define the velocity:
\begin{equation}\label{velocitylarge}
V(u):=\lambda(u)-(C_2+2L)\ind_{\Pi_a^c}(u),
\end{equation}
where $L=\sup_{u \in \statespace}\lambda(u)$, and $C_2$ is the constant from Lemma \ref{qcontrol}. This function is still upper semi-continuous in $u$ on $\statespace$. We solve the following ODE with discontinuous right-hand side,
\begin{equation}\label{shift}
\begin{cases}
\dot{h}(t)=V(u(h(t),t)), \\
h(\overline{t})=x_0.
\end{cases}
\end{equation}
The existence of such a shift satisfying \eqref{shift} in the sense of Filipov is given by the following lemma. A proof is given in \cite{MR2807139} (partially in the appendix). In the following, recall that $\eta$ is the fixed $C^3$ strictly convex entropy, while $\eta_k$ will be some Kru\v{z}kov entropy (we do not specify in the lemma as the Kru\v{z}kov entropy used will be different in the large and small regimes). 
\begin{lemma}\label{filipov}
Let $V:\statespace \to \R$ be bounded and upper semi-continuous on $\statespace$ and continuous on $U$ an open, full-measure subset of $\statespace$. Let $u$ be a $\eta, \eta_k$-entropic weak solution to \eqref{cl} satisfying Definition \ref{strongtrace}, let $x_0 \in \R$, and let $\overline{t}\in[0,\infty)$. Then, there exists a Lipschitz function $h:[\overline{t}, \infty) \to \R$ such that:
\begin{align*}
V_{min}(t) &\leq \dot{h}(t) \leq V_{max}(t), \\
h(\overline{t})&=x_0, \\
Lip[h] &\leq ||V||_{\infty},
\end{align*}
for almost every $t$, where $u_\pm=u(h(t)\pm,t)$ and:
\begin{align*}
V_{max}(t)=\max(V(u_+),V(u_-)) \indent \text{and} \indent V_{min}(t)=\begin{cases}
\min(V(u_+),V(u_-)) & u_+,u_- \in U \\
-||V||_\infty & \text{otherwise}
\end{cases}.
\end{align*}
Furthermore, for almost every $t$:
\begin{align*}
    f(u_+)-f(u_-)&=\dot{h}(u_+-u_-), \\
    q(u_+)-q(u_-)&\leq \dot{h}(\eta(u_+)-\eta(u_-)), \\
    q_{k}(u_+)-q_{k}(u_-)&\leq \dot{h}(\eta_{k}(u_+)-\eta_{k}(u_-)),
\end{align*}
i.e. for almost every $t$, either $(u_-, u_+, \dot{h})$ is a $\eta,\eta_{k}$-entropic shock or $u_+=u_-$. 
\end{lemma}
\underline{\textbf{Step 3: Proof of \eqref{sufficientlarge}.}} Denote $u_{\pm}:=u(h(t)\pm,t)$. We will prove \eqref{sufficientlarge} for almost every fixed time $t$. We split into the following four cases (which allow for $u_+=u_-$):
\begin{align*}
\text{Case 1. } & u_-, u_+ \in \Pi_a^c, \\ 
\text{Case 2. } & u_- \in \Pi_a^c, u_+ \in \Pi_a, \\
\text{Case 3. } & u_- \in \Pi_a, u_+ \in \Pi_a^c, \\
\text{Case 4. } & u_-,u_+ \in \Pi_a.
\end{align*}
As we only need to prove \eqref{sufficientlarge} for almost every $t$, by Lemma \ref{filipov}, we may consider only times for which either $(u_-, u_+, \dot{h}(t))$ is a $\eta,\eta_{\phitan(u_L)}$-entropic shock or $u_+=u_-$. 
\newline \underline{Case 1:} In this case, by \eqref{velocitylarge}, we see $\dot{h}(t) < \inf_{u \in \statespace}\lambda(u)$. If $u_- \neq u_+$, then this is a contradiction to property \ref{enum:4}. So necessarily $u_-=u_+$. Denote $u=u_-=u_+$. Then, as $u \not \in \Pi_a$, $\tilde{\eta}(u) > 0$. So, using $a=\frac{a_2}{a_1}$ and $\dot{h}(t) \leq C_2$:
\begin{align*}
\dot{h(t)}(a_1\eta(u_-|u_L)-a_2\eta(u_+|u_R))-a_1q(u_-;u_L)+a_2q(u_+;u_R)&=a_2(-\tilde{q}(u)+\dot{h}(t)\tilde{\eta}(v)) \\
&\leq a_2(-\tilde{q}(u)-C_2\tilde{\eta}(u)).
\end{align*}
If $\tilde{q}(u) \geq 0$, then the right-hand side is non-positive. If $\tilde{q}(u) \leq 0$, Lemma \ref{qcontrol} grants that the right-hand side is non-positive. 
\newline \underline{Case 2:} In this case, we must have $u_- \neq u_+$. Further:
\[
\dot{h}(t) \in [\lambda(u_-)-(C_2+2L), \lambda(u_+)].
\]
If $u_- > u_+$, then this is a contradiction to property \ref{enum:4}. So, we must have $u_- < u_+$. In this case, if $u_->\phidiss(u_L)+\eps$, then $(u_-, u_+, \dot{h}(t))$ cannot be $\eta$-entropic by \eqref{midshockruleout} and property \ref{enum:3}. So, by the definition of $\weakone$, $u_- < \phidiss(u_L)-\eps$. This implies:
\[
\dot{h(t)}(a_1\eta(u_-|u_L)-a_2\eta(u_+|u_R))-a_1q(u_-;u_L)+a_2q(u_+;u_R)=a_2D_{RH}(u_-,u_+,\sigma_\pm) \leq 0
\]
by Proposition \ref{DRHlarge2}.
\newline \underline{Case 3:} In this case necessarily $u_- \neq u_+$. Proposition \ref{DRHlarge1} gives \eqref{sufficientlarge} immediately. 
\newline \underline{Case 4:} If $u_- \neq u_+$, then Proposition \ref{DRHlarge1} gives \eqref{sufficientlarge}. If $u_-=u_+$, we use Proposition \ref{Dcontlarge} instead. This completes the proof of Theorem \ref{dissipationfnllargeshock}.
\subsubsection{Proof of Corollary \ref{thmlargeshock}}
The idea behind the proof of the $a$-contraction property is to integrate \eqref{relativeentropic} with $b=u_L$ on $x < h(t)$, and $b=u_R$ on $x > h(t)$. The strong trace property gives:
\begin{equation}\label{basiccomputation}
\dv{}{t}E_t(u) \leq \dot{h}(t)(a_1\eta(u_-(t)|u_L)-a_2\eta(u_+(t)|u_R))-a_1q(u_-(t);u_L)+a_2q(u_+(t);u_R).
\end{equation}
It suffices to show that this is non-positive for almost every $t \geq 0$, but this is precisely Theorem \ref{dissipationfnllargeshock}, giving the result. We take $\overline{t}=0$ and $x_0=0$, and use Lemma \ref{quadraticentropy} to obtain the $L^2$-stability. 

\subsubsection{Proof of Corollary \ref{uniformalarge}}
Fix a closed interval $[\underline{b}, \overline{b}]$ with $0 < \underline{b} < \overline{b}$. We begin with a lemma that shows for $\eps$ sufficiently small, all the classes $\weakone$ with $u_L \in [\underline{b}, \overline{b}]$ are contained in $\weaktwo$. 
\begin{lemma}
There exists $\eps_0$ sufficiently small such that:
\[
\weaktwo \subset \bigcap_{u_L \in [\underline{b}, \overline{b}]}\mathcal{S}_{\text{weak},u_L,\eps_0}.
\]
\end{lemma}
\begin{proof}
Let $u \in \weaktwo$, $u_L \in [\underline{b}, \overline{b}]$. It suffices to show that $u$ is $\eta_{\phitan(u_L)}$-entropic and $u(t,x) \not \in B(\phidiss(u_L), \eps_0)$ for $\eps_0$ sufficiently small. Firstly, as $u$ is $\eta_{\phitan(\underline{b})}-$entropic, and $u_0 \geq \phitan(\underline{b})$, $u$ obeys the maximum principle, i.e. $u(t,x) \geq \phitan(\underline{b})$ for all $(t,x) \in \R^+ \times \R$ (for a proof of this, see the proof of Theorem 6.2.3 in \cite{MR3468916}). As $\phitan$ is monotone decreasing, $\phitan(\underline{b}) \geq \phitan(u_L)$ for all $u_L \in [\underline{b}, \overline{b}]$. So, $u$ is clearly $\eta_{\phitan(u_L)}$-entropic, as it is a weak solution to \eqref{cl} and the Kru\v{z}kov entropies/entropy fluxes $\eta_{\phitan(u_L)}$ reduce to:
\begin{align*}
\eta_{\phitan(u_L)}(u(t,x))&=|u(t,x)-\phitan(u_L)|=u(t,x)-\phitan(u_L), \\
q_{\phitan(u_L)}(u(t,x))&=sgn(u(t,x)-\phitan(u_L))(f(u(t,x))-f(\phitan(u_L))) \\
&=(f(u(t,x))-f(\phitan(u_L))).
\end{align*}
Next, for any fixed $u_L$, $\phidiss(u_L) < \phitan(u_L)$, so by taking $\eps_0 < \min_{u_L \in [\underline{b},\overline{b}]}\phitan(u_L)-\phidiss(u_L)$, we obtain:
\[
u(t,x) \geq \phitan(\underline{b}) \geq \phitan(u_L) > \phidiss(u_L)+\eps_0.
\]
In other words, $u(t,x) \not \in B(\phidiss(u_L), \eps_0)$.
\end{proof}
Finally, note that in the proof of Theorem \ref{dissipationfnllargeshock}, we have shown that $a^*$ depends continuously on $\eps, s_0, d(u_L,0)$, and $d(u_R,\phitan(u_L))$ (in the case that the last quantity is positive, which it will always be if $u_R \in [\underline{b}, \overline{b}]$ as well). So, using the fact that $\phitan$ is monotone decreasing, taking the minimum of $a^*$ over the compact set:
\[
(\eps_0, s_0, d(u_L,0), d(u_R,\phitan(u_L))) \in \{\eps_0\} \times [\tilde{s_0}, \overline{b}-\underline{b}] \times [\underline{b}, \overline{b}], \times [d(\underline{b},\phitan(\underline{b})), d(\overline{b}, \phitan(\overline{b}))],
\]
we see that $a^*$ may be taken uniformly over all such shocks. Finally, by Lemma \ref{filipov}, for any shock $(u_L, u_R)$, we have the following bound on the shift:
\[
-C_2-3|L| \leq |\dot{h}(t)| \leq C_2+3|L|,
\]
where $C_2$ is the constant from Lemma \ref{qcontrol}. Noting that $C_2$ depends continuously on $u_L, s_0$, and $a$, taking $C^*$ to be the maximum of $C_2$ over the compact set:
\[
(a^*, s_0, u_L) \in \{a^*\} \times [\tilde{s_0} \times \overline{b}-\underline{b}] \times [\underline{b}, \overline{b}],
\]
and setting $\maximalspeedlarge=C^*+3|L|$, we obtain the result.

\subsection{Proof of Propositions \ref{Dcontlarge} and \ref{DRHlarge1}}
In this section, we prove Propositions \ref{Dcontlarge} and \ref{DRHlarge1}. The proof of Proposition \ref{Dcontlarge} is a special case of a specific part of the proof of Proposition \ref{DRHlarge1}. The first lemma needed gives an explicit formula for the entropy lost at a shock $(u_-,u_+)$ with $u_+=S_{u_-}(s)$. The formula originated in work of Lax \cite{MR93653} and has been used frequently in past works on the $a$-contraction theory \cite{MR3537479}, \cite{MR3519973}. For a proof, see \cite{MR3537479}.
\begin{lemma}\label{entropyformula}
Let $(u_-,u_+)$ be a $\eta$-entropic Rankine-Hugoniot discontinuity with velocity $\sigma_\pm$. Then, for any $v \in \statespace$:
\[
q(u_+;v)-\sigma_\pm \eta(u_+|v) \leq q(u_-;v)-\sigma_\pm \eta(u_-|v).
\]
Further, if there exists $s \geq 0$ such that $u_+=S_{u_-}(s)$, and $\sigma_\pm=\sigma_{u_-}(s)$, then:
\[
q(u_+;v)-\sigma_\pm \eta(u_+|v)=q(u_-;v)-\sigma_\pm \eta(u_-|v)+\int_0^s\sigma'_{u_-}(\tau)\eta(u_-|S_{u_-}(\tau)) \diff \tau.
\]
\end{lemma}
The next lemma is a corollary of this which is also classical for the $a$-contraction theory (it first appeared in \cite{MR2807139}. See also \cite{MR3537479}, \cite{MR3519973}). We follow the proof in the appendix of \cite{MR3519973}, although we give it here to emphasize that $a$ and $\kappa$ depend continuously on $s_0$.
\begin{lemma}\label{entropyformula2}
For any $u \in \statespace$ and any $s \geq 0, s' > 0$, we have:
\begin{equation}\label{entropyformulaequation}
q(S_u(s);S_u(s'))-\sigma_u(s)\eta(S_u(s)|S_s(s'))=\int_{s'}^s\sigma_u'(\tau)(\eta(u|S_u(\tau))-\eta(u|S_u(s')))\diff \tau.
\end{equation}
In particular, for any $u \in \Pi_a$, for $a$ sufficiently small, if $s'=s^u_{u_R}$ and $s \leq \frac{s_0}{2}$, then:
\[
q(S_u(s);u_R)-\sigma_u(s)\eta(S_u(s)|u_R) \leq -\kappa |\sigma_{u}(s)-\sigma_u(s^u_{u_R})|,
\]
where $\kappa$ depends on the system and continuously on $s_0$, $d(u_L,0)$, and $a$.
\end{lemma}
\begin{proof}
Use the estimate of Lemma \ref{entropyformula} twice with $v=S_u(s')$ and $u_-=u$. The first time take $u_+=S_u(s)$, and the second time $u_+=S_u(s')$. The difference of the two resulting identities gives the first result. The second follows from $\eqref{entropyformulaequation}$. Indeed, let $a$ be small enough such that $S_u(\frac{s_0}{2}) > u_R$ for all $u \in \Pi_a$ (this is possible by Lemma \ref{pistructurelarge}), and let $s'=s^u_{u_R}$ and $s \leq \frac{s_0}{2}$. Then:
\begin{align*}
q(S_u(s);u_R)-\sigma_u(s)\eta(S_u(s)|u_R) &= \int_{{\frac{s_0}{2}}}^{s}\sigma_u'(\tau)(\eta(u|S_u(\tau))-\eta(u|S_u(s')))\diff \tau \\
&+ \int_{s^u_{u_R}}^{\frac{s_0}{2}}\sigma_u'(\tau)(\eta(u|S_u(\tau))-\eta(u|S_u(s')))\diff \tau\\
&=:I_1+I_2.
\end{align*}
Now, choose a positive constant $\kappa$ satisfying:
\[
\kappa \leq \min_{u\in \Pi_a}(\eta(u|u_R)-\eta(u|S_u(\frac{s_0}{2})).
\]
Then, using properties \ref{enum:3} and \ref{enum:6}, we see:
\begin{align*}
I_1 &\leq -\kappa|\sigma_u(\frac{s_0}{2})-\sigma_u(s)| \\
&\leq -\kappa|\sigma_u(s^u_{u_R})-\sigma_u(s)|+\kappa |\sigma_u(s^u_{u_R})-\sigma_u(\frac{s_0}{2})|.
\end{align*}
Now, $I_2 < 0$ for all $u \in \Pi_a$, so we choose $\kappa$ sufficiently small so that additionally:
\[
\kappa |\sigma_u(s^u_{u_R})-\sigma_u(\frac{s_0}{2})| \leq \min_{u\in \Pi_a}-I_2,
\]
which yields:
\[
q(S_u(s);u_R)-\sigma_u(s)\eta(S_u(s)|u_R) \leq -\kappa |\sigma_{u}(s)-\sigma_u(s^u_{u_R})|.
\]
\end{proof}
Now, we show a lemma that specifies why we desire for $\Pi_a$ to be localized around $u_L$. The proof follows that of Lemma 4.3 in \cite{MR3519973}, although we take particular care to note that the implicit constants depend continuously on the parameters $s_0, d(u_L,0)$.
\begin{lemma}\label{localestimate}
There exist $C_0, \beta > 0$ and $v \in (\sigma,\lambda(u_L))$, such that for any $u \in B(u_L,C_0)$:
\begin{align*}
    v &< \lambda(u), \\
    -q(u;u_L)+v\eta(u|u_L) &\leq -\beta \eta(u|u_L), \\
    q(u;u_R)-v\eta(u|u_R) &\leq -\beta \eta(u|u_R).
\end{align*}
Here, $C_0, \beta$, and $v$ depend only on the system, and continuously on $s_0$ and  $d(u_L,0)$.
\end{lemma}
\begin{proof}
Use Lemma \ref{entropyformula2} with $S_{u_L}(s_0)=u_R, s'=s_0$, and $s=0$, obtaining:
\[
q(u_L;u_R)-\lambda(u_L)\eta(u_L|u_R) < 0,
\]
which we may write as:
\[
q(u_L;u_R)-\lambda(u_L)\eta(u_L|u_R)+3\beta_1\eta(u_L|u_R) < 0,
\]
for $\beta_1$ sufficiently small and depending continuously on $s_0$ and $d(u_L,0)$ (more precisely, $\beta_1 < \bigl(-\frac{\frac{1}{6}\int_{s_0}^0\sigma_{u_L}'(\tau)(\eta(u_L|S_{u_L}(\tau))-\eta(u_L|u_R))\diff \tau}{\eta(u_L|u_R)}\bigr)$). Letting $v=\lambda(u_L)-\beta_1$, we see:
\[
q(u_L;u_R)-v \eta(u_L|u_R)+2\beta_1\eta(u_L|u_R) < 0.
\]
Now, using the continuity of $q(\cdot;u_R)$, $\eta(\cdot|u_R)$, and $\lambda(\cdot)$ on $\statespace$, there exists $C_{0,1}$ small enough such that:
\begin{align*}
v &< \lambda(u), \\
q(u;u_R)-v \eta(u|u_R) &< -\beta_1\eta(u|u_R), 
\end{align*}
for $u \in B(u_L, C_{0,1})$. Note that $C_{0,1}$ may be taken to depend continuously on $s_0$ and $d(u_L,0)$ by the continuous dependence of $\beta_1$ on these same quantities and the continuity of the partial derivatives of $q(\cdot, \cdot), \lambda(\cdot)$, and $\eta(\cdot, \cdot)$ on $\statespace$.
\par Now, doing an expansion at $u=u_L$, we find:
\[
-q(u;u_L)+v \eta(u|u_L)=(v-\lambda(u_L))\eta''(u_L)(u-u_L)^2+\mathcal{O}(|u-u_L|^3).
\]
Hence:
\begin{align*}
-q(u;u_L)+v \eta(u|u_L)&=-\beta_1 \eta''(u_L)(u-u_L)^2+\mathcal{O}(|u-u_L|^3) \\
&\leq -\frac{C\beta_1}{2}\eta(u|u_L), \indent \text{for $u \in B(u_L, C_{0,2})$,}
\end{align*}
for $C$ depending on the system and $C_{0,2}$ sufficiently small depending on the system and continuously on $s_0$ and $d(u_L,0)$ (due to the continuity of $\eta'''$ on $\statespace$), where we have used Lemma \ref{quadraticentropy}. Taking $\beta=\min(\beta_1, \frac{C\beta_1}{2})$ and $C_0=\min(C_{0,1}, C_{0,2})$ gives the result. 
\end{proof}
The last lemma will be crucially used in the proof of Proposition \ref{DRHlarge1} to show that, in the case $\sigma_{u_-}(s) > v$, in fact $(u_-, u_+)$ is very small.
\begin{lemma}\label{smallshockDRHlarge1}
There exists $a^*$ sufficiently small such that the following holds. Let $a < a^*$, $(u_-,u_+)$ be a $\eta_{\phitan(u_L)}$-entropic shock with $u_- \in \Pi_a$, and $u_+=S_{u_-}(s)$ for some $s \geq 0$. Let $\sigma_{u_-}(s) > v$, where $v$ is defined in Lemma \ref{localestimate}.  Then, $s \leq \frac{s_0}{2}$. Here, $a^*$ depends only on the system and continuously on $s_0$ and $d(u_L,0)$. 
\end{lemma}
\begin{proof}
Firstly, we take $a$ sufficiently small so that $S_{u_-}(\frac{s_0}{2}) > u_R$ for all $u_- \in \Pi_a$. Then, assume that $s > \frac{s_0}{2}$. Then, by property \ref{enum:3}, if $\frac{s_0}{2} < s \leq s^{u_-}_{\phitan}$, then there exists $v_0$ (depending continuously on $d(u_L,0)$ and $a$) such that $\sigma_{u_-}(s) < v_0< v$, uniformly in $u_-$, a contradiction (where we have taken $v$ closer to $\lambda(u_L)$ and $a$ smaller if needed). So, we must have $s \geq s^{u_-}_{\phitan}$. Now, we split into three cases.
\newline \underline{Case 1:} Assume that $u_-=u_L$. Then, $(u_-,u_+)$ is $\eta_{\phitan(u_L)}-$entropic, so by Lemma \ref{Kru\v{z}kov1} we must have $s \leq s^{u_-}_{\phitan}$, in which case $\sigma_{u_-}(s) < v_0$, a contradiction. 
\newline \underline{Case 2:} Assume that $u_- < u_L$. Then, $\phitan(u_L)<\phitan(u_-)$, so by Lemma \ref{Kru\v{z}kov1} we must have $s \leq s^{u_-}_{\phitan(u_L)}$. So, by property \ref{enum:3}, we have:
\begin{align*}
\sigma_{u_-}(s)-v &\leq \sigma_{u_-}(s^{u_-}_{\phitan(u_L)})-v \\
&\leq \sigma_{u_-}(s^{u_-}_{\phitan(u_L)})-\sigma_{u_-}(s^{u_-}_{\phitan})+\sigma_{u_-}(s^{u_-}_{\phitan})-v \\
&\leq C|u_--u_L|+(v_0-v) \\
&\leq C\sqrt{a}+(v_0-v),
\end{align*}
where we have used also the smoothness of $\phitan$ and Lemma \ref{pistructurelarge}. So, by taking $a$ sufficiently small, we see $\sigma_{u_-}(s) < v$, a contradiction.
\newline \underline{Case 3:} Assume that $u_- > u_L$. Then, by Lemma \ref{Kru\v{z}kov2}, we have $s \leq s^{u_-}_{\phicom}$. So, by property \ref{enum:3}, we have:
\begin{align*}
\sigma_{u_-}(s)-v &\leq \sigma_{u_-}(s^{u_-}_{\phicom})-v \\
&\leq \sigma_{u_-}(s^{u_-}_{\phicom})-\sigma_{u_-}(s^{u_-}_{\phitan})+\sigma_{u_-}(s^{u_-}_{\phitan})-v \\
&\leq C|u_--u_L|+(v_0-v) \\
&\leq C\sqrt{a}+(v_0-v),
\end{align*}
where we have used Lemmas \ref{companionlemma} and \ref{pistructurelarge}. So, by taking $a$ sufficiently small, we see $\sigma_{u_-}(s) < v$, a contradiction.
\end{proof}
\underline{\textbf{Proof of Proposition \ref{DRHlarge1} (and \ref{Dcontlarge} along the way).}}
\par Now, we give the proofs of Propositions \ref{Dcontlarge} and \ref{DRHlarge1}. Firstly, take $a$ sufficiently small so that $\Pi_a \subset B(u_L,C_0)$, where $C_0$ is the constant in Lemma \ref{localestimate} (this is possible by Lemma \ref{pistructurelarge}), and also $a < a^*$, where $a^*$ is the constant from Lemma \ref{smallshockDRHlarge1}. Now, let $(u_-,u_+)$ be a $\eta,\eta_{\phitan(u_L)}$-entropic shock with $u_- \in \Pi_a$. Then, by property \ref{enum:5}, $u_+=S_{u_-}(s)$ for some $0 \leq s \leq s^{u_-}_{\phidiss}$. So, we may write $D_{RH}(u_-,u_+,\sigma_\pm)$ as follows:
\begin{align*}
D_{RH}(u_-,u_+,\sigma_\pm)&=\frac{1}{a}(-q(u_-;u_L)+\sigma_{u_-}(s)\eta(u_-|u_L)) \\
&+q(S_{u_-}(s);u_R)-\sigma_{u_-}(s)\eta(S_{u_-}(s)|u_R).
\end{align*}
Firstly, assume that $\sigma_{u_-}(s) > v$. Then, by Lemma \ref{smallshockDRHlarge1}, we see $s \leq \frac{s_0}{2}$, so Lemma \ref{entropyformula2} grants:
\[
q(S_{u_-}(s);u_R)-\sigma_{u_-}(s)\eta(S_{u_-}(s)|u_R) \leq -\kappa |\sigma_{u_-}(s)-\sigma_{u_-}(s^u_{u_R})|.
\]
In fact, property \ref{enum:3} gives:
\[
q(S_{u_-}(s);u_R)-\sigma_{u_-}(s)\eta(S_{u_-}(s)|u_R) \leq -\kappa |\sigma_{u_-}(\frac{s_0}{2})-\sigma_{u_-}(s^u_{u_R})|.
\]
Next, we estimate:
\begin{align*}
-q(u_-;u_L)+\sigma_{u_-}(s)\eta(u_-|u_L) &\leq -q(u_-;u_L)+\lambda(u_-)\eta(u_-|u_L) \\
&=-q(u_-;u_L)+v\eta(u_-|u_L)+(\lambda(u_-)-v)\eta(u_-|u_L) \\
&\leq (\lambda(u_-)-v)\eta(u_-|u_L) \\
&\leq a(\lambda(u_-)-v)\eta(u_-|u_R) \\
& \leq aC(\lambda(u_-)-v),
\end{align*}
where we have used Lemma $\ref{localestimate}$ and the definition of $\Pi_a$, and $C=\sup_{u_- \in \statespace}\eta(u_-|u_R)$ depends only on the system, $s_0$, and $d(u_L,0)$. Consolidating both estimates, we obtain:
\[
D_{RH}(u_-,u_+,\sigma_\pm) \leq C(\lambda(u_-)-v)-\kappa |\sigma_{u_-}(\frac{s_0}{2})-\sigma_{u_-}(s^u_{u_R})|
\]
This can be made non-positive uniformly in $u_- \in \Pi_a$ by taking $\lambda(u)-v$ sufficiently small (i.e. picking $\beta$ sufficiently small, cf. the proof of Lemmas \ref{localestimate} and Lemma \ref{smallshockDRHlarge1}). This completes the proof of Proposition \ref{Dcontlarge}, which is a special case of the above considerations with $s=0$ (recall that $\sigma_{u_-}(0)=\lambda(u_-)$).
\par Now, assume that $\sigma_{u_-}(s) \leq v$. Then, we see:
\begin{align*}
D_{RH}(u_-,u_+,\sigma_\pm) &\leq \frac{1}{a}(-q(u_-;u_L)+v \eta(u_-|u_L)) \\
&+q(S_{u_-}(s);u_R)-\sigma_{u_-}(s)\eta(S_{u_-}(s)|u_R).
\end{align*}
Using Lemmas \ref{localestimate} and \ref{quadraticentropy}, we bound the first term by:
\begin{equation*}
D_{RH}(u_-,u_+,\sigma_\pm) \leq \frac{-c^*\beta}{a}|u_--u_L|^2+q(S_{u_-}(s);u_R)-\sigma_{u_-}(s)\eta(S_{u_-}(s)|u_R).
\end{equation*}
So, it suffices to show a bound:
\begin{equation}\label{sufficientbound}
q(S_{u_-}(s);u_R)-\sigma_{u_-}(s)\eta(S_{u_-}(s)|u_R) \leq C|u_--u_L|^2,
\end{equation}
for $a$ sufficiently small, where $a, C$ depend only on the system, and continuously on $s_0$, $d(u_L,0)$, and $d(u_R, \phitan(u_L))$  if $u_R \neq \phitan(u_L)$. Note that a priori, $\beta$ depends on $a$, but a careful look at the proof of Lemmas \ref{localestimate} and \ref{smallshockDRHlarge1} show that we may take $\beta$ fixed as $a \downarrow 0$. 
Using Lemma \ref{entropyformula2}, we see:
\begin{equation}\label{+termDRHlarge1}
q(S_{u_-}(s);u_R)-\sigma_{u_-}(s)\eta(S_{u_-}(s)|u_R)=\int_{s^{u_-}_{u_R}}^s\sigma_{u_-}'(\tau)(\eta(u_-|S_{u_-}(\tau))-\eta(u_-|u_R))\diff \tau.
\end{equation}
We split into two cases, as in the first case below, $a$ will depend on $d(u_R, \phitan(u_L))$, and in the second case it will not. In the first case, we will be able to show that right-hand side of \eqref{sufficientbound} is actually negative, but in the second, we will only be able to obtain a quadratic bound.
\newline \underline{Case 1:} Assume $u_R > \phitan(u_L)$. Then, by the smoothness of $\phitan$, by taking $a$ sufficiently small we may have $\phitan(u_-) < u_R$ for all $u_- \in \Pi_a$, where now $a$ also depends continuously on $d(u_R, \phitan(u_L))$. If $s \leq s^{u_-}_{\phitan}$, then by property \ref{enum:3}, the right-hand side of \eqref{+termDRHlarge1} is non-positive and we are done immediately. Otherwise, we split into three sub-cases.
\newline \underline{Case 1a:} Assume that $u_-=u_L$. Then, $(u_- u_+)$ is $\eta_{\phitan(u_L)}$-entropic, so by Lemma \ref{Kru\v{z}kov1} we must have $s \leq s^{u_-}_{\phitan}$. In this case, the right-hand side of \eqref{+termDRHlarge1} is non-positive.
\newline \underline{Case 1b:} Assume $u_- < u_L$. Then, $\phitan(u_L) < \phitan(u_-)$, so by Lemma \ref{Kru\v{z}kov1}, $s \leq s^{u_-}_{\phitan(u_L)}$. So, we compute:
\begin{align*}
&\int_{s^{u_-}_{u_R}}^s\sigma_{u_-}'(\tau)(\eta(u_-|S_{u_-}(\tau))-\eta(u_-|u_R))\diff \tau \\
&= \int_{s^{u_-}_{u_R}}^{s^{u_-}_{\phitan}}\sigma_{u_-}'(\tau)(\eta(u_-|S_{u_-}(\tau))-\eta(u_-|u_R))\diff \tau \\
&+\int_{s^{u_-}_{\phitan}}^s\sigma_{u_-}'(\tau)(\eta(u_-|S_{u_-}(\tau))-\eta(u_-|u_R))\diff \tau \\
& \leq -C+\int_{s^{u_-}_{\phitan}}^s\sigma_{u_-}'(\tau)(\eta(u_-|S_{u_-}(\tau))-\eta(u_-|u_R))\diff \tau,
\end{align*}
where $C$ may be taken uniformly in $u_- \in \Pi_a$ and depends continuously on $s_0, d(u_L,0)$ and $d(u_R, \phitan(u_L))$. Now, by property \ref{enum:3}, the second term is bounded by:
\begin{align*}
&\int_{s^{u_-}_{\phitan}}^s\sigma_{u_-}'(\tau)(\eta(u_-|S_{u_-}(\tau))-\eta(u_-|u_R))\diff \tau \\
&\leq \int_{s^{u_-}_{\phitan}}^{s_{\phitan(u_L)}^{u_-}}\sigma_{u_-}'(\tau)(\eta(u_-|S_{u_-}(\tau))-\eta(u_-|u_R))\diff \tau \\
&\leq C|\phitan(u_L)-\phitan(u_-)| \\
&\leq C|u_--u_L| \\
& \leq C\sqrt{a},
\end{align*}
where in the second to last inequality we have used the smoothness of $\phitan$ and in the last inequality we have used Lemma \ref{pistructurelarge}. So, taking $a$ sufficiently small, we see that in this case, \eqref{+termDRHlarge1} may be made non-positive as well. 
\newline \underline{Case 1c:} Assume that $u_->u_L$. Then, $\phitan(u_L) > \phitan(u_-)$, so by Lemma \ref{Kru\v{z}kov2}, $s \leq s^{u_-}_{\phicom}$. So, following the same computation as in Case 1b:
\begin{align*}
&\int_{s^{u_-}_{u_R}}^s\sigma_{u_-}'(\tau)(\eta(u_-|S_{u_-}(\tau))-\eta(u_-|u_R))\diff \tau \\
&= \int_{s^{u_-}_{u_R}}^{s^{u_-}_{\phitan}}\sigma_{u_-}'(\tau)(\eta(u_-|S_{u_-}(\tau))-\eta(u_-|u_R))\diff \tau \\
&+\int_{s^{u_-}_{\phitan}}^s\sigma_{u_-}'(\tau)(\eta(u_-|S_{u_-}(\tau))-\eta(u_-|u_R))\diff \tau \\
& \leq -C+\int_{s^{u_-}_{\phitan}}^{s^{u_-}_{\phicom}}\sigma_{u_-}'(\tau)(\eta(u_-|S_{u_-}(\tau))-\eta(u_-|u_R))\diff \tau.
\end{align*}
Estimating the integral on the right-hand side:
\begin{align*}
\int_{s^{u_-}_{\phitan}}^{s^{u_-}_{\phicom}}\sigma_{u_-}'(\tau)(\eta(u_-|S_{u_-}(\tau))-\eta(u_-|u_R))\diff \tau &\leq C|\phitan(u_-)-\phicom(u_-)| \\
&\leq C|\phitan(u_L)-\phicom(u_-)| \\
&= C|\phicom(u_L)-\phicom(u_-)| \\
&\leq C|u_--u_L| \\
&\leq C\sqrt{a},
\end{align*}
where in the second to last inequality we have used Lemma \ref{companionlemma} and in the last we have used Lemma \ref{pistructurelarge}. So, taking $a$ sufficiently small, we see that in this case, \eqref{+termDRHlarge1} may be made non-positive as well.
\newline \underline{Case 2:} Now, we treat the critical case $u_R=\phitan(u_L)$. Again, we split into three sub-cases.
\newline \underline{Case 2a:} Assume $u_-=u_L$. Then, by Lemma \ref{Kru\v{z}kov1}, $s \leq s^{u_-}_{u_R}$, so the left-hand side of \eqref{+termDRHlarge1} is non-positive.
\newline \underline{Case 2b:} Assume $u_- < u_L$. Then, as $(u_-,u_+)$ is $\eta_{\phitan(u_L)}$-entropic, by Lemma \ref{Kru\v{z}kov1}, $s \leq s^{u_-}_{\phitan(u_L)}$. So, we get:
\begin{align*}
&\int_{s^{u_-}_{u_R}}^s\sigma_{u_-}'(\tau)(\eta(u_-|S_{u_-}(\tau))-\eta(u_-|u_R))\diff \tau \\
&\leq \int_{s^{u_-}_{u_R}}^{s^{u_-}_{\phitan(u_L)}}\sigma_{u_-}'(\tau)(\eta(u_-|S_{u_-}(\tau))-\eta(u_-|u_R))\diff \tau \\
&\leq C|\phitan(u_-)-\phitan(u_L)|^2 \\
&\leq C|u_--u_L|^2,
\end{align*}
by the smoothness of $\phitan$, where $C$ depends only on the system. 
\newline \underline{Case 2c:} Assume $u_- > u_L$. Then, $s \leq s^{u_-}_{\phicom}$. Again by property \ref{enum:3}, the left-hand side is bounded by the dissipation when $s=s^{u_-}_{\phicom}$. So, we get:
\begin{align*}
&\int_{s^{u_-}_{u_R}}^s\sigma_{u_-}'(\tau)(\eta(u_-|S_{u_-}(\tau))-\eta(u_-|u_R))\diff \tau \\
&\leq \int_{s^{u_-}_{u_R}}^{s^{u_-}_{\phicom}}\sigma_{u_-}'(\tau)(\eta(u_-|S_{u_-}(\tau))-\eta(u_-|u_R))\diff \tau \\
&\leq C|\phicom(u_-)-\phitan(u_L)|^2 \\
&\leq C|u_--u_L|^2,
\end{align*}
by Lemma \ref{companionlemma}, where $C$ depends on the system and continuously on $d(u_L,0)$. 

\subsection{Proof of Proposition \ref{DRHlarge2}}
Firstly, let $u_+=u_L$. Recall from \eqref{entropydissipation} that the entropy dissipation is defined as:
\[
E_\eta(u_-, u_+)=-\sigma_\pm(\eta(u_+)-\eta(u_-))+q(u_+)-q(u_-).
\]
For $u_+=u_L$, we see from \eqref{DRH}: 
\[
D_{RH}(u_-,u_+,\sigma_\pm) \leq C-\frac{1}{a}E_\eta(u_L,u_-),
\]
where $C$ depends only on the system. Now, $u_- \leq \phidiss(u_L)-\eps$, so by Lemma \ref{entropydissipationlemma}, there exists $\kappa$ such that $E_\eta(u_L,u_-) \geq \kappa$, where $\kappa$ depends continuously on $\eps, d(u_L,0)$. So, we obtain:
\[
D_{RH}(u_-,u_+,\sigma_\pm) \leq -\frac{1}{a}\kappa+C.
\]
Taking $a$ sufficiently small, we obtain $D_{RH} \leq 0$. Now, let $u_+ \in \Pi_a$. Then,
\[
\sigma_\pm \eta(u_-|u_L)-q(u_-;u_L)=\sigma(u_-,u_L)\eta(u_-|u_L)-q(u_-;u_L)+C|u_+-u_L|,
\]
where $C$ depends only on the system. Now, invoking Lemma \ref{pistructurelarge}, we obtain:
\[
\sigma_\pm \eta(u_-|u_L)-q(u_-;u_L) \leq -E_\eta(u_L, u_-)+C\sqrt{a} \leq -\kappa+C\sqrt{a}.
\]
So, in total:
\[
D_{RH}(u_-,u_+,\sigma_\pm) \leq C+\frac{1}{a}(-\kappa+C\sqrt{a}) \leq 0,
\]
for $a$ sufficiently small.

\section{Dissipation rates in the small shock regime}\label{sectionsmall}
\subsection{Outline of the proof of Theorem \ref{theoremsmallshock}}
The approach taken in the proof of Theorem \ref{theoremsmallshock} is vastly different than in the large shock regime. In particular, we need much finer control on the weight $a=\frac{a_2}{a_1}$, which we can only hope to obtain by working locally around a degenerate shock (constant state). To this end, we will first work with shocks localized around a fixed state $d \in [\underline{b}, \overline{b}]$, and extend to a global statement via a Lebesgue number argument. The local work is a recasting of \cite{MR4667839} in the scalar setting, where many of the proofs are simplified. However, we have the added complication that the characteristic field is not globally genuinely nonlinear. The significant change is that we must take care to show that all the computations take place in the region where $f''$ does not degenerate (see Lemma \ref{maximalshock}, and the remark beforehand). The macroscopic structure of the proof is the same as in Section \ref{sectionlarge}. We need analogs of Propositions \ref{Dcontlarge} and \ref{DRHlarge1} (Proposition \ref{DRHlarge2} is now unimportant as we are working in the more restrictive class of weak solutions $\weaktwo$). In order to obtain the finer control on $a$, we will parameterize $\frac{1}{a}=\frac{a_1}{a_2}$ as:
\[
\frac{a_1}{a_2}=1+Cs_0,
\]
where $C,s_0 > 0$ are now parameters that may be varied. In particular, for any $C>0, s_0 > 0$, consider shocks $(u_L,u_R)$ with strength $s_0$. We define:
\begin{align}
\tilde{\eta}(u)&=(1+Cs_0)\eta(u|u_L)-\eta(u|u_R), \label{etatildedef} \\
\tilde{q}(u)&=(1+Cs_0)q(u;u_L)-q(u;u_R). \label{qtildedef}
\end{align}
Further, in the small shock regime, we will rename the set $\Pi_a$ to $\Pismall$ to illustrate the dependence on these new variables:
\[
\Pismall:=\{u \in \statespace|\tilde{\eta}(u) \leq 0\}.
\]
The analogs of $D_{cont}$ and $D_{RH}$ in our new terminology are as follows:
\begin{align*}
D_{cont}(u)&=-\tilde{q}(u)+\lambda(u)\tilde{\eta}(u), \\
D_{RH}(u_-,u_+,\sigma_\pm)&=q(u_+;u_R)-\sigma_\pm \eta(u_+|u_R)-(1+Cs_0)(q(u_-;u_L)-\sigma_\pm \eta(u_-|u_L)).
\end{align*}
The first proposition needed is an analog of Proposition \ref{Dcontlarge} in the small shock regime:
\begin{proposition}\label{Dcontsmall}
Consider \eqref{cl} with $f$ satisfying $\eqref{concaveconvex}$. Fix $d \in [\underline{b}, \overline{b}]$. There is a universal constant $K$ depending only on the system and $d$, such that for any $C_1 > 0$ large enough, there exists $\tilde{s_0}(C_1) > 0$, such that for any $C > 0$ with $\frac{C_1}{2} \leq C \leq 2C_1$, any shock $(u_L, u_R)$ with $|u_L-d|+|u_R-d| \leq \tilde{s_0}$, and $|u_L-u_R|=s_0$ with $0 < s_0 < \tilde{s_0}$, and any $u \in \Pismall$, we have:
\[
D_{cont}(u) \leq -Ks_0^3.
\]
\end{proposition}
Although we only need non-positivity of $D_{cont}$ in $\Pismall$ in the proof of Theorem \ref{theoremsmallshock}, the negativity is in fact needed to prove the next proposition, which is the analog of Proposition \ref{DRHlarge1}:
\begin{proposition}\label{DRHsmall}
Consider \eqref{cl} with $f$ satisfying $\eqref{concaveconvex}$. Fix $d \in [\underline{b}, \overline{b}]$. For any $C_1 > 0$ large enough, there exists $\tilde{s_0}(C_1) > 0$, such that for any $C >0$ with $\frac{C_1}{2} \leq C \leq C_1$, any shock $(u_L, u_R)$ with $|u_L-d|+|u_R-d| \leq \tilde{s_0}$, and $|u_L-u_R|=s_0$ with $0 < s_0 < \tilde{s_0}$, and any $\eta,\eta_{\phitan(\underline{b})}$-entropic shock $(u_-, u_+)$ with $u_- \in \Pismall$ and $u_+ >\phitan(\underline{b})$, we have:
\[
D_{RH}(u_-,u_+,\sigma_\pm) \leq 0.
\]
\end{proposition}
\begin{notation} In this section, we adhere to the following notation: \newline
\begin{enumerate}
    \item In this section, $C$ denotes the specific constant used in the definition of $\Pismall$. This is in contrast with Section 2. Constants denoted $K$ will denote universal constants depending only on the system (where the ``system'' is now composed of $f, \eta, \statespace$, and $[\underline{b}, \overline{b}]$), and continuously on $d \in [\underline{b}, \overline{b}]$. In particular, they will not depend on $C$ or $s_0$.
    \item We will use the notation $a \lesssim b$, which means there exists $K$ (with dependence as above) such that $a \leq Kb$. We also write $a \sim b$ to mean $a \lesssim b$ and $b \lesssim a$. 
    \item The notation $a=b+\mathcal{O}(c)$ is to be interpreted as $|a-b| \lesssim c$.
\end{enumerate}
\end{notation}
\begin{remark}\label{uniformityremarksmall}
In everything that follows, similar to the computations done in the large shock regime, the implicit constants $K$ will depend continuously on $d$. We will be less explicit about this at each step in this section because the continuous dependence on $d$ may be tracked in the same manner as the continuous dependence of the implicit constants on various quantities in the large shock regime. Thus, the thresholds for $C_1$ in Propositions \ref{Dcontsmall} and \ref{DRHsmall} will depend continuously on $d$. This will be used crucially in the proof of Proposition \ref{localtheoremsmallshock}.
\end{remark}
\subsection{Preliminaries}
\subsubsection{Properties of the system}
We retain all the properties stated in Properties $1$ in Section \ref{system properties}.
\subsubsection{Study of the set $\Pismall$}
In this section, we present lemmas that illuminate the geometry of $\Pismall$. We refer to \cite{MR4667839} for the proof of all the results in this section. The first is the analogous statement to Lemma \ref{pistructurelarge} in the small shock regime.
\begin{lemma}[Lemma 3.4 in \cite{MR4667839}]\label{pistructuresmall}
For any $C>0$ and $s_0 > 0$, $\tilde{\eta}$ and $\tilde{q}$ are an entropy, entropy-flux pair for $\eqref{cl}$, and $\Pismall$ is a closed interval. For $C$ sufficiently large and $s_0$ sufficiently small, $\Pismall$ is compactly contained in $\statespace$, and $diam(\Pismall) \sim C^{-1}$, where the implicit constant depends only on the system. We have formulae:
\begin{align*}
\tilde{\eta}'(u)&=Cs_0(\eta'(u)-\eta'(u_L))+(\eta'(u_R)-\eta'(u_L)), \\
\tilde{\eta}''(u)&=Cs_0\eta''(u). 
\end{align*}
Finally, for any $C$ sufficiently big, $s_0$ sufficiently small, and $u \in \Pismall$:
\begin{align*}
    |\tilde{\eta}(u)| &\lesssim s_0, \\
    \tilde{\eta}''(u) &\sim Cs_0.
\end{align*}
\end{lemma}
The second lemma gives control on the derivative of $\tilde{\eta}$ on the boundary of $\Pismall$:
\begin{lemma}[Lemma 3.5 in \cite{MR4667839}]\label{pigradientcontrolbdry}
For $C$ sufficiently large, $s_0$ sufficiently small, and any $\overline{u} \in \partial \Pismall$:
\[
s_0 \lesssim |\tilde{\eta}'(\overline{u})|.
\]
The implicit constant depends only on the system.
\end{lemma}
The last lemma gives precise asymptotics on the negativity of $\tilde{\eta}(u)$ inside of $\Pismall$. 
\begin{lemma}[Lemma 3.6 in \cite{MR4667839}]\label{negativityinpi}
For $C$ sufficiently large, $s_0$ sufficiently small, and $u \in \Pismall$, we have:
\[
-\tilde{\eta}(u) \lesssim s_0d(u,\partial \Pismall).
\]
The implicit constant depends only on the system.
\end{lemma}

\subsection{Proof of Theorem \ref{theoremsmallshock}}\label{shiftconstructionsmall}
In this section, we prove that Propositions \ref{Dcontsmall} and \ref{DRHsmall} imply Theorem \ref{theoremsmallshock}. The construction of the shift is almost the same as in the proof of Theorem \ref{dissipationfnllargeshock}, but with two key differences. The first is that after constructing the shift, we need to perform the local-to-global argument. Secondly, additional care must be taken to obtain the maximal shock speed $\maximalspeedsmall$, as for $u_R$ very close to $u_L$, the constant $C_2$ in the proof of Lemma \ref{qcontrol} degenerates. Due to this, we give an analog to Lemma \ref{qcontrol} that provides greater control on the constant $C_2$ in the small shock regime. This is possible because we now have more control on the geometry of $\Pismall$ due to the localization.
\begin{lemma}[Lemma 4.1 in \cite{MR4667839}]\label{qcontrolsmall}
There exist constants $\delta > 0$, $C^* > 0$ such that for any $u_L, u_R \in \statespace$ with $|u_L-d|+|u_R-d| \leq \tilde{s_0}$:
\begin{align*}
    \tilde{\eta}(u) &\geq \delta s_0 d(u, \partial \Pismall) \text{ for any } u \in \statespace \setminus \Pismall, \\
    |\tilde{q}(u)| &\leq C^*|\tilde{\eta}(u)| \text{ whenever } u \in \statespace \setminus \Pismall \text{ verifies } \tilde{q}(u) \leq 0.
\end{align*}
The constant $\delta$ depends only on the system, while $C^*$ depends continuously on $C$ as well.
\end{lemma}
For the proof, we refer to \cite{MR4667839}. Now, we prove the following proposition, which shows the non-positivity of the dissipation functional for shocks localized around a fixed state $d$. 
\begin{proposition}\label{localtheoremsmallshock}
Consider \eqref{cl} with $f$ verifying \eqref{concaveconvex}. Fix a closed interval $[\underline{b},\overline{b}]$, with $0 < \underline{b} < \overline{b}$. There exists $C_0 > 0$ such that the following holds. Fix $d \in [\underline{b},\overline{b}]$. Then, there exists $\eps_d > 0$ such that the following is true: consider any Ole\u{\i}nik shock $(u_L,u_R)$ with $|u_L-d|+|u_R-d| \leq \eps_d$, and let $|u_L-u_R|=s_0$. Then, for any $a_1, a_2 > 0$ satisfying:
\[
1+\frac{C_0}{2}s_0 \leq \frac{a_1}{a_2}\leq 1+2C_0s_0,
\]
for any $u \in \weaktwo$, any $\overline{t} \in [0,\infty)$, and any $x_0 \in \R$, there exists a Lipschitz shift function $h: [\overline{t} \to \infty) \to \R$, with $h(\overline{t})=x_0$, such that the following dissipation functional verifies:
\begin{equation*}
a_1(\dot{h}(t)\eta(u_-(t)|u_L)-q(u_-(t);u_L))-a_2(\dot{h}(t)\eta(u_+(t)|u_R)-q(u_+(t);u_R)) \leq 0,
\end{equation*}
for almost every $t \in [\overline{t},\infty)$.
\end{proposition}
\begin{proof}
Fix $d \in [\underline{b},\overline{b}]$. We claim that there exists $\eps$ (uniform in $d$) so that:
\begin{equation}\label{midshockruleoutsmall}
\phidiss(\phitan(\underline{b})) < v \text{ for all } v \in B(d,\eps).
\end{equation}
As $d \geq \underline{b}$, it suffices to assume $d=\underline{b}$. Then, $\phidiss(\underline{b}) < \phitan(\underline{b})$, so by taking $\eps$ sufficiently small, by the continuity of $\phidiss$, we may have:
\[
\phidiss(v) < \phitan(\underline{b}),
\]
for all $v \in B(d,\eps)$. But this implies:
\[
\phidiss(\phitan(\underline{b})) < \phidiss(\phidiss(v))=v.
\]
Now, for the last time, fix $C_1$ sufficiently large and $\tilde{s_0}$ sufficiently small so that $\Pismall \subset B(d,\eps)$ for any $u_L \in B(d,\frac{\eps}{2})$ and any $C > \frac{C_1}{2}$ (this is possible by Lemma \ref{pistructuresmall}), and also that satisfy Propositions \ref{Dcontsmall} and \ref{DRHsmall}. Note that as $C_1$ varies continuously in $d$ (see Remark \ref{uniformityremarksmall}), we may take $C_0=\sup_{d \in [\underline{b},\overline{b}]}C_1$ to be uniform in $d$, by taking $\tilde{s_0}$ smaller if needed. At this point, for the rest of the proof, $C$ is proportional to the fixed constant $C_0$. Then, let $\eps_d=\min(\frac{\eps}{2}, \frac{\tilde{s_0}}{2})$, where $\tilde{s_0}=\tilde{s_0}(C_0,d)$, and fix a $\eta$-entropic shock $(u_L,u_R)$  with $u_L,u_R \in B(d, \eps_d)$. From here, the construction of the shift is exactly the same as in Section \ref{shiftconstruction}, with the following changes:
\begin{enumerate}
    \item We use $C^*$, the constant from Lemma \ref{qcontrolsmall} as the constant in the definition of the artifical velocity rather than the constant $C_2$ from Lemma \ref{qcontrol}. Note that we can take the same constant $C^*$ for shocks which are as small as we want, as long as the shock is localized around $d$. Note that $C^*$ is independent of both the state $d$ and $C$, as we have fixed $C$ to be proportional to the fixed constant $C_{0}$.
    \item We use Propositions \ref{Dcontsmall} and \ref{DRHsmall} anywhere Propositions \ref{Dcontlarge} and \ref{DRHlarge1} are used in the large shock regime.
    \item In case 2, we use \eqref{midshockruleoutsmall} instead of \eqref{midshockruleout}, and note shocks with $u_- < \phitan(\underline{b})$ are not possible via the definition of $\weaktwo$ and the maximum principle.
\end{enumerate}
\end{proof}
Now, we prove Theorem \ref{theoremsmallshock}. Consider the following open cover of $[\underline{b}, \overline{b}]$:
\[
\{(d,\eps_d)\}, \text{ where } d \in [\underline{b}, \overline{b}],
\]
where $\eps_d$ is such that Proposition \ref{localtheoremsmallshock} holds with constant $C_0$. By compactness, we can extract a finite sub-cover:
\[
\{(d_n, \eps_{d_n})\}, \text{ for } n \in [1,N],
\]
where for a shock $(u_L, u_R)$ contained in $B(d_n, \eps_{d_n})$, Proposition \ref{localtheoremsmallshock} holds with constant $C_0$. Further, the velocity of the shift is bounded by:
\[
|\dot{h(t)}| \leq C^*+3L,
\]
where $L=\sup_{u \in \statespace}\lambda(u)\}$.
Next, using the Lebesgue number lemma, there exists $\delta > 0$ such that for $|u_L-u_R| < \delta$, $\{u_L,u_R\} \subset B(d_n, \eps_{d_n})$ for some $n$. Finally, take $\tilde{s_0}=\min_n(\delta, \eps_{d_n})$. It remains to show \eqref{weightestimate}. Indeed, let $u_L, u_R \in [\underline{b}, \overline{b}]$, with $|u_L-u_R|=s_0 < \tilde{s_0}$. Then, $\{u_L,u_R\} \subset B(d_n, \eps_{d_n})$ for some $n$, so Proposition \ref{localtheoremsmallshock} holds with:
\[
1+\frac{C_0}{2}s_0 \leq \frac{a_1}{a_2} \leq 1+2C_0s_0.
\]
Finally, the velocity of the shift in all cases is bounded by:
\[
|\dot{h}(t)| \leq C^*+3L=:\maximalspeedsmall.
\]

\subsection{Proof of Proposition \ref{Dcontsmall}}
In this section, we prove Proposition \ref{Dcontsmall}. For this section, we fix a state $d$, and a single Ole\u{\i}nik shock $(u_L,u_R)$ with $u_R=S_{u_L}(s_0)$, and both $u_L, u_R \in [\underline{b}, \overline{b}]$. Firstly, note that $\Pismall$ is compactly contained in $(0,\infty)$ ($\tilde{\eta}(u_R) > 0$, so $\Pismall$ must be contained in $(u_R, \infty)$). We recall that the continuous entropy dissipation is:
\begin{align}\label{Dcontsmallpf}
D_{cont}(u)=-\tilde{q}(u)+\lambda(u)\tilde{\eta}(u),
\end{align}
where $\tilde{\eta}(u)=(1+Cs_0)\eta(u|u_L)-\eta(u|u_R)$ and $\tilde{q}(u)=(1+Cs_0)q(u;u_L)-q(u;u_R)$. It suffices to show $D_{cont}(u) \lesssim -s_0^3$ for $C$ sufficiently large, $s_0$ sufficiently small, and $u \in \Pismall$.
\par The dissipation rate is motivated by the scalar case with convex flux, in which the rate is computed explicitly for Burger's equation (see \cite{MR3322780}, \cite{MR4305935}, \cite{MR4101366}). The dissipation rate was proven in the system case in \cite{MR4667839}. Although we are working in the scalar case, the fact that the flux is not genuinely nonlinear introduces some system-like qualities (we need to use a weight $a$, and there are in some sense two distinct characteristic fields separated by the inflection point). For this reason, we follow the proof of \cite{MR4667839}, although many of the proofs are greatly simplified as a consequence of being in the scalar setting. The strategy is to show that there is a unique maximizer $u^*$ for $D_{cont}$ in $\Pismall$, and then show that $D_{cont}(u^*) \lesssim -s_0^3$. Firstly, we show that a maxima cannot occur in the interior of $\Pismall$:
\begin{lemma}[Lemma 5.1 in \cite{MR4667839}]\label{critptbdry}
Let $u$ be a critical point of $D_{cont}$. Then, $u \in \partial \Pismall$. 
\end{lemma}
\begin{proof}
We compute:
\begin{align*}
D_{cont}'(u)=\tilde{\eta}(u)f''(u).
\end{align*}
As $u$ is a critical point:
\begin{equation*}
\tilde{\eta}(u)f''(u)=0
\end{equation*}
As the characteristic field is genuinely nonlinear in $\Pismall$ by property \ref{enum:1}, we see $\tilde{\eta}(u)=0$.
\end{proof}
As $\Pismall$ is a closed interval, there are exactly two points which may be maximizers of $D_{cont}$ (for existence, note that $D_{cont}$ is a continuous function on the compact set $\Pismall$, so it must attain a maximum). The next lemma shows which one of the two points on the boundary of $\Pismall$ is the maximizer. 
\begin{lemma}[Lemma 5.2 in \cite{MR4667839}]
Let $u^*$ be a maximizer of $D_{cont}$ on $\Pismall$. Then, $u^* < u_L$.
\end{lemma}
\begin{proof}
We compute:
\[
D_{cont}''(u^*)=f''(u^*)\tilde{\eta}'(u^*).
\]
As $u^*$ is a maximizer,  we must have $D_{cont}'''(u^*) \leq 0$. So, if $u^*$ is the right endpoint of $\Pismall$, then $\tilde{\eta}'(u^*) > 0$, a contradiction.
\end{proof}
The previous two lemmas show that it suffices to compute $D_{cont}(u^*)$, where $u^*$ is the left endpoint of $\partial \Pismall$, and show that $D_{cont}(u^*) \lesssim -s_0^3$. To compute this, we introduce some notation. We define $D_R(u)$, and $D_L(u)$ as follows:
\begin{align*}
D_R(u)&:=q(u;u_R)-\lambda(u)\eta(u|u_R), \\
D_L(u)&:=-q(u;u_L)+\lambda(u)\eta(u|u_L).
\end{align*}
Both $D_R$ and $D_L$ are approximately half of the entropy dissipation, i.e. they roughly sum to $D_{cont}$:
\begin{equation*}
D_{cont}(u)=D_R(u)+(1+Cs_0)D_L(u).
\end{equation*}
So, it suffices to show that actually, $D_L(u^*) \lesssim -s_0^3$ and $D_R(u^*) \lesssim -s_0^3$. This is precisely the content of the following lemma.
\begin{lemma}[Lemma 5.7 in \cite{MR4667839}]
For $C$ sufficiently large and $s_0$ sufficiently small:
\begin{align*}
    D_R(u^*) &\lesssim -s_0^3, \\
    D_L(u^*) &\lesssim -s_0^3.
\end{align*}
\end{lemma}
\begin{proof}
Firstly, we show that $|u^*-u_L| \sim s_0$. As $u^* > u_R$, $|u^*-u_L| \lesssim s_0$, so it suffices to show the reverse inequality. We have:
\[
s_0^2=|u_L-u_R|^2 \sim \eta(u_L|u_R)=-\tilde{\eta}(u_L) \lesssim s_0d(u_L,\partial \Pismall),
\]
by Lemmas \ref{quadraticentropy} and \ref{negativityinpi}. Dividing by $s_0$ gives the result. 
\par Next, we estimate $D_L$ as follows. Firstly, we compute:
\[
D_L'(u)=f''(u)\eta(u|u_L) \sim |u-u_L|^2.
\]
Then:
\begin{equation*}
D_L(u^*)=\int_{u_L}^{u^*}D_L'(u)\diff u \lesssim -\int_{u^*}^{u_L}|u-u_L|^2\diff u \lesssim -|u^*-u_L|^3,
\end{equation*}
so we obtain:
\begin{equation*}
D_L(u^*) \lesssim -s_0^3,
\end{equation*}
since $|u^*-u_L| \sim s_0$. The same argument applied to $D_R$ grants the result. 
\end{proof}
This completes the proof of Proposition \ref{Dcontsmall}.
\subsection{Proof of Proposition \ref{DRHsmall}}
In this section, we prove Proposition \ref{DRHsmall}. Again, we follow the same proof outline pioneered in \cite{MR4667839}. This section is more delicate than the proof of Proposition \ref{Dcontsmall}, and in fact the negativity computed in Proposition \ref{Dcontsmall} is needed to show Proposition \ref{DRHsmall}. For convenience, we now give a brief recount of the strategy taken in \cite{MR4667839}, which we also follow here. 
\par Firstly, for any state $u \in \Pismall$, we determine that there is a ``maximal'' shock $(u,u_+)$, such that:
\[
D_{max}(u):=\max_{s \leq s^u_{\phitan(\underline{b}))}}D_{RH}(u,S_u(s),\sigma_u(s))=D_{RH}(u,u_+,\sigma_\pm).
\]
This is proved in Lemma \ref{maximalshock}. Recall that by the definition of $\weaktwo$, we need only consider shocks with both states bounded from below by $\phitan(\underline{b})$, which explains the maximum over the restricted range in the above definition. Next, we decompose $\Pismall$ into multiple sets and show $D_{max}(u) \leq 0$ on each set separately. We divide the proof into two parts based on the length scale of sets in each part:
\newline \underline{Part 1:} $D_{max}(u) \leq 0$ for $u$ in the set $\Pismall \setminus \Rsmall$, which has length scale $diam(\Pismall)\sim C^{-1}$.
\newline \underline{Part 2:} $D_{max}(u) \leq 0$ for $u$ in the set $\Rsmall$, which has length scale $s_0$. 
\newline Within part 1, we decompose the proof into two steps.
\newline \underline{Step 1.1:} Firstly, we show that $D_{max}(u) \leq 0$ for $u$ outside of a set $\Qsmall$. In this case, we use that $u$ is sufficiently far from $u^*$, the unique maximizer of $D_{cont}$ in $\Pismall$, to gain an improvement of Proposition \ref{Dcontsmall} (see Proposition \ref{improvedDcontsmall}).
\newline \underline{Step 1.2:} Secondly, we show that $D_{max}(u) \leq 0$ for $u$ in the set $\Qsmall \setminus \Rsmall$. 
\newline Within part 2, we decompose the proof into four steps, corresponding to a decomposition of $\Rsmall$ into three parts:
\newline \underline{Step 2.1:} We show $D_{max}(u) \leq 0$ in a set $\Rsmallbd$, via methods similar to Step 1.1.
\newline \underline{Step 2.2:} We show $D_{max}(u) \leq 0$ is a set $\Rsmallnot$, a set of states close to $u_L$, by showing that $u_L$ is a local maximizer of $D_{max}$, and $D_{max}(u_L)=0$.
\newline \underline{Step 2.3:} We show $D_{max}(u) \leq 0$ on the remaining region $\Rsmall \setminus (\Rsmallbd \cup \Rsmallnot)$ via an ODE argument. 
\newline Before we begin part 1, we justify the definition of $D_{max}(u)$, i.e. we show that the function:
\[
s \mapsto \mathcal{F}(s,u):=D_{RH}(u,S_u(s),\sigma_u(s))
\]
has a unique maximum $s^*(u)$, with $s^*(u) \lesssim C^{-\frac{1}{2}}s_0^{\frac{1}{2}}$. The proof is based on Lemma 6.1 in \cite{MR4667839}, although the lack of genuine nonlinearity makes things slightly more subtle here.

\begin{lemma}\label{maximalshock}
For $C$ sufficiently large and $s_0$ sufficiently small, the function $s \mapsto \mathcal{F}(s,u)$ has a unique maximum, denoted $s^*$, for $s \in [0, s^u_{\phitan(\underline{b})}]$, which is uniquely defined by the equation:
\[
\eta(u|S_u(s^*))=-\tilde{\eta}(u).
\]
Further, for $u^+=S_u(s^*), |u-u_+| \lesssim C^{\frac{-1}{2}}s_0^{\frac{1}{2}}$, so that for $C$ sufficiently large and $s_0$ sufficiently small, $\mathcal{M}$ is compactly contained in $(0,\infty)$. 
\end{lemma}
\begin{proof}
We split into two cases, depending on whether $\Pismall$ is contained in the interval $[\underline{b},\overline{b}]$ or not.
Firstly, suppose $u > \underline{b}$ for all $u \in \Pismall$. In this case, we also have $\phitan(u) < \phitan(\underline{b})$, as $\phitan$ is monotonically decreasing. Then, we expand $\mathcal{F}(s,u)$ as:
\begin{align*}
\mathcal{F}(s,u)&=\left(q(u;u_R)-\sigma_u(s)\eta(u|u_R)+\int_0^s\sigma_u'(\tau)\eta(u|S_u(\tau))\diff \tau\right) \\
&-(1+Cs_0)(q(u;u_L)-\sigma_u(s)\eta(u|u_L),
\end{align*}
by Lemma \ref{entropyformula}. Taking a derivative:
\[
\dv{\mathcal{F}}{s}=\sigma_u'(s)\left(\tilde{\eta}(u)+\eta(u|S_u(s))\right).
\]
Note that $\sigma_u'(s) < 0$ for all $s \in [0, s^u_{\phitan(\underline{b})})$ by property \ref{enum:3}. So, for $s \in [0, s^u_{\phitan(\underline{b})})$, we obtain $\mathcal{F}(s,u)$ is decreasing in $s$ if and only if:
\[
g(s)=\tilde{\eta}(u)+\eta(u|S_u(s)) \geq 0.
\]
Now, $g(0)=\tilde{\eta}(u) \leq 0$, as $u \in \Pismall$. Further, using property \ref{enum:6}, we see $g'(s) > 0$. So, there is a unique point $s^*$ such that $g(s^*)=0$ for $s \in [0, s^u_{\phitan(\underline{b})})$. It follows that $s^*$ is the unique maximum for $s \mapsto \mathcal{F}(s,u)$ on the interval $[0, s^u_{\phitan(\underline{b})})$, and is characterized by $g(s^*)=0$. It remains to compare $\mathcal{F}(s^*,u)$ and $\mathcal{F}(s^u_{\phitan(\underline{b})},u)$. Using property \ref{enum:3} again, and the fact that $g(s) > 0$ for all $s \in (s^*, s^u_{\phitan(\underline{b})}]$:
\[
\mathcal{F}(s^u_{\phitan(\underline{b})},u)-\mathcal{F}(s^*,u)=\int_{s^*}^{s^u_{\phitan(\underline{b})}}\sigma_u'(s)\left(\tilde{\eta}(u)+\eta(u|S_u(s))\right)  \diff s< 0,
\]
so $s^*$ is the unique maximizer. Further, by letting $u_+=S_u(s^*)$, we estimate:
\[
|u-u_+|^2 \lesssim \eta(u|u_+)=-\tilde{\eta}(u) \lesssim s_0d(u, \partial \Pismall) \lesssim C^{-1}s_0
\]
by using Lemmas \ref{quadraticentropy}, \ref{negativityinpi}, and \ref{pistructuresmall}.
\par The case when there exists $u \in \Pismall$ with $u < \underline{b}$ is more delicate. The main issue is that for $u < \underline{b}$, $\phitan(u) < \phitan(\underline{b})$, and so the derivative does not have only one critical point as was the case before. We deal with this as follows. Making the same considerations as above and using property \ref{enum:3}, we see that $\mathcal{F}(s,u)$ has critical points at $s^*$ and $s^u_{\phitan}$, and the derivative has the following sign in between:
\begin{align*}
    \begin{cases}
        \dv{\mathcal{F}}{s} \geq 0 & s \in [0, s^*], \\
        \dv{\mathcal{F}}{s} \leq 0 & s \in [s^*, s^u_{\phitan}], \\
        \dv{\mathcal{F}}{s} \geq 0 & s \in [ s^u_{\phitan}, s^u_{\phitan(\underline{b})}].
    \end{cases}
\end{align*}
So it suffices to once again compare $\mathcal{F}(s^*,u)$ and $\mathcal{F}(s^u_{\phitan(\underline{b})},u)$. Firstly, $\phitan(\underline{b}) < 0$, so there exists $K$ such that $\phitan(\underline{b})<-\frac{2\underline{b}}{K}$. Take $C$ sufficiently large and $s_0$ sufficiently small so that $S_u(s^*) \geq \frac{\underline{b}}{K}$ for all $u \in \Pismall$, $\phitan(u) < -\frac{\underline{b}}{K}$, and also $u > \frac{\underline{b}}{2}$ for all $u \in \Pismall$. 
We check:
\begin{align*}
\mathcal{F}(s^u_{\phitan(\underline{b})},u)-\mathcal{F}(s^*,u)&=\int_{s^*}^{s^u_{\phitan(\underline{b})}}\sigma_u'(s)\left(\tilde{\eta}(u)+\eta(u|S_u(s))\right) \diff s \\
&=\int_{s^*}^{s^u_{\phitan}}\sigma_u'(s)\left(\tilde{\eta}(u)+\eta(u|S_u(s))\right) \diff s \\
&+\int_{s^u_{\phitan}}^{s^u_{\phitan(\underline{b})}}\sigma_u'(s)\left(\tilde{\eta}(u)+\eta(u|S_u(s))\right) \diff s \\
&\lesssim \int_{s^u_{\frac{\underline{b}}{K}}}^{s^u_{\frac{-\underline{b}}{K}}}\sigma_u'(s)\left(\tilde{\eta}(u)+\eta(u|S_u(s))\right) \diff s+ |u-\underline{b}| \\
&\lesssim -1+|u-\underline{b}| \\
&\lesssim -1+(C^{-1}+\tilde{s_0}),
\end{align*}
where we have used the smoothness of $\phitan$, Lemma \ref{pistructuresmall}, and $|u_L-d| < \tilde{s_0}$. So, taking $C$ larger and $\tilde{s_0}$ smaller as needed, we obtain that $s^*$ is the unique maximizer in this case as well.
\end{proof}

\begin{remark}
As a consequence of the bound $s^*(u) \lesssim C^{-\frac{1}{2}}s_0^{\frac{1}{2}}$, we obtain that for $C$ sufficiently large and $s_0$ sufficiently small, the set of maximal shocks:
\[
\mathcal{M}=\{u_+|u \in \Pismall\}
\]
where $u^+=S_u(s^*(u))$, is compactly contained in $(0,\infty)$. This is particularly significant as we obtain $\sigma_u'(s),\lambda'(u) \sim -1$ for all $u \in \Pismall$, $s \leq s^*$. This is crucially used in the proof of Proposition \ref{OutsideQsmall}.
\end{remark}

\subsubsection{Step 1.1: Shocks outside $\Qsmall$}
In this section, we show that there exists a set $\Qsmall$ such that in the region $\Pismall \setminus \Qsmall$, $D_{cont}$ satisfies an improved estimate over Proposition \ref{Dcontsmall}. This lets us show the negativity of $D_{max}$ on the same region. In particular, take $C$ sufficiently large and $s_0$ sufficiently small such that Proposition \ref{Dcontsmall} is satisfied, i.e. $D_{cont}$ has a unique maximizer $u^*$ on $\Pismall$ (recall that $u^*$ is actually the left endpoint of $\Pismall$). Then, define the interval:
\begin{equation}\label{Qsmalldef}
\Qsmall:=\{p=u^*+b|p \in \Pismall, 0 \leq b \leq C^{\frac{1}{6}}s_0\}.
\end{equation}
\begin{remark}
Note that for $s_0$ sufficiently small, $u^*+C^{\frac{1}{6}}s_0 <\tilde{u}$, where $\tilde{u}$ is the right endpoint of $\Pismall$. This follows from Lemma \ref{pistructuresmall}.
\end{remark}
The main objective in this step is:
\begin{proposition}[[Proposition 6.1 in \cite{MR4667839}]\label{OutsideQsmall}
For $C$ sufficiently large, $s_0$ sufficiently small, and any $u \in \Pismall \setminus \Qsmall$, the maximal entropy dissipation satisfies:
\[
D_{max}(u) \lesssim -s_0^3.
\]
\end{proposition}
The key ingredient needed in the proof of Proposition \ref{OutsideQsmall} is the following improvement of Proposition \ref{Dcontsmall} in this region:
\begin{proposition}[Proposition 6.2 in \cite{MR4667839}]\label{improvedDcontsmall}
For $C$ sufficiently large, $s_0$ sufficiently small, and any $u \in \Pismall \setminus \Qsmall$, the continuous entropy dissipation satisfies:
\[
D_{cont}(u) \lesssim -C^{\frac{1}{3}}s_0^3.
\]
\end{proposition}
\underline{\textbf{Proof of Proposition \ref{improvedDcontsmall}:}}
By Lemma \ref{critptbdry}, $D_{cont}$ does not have any critical points in the interior of $\Pismall$, so we need only show the improved dissipation rate on the boundary of the region $\Pismall \setminus \Qsmall$. This consists of two points: $u^*+C^{\frac{1}{6}}s_0$ and $\tilde{u}$, where $\tilde{u}$ is the right endpoint of $\Pismall$. We first give a couple of lemmas needed to show the improved estimate at $u^*+C^{\frac{1}{6}}s_0$.
\begin{lemma}[Lemma 5.3 in \cite{MR4667839}]\label{secondderivDcont}
For $C$ sufficiently large and $s_0$ sufficiently small:
\[
D_{cont}''(u^*) \lesssim -s_0.
\]
\end{lemma}
\begin{proof}
We compute:
\[
D_{cont}''(u)=\tilde{\eta}(u)f'''(u)+f''(u)\tilde{\eta}'(u).
\]
Using $\tilde{\eta}(u^*)=0$, we obtain:
\[
D_{cont}''(u^*)=f''(u^*)\tilde{\eta}'(u^*).
\]
Lemma $\ref{pistructuresmall}$ then gives the result.
\end{proof}
The second lemma is an estimate on $D_{cont}'''$ in $\Pismall$.
\begin{lemma}[Lemma 5.9 in \cite{MR4667839}]\label{thirdderivDcont}
For $C$ sufficiently large, $s_0$ sufficiently small, and any $u \in \Pismall$:
\[
|D_{cont}'''(u)| \lesssim Cs_0.
\]
\end{lemma}
\begin{proof}
A routine computation gives:
\[
D_{cont}'''(u)=\tilde{\eta}(u)f^{(4)}(u)+2\tilde{\eta}'(u)f'''(u)+f''(u)\tilde{\eta}''(u).
\]
The two terms with derivatives of $\tilde{\eta}$ are $\lesssim Cs_0$ by Lemma \ref{pistructuresmall}, and by Lemmas \ref{negativityinpi} and \ref{pistructuresmall}, the remaining term verifies:
\[
|\tilde{\eta}(u)f^{(4)}(u)| \lesssim s_0d(u,\partial \Pismall) \lesssim C^{-1}s_0.
\]
\end{proof}
Now, we prove the improved dissipation rate.
\begin{lemma}[Lemma 6.2 in \cite{MR4667839}]
For $C$ sufficiently small and $s_0$ sufficiently small, we have:
\[
D_{cont}(u^*+C^{\frac{1}{6}}s_0) \lesssim -C^{\frac{1}{3}}s_0^3.
\]
\end{lemma}
\begin{proof}
Using Taylor's theorem, for $u=u^*+C^{\frac{1}{6}}s_0$, we estimate:
\begin{align*}
D_{cont}(u)&=D_{cont}(u^*)+D_{cont}'(u^*)(u-u*) \\
&+\frac{1}{2}D_{cont}''(u^*)(u-u^*)^2+\frac{1}{2}\int_{u^*}^uD_{cont}'''(u^*)(u-\tau)^2 \diff \tau \\
&\lesssim -s_0^3-(s_0)(C^{\frac{1}{6}}s_0)^2+(Cs_0)(C^{\frac{1}{6}}s_0)^3 \\
&\lesssim -C^{\frac{1}{3}}s_0^3,
\end{align*}
for $s_0$ sufficiently small, where we have used Proposition \ref{Dcontsmall}, $D_{cont}'(u^*)=0$, and Lemmas \ref{secondderivDcont} and \ref{thirdderivDcont}.
\end{proof}
The next step is to prove the improved dissipation rate at $\tilde{u}$ (recall that $\tilde{u}$ is the right endpoint of $\Pismall$):
\begin{lemma}[Lemma 6.5 in \cite{MR4667839}]
For $C$ sufficiently large and $s_0$ sufficiently small, we have:
\[
D_{cont}(\tilde{u}) \lesssim -C^{-2}s_0.
\]
\end{lemma}
\begin{proof}
For any $u > u^*$, we have:
\begin{align}\label{Dcontattildeu}
D_{cont}(\tilde{u})&=D_{cont}(u)+\int_0^{\tilde{u}-u}D_{cont}'(v) \diff v \notag \\ 
&\lesssim -s_0^3+\int_0^{\tilde{u}-u}\tilde{\eta}(v)f''(v)\diff v \notag \\ 
&\lesssim -s_0^3+\int_0^{\tilde{u}-u}\tilde{\eta}(v)\diff v,
\end{align}
where we have used Proposition \ref{Dcontsmall}. Now, we leverage Lemma \ref{pigradientcontrolbdry}, which implies $|\tilde{\eta}'(\tilde{u})| \sim s_0$ (where we have also used Lemma \ref{pistructuresmall}), to show that in fact, for $v$ sufficiently close to $\tilde{u}$, $\tilde{\eta}(v) \lesssim -s_0(\tilde{u}-v)$. Indeed, as $\tilde{\eta}$ is increasing at $\tilde{u}$, we see:
\[
-\tilde{\eta}'(\tilde{u}) \sim -s_0.
\]
So, by the mean value theorem, there exists $s$ with $|s-\tilde{u}| \leq |v-\tilde{u}|$ such that:
\begin{align*}
-\tilde{\eta}'(v)&=\tilde{\eta}'(\tilde{u})-\tilde{\eta}''(s)(s-\tilde{u}) \\
&\lesssim -s_0+Cs_0|v-\tilde{u}| \\
&\lesssim -s_0,
\end{align*}
by taking $|v-\tilde{u}| \lesssim C^{-1}$. So, for such $v$, we obtain:
\[
\tilde{\eta}(v)=-\int_0^{\tilde{u}-v}\tilde{\eta}'(s)\diff s \lesssim -s_0(\tilde{u}-v).
\]
Substituting this into \eqref{Dcontattildeu}, we see:
\[
D_{cont}(\tilde{u}) \lesssim -s_0^3+\int_0^{\tilde{u}-u}\tilde{\eta}(v)\diff v \lesssim -s_0^3-C^{-2}s_0 \lesssim -C^2s_0,
\]
for $C$ sufficiently large and $s_0$ sufficiently small. 
\end{proof}
\begin{remark}
Note that for $C$ sufficiently large and $s_0$ sufficiently small $-C^{-2}s_0 \lesssim -C^{\frac{1}{3}}s_0^3$. 
\end{remark}
This completes the proof of Proposition \ref{improvedDcontsmall}.

\underline{\textbf{Proof of Proposition \ref{OutsideQsmall}:}}
Now, we give the proof of Proposition \ref{OutsideQsmall}. We split into two cases, depending on whether $u_+=S_u(s^*)$, the unique maximum of $D_{max}$ established in Lemma \ref{maximalshock}, is inside $\Pismall$ or not. Firstly, we prove a lemma that shows that if $u_+ \in \Pismall$, there is no need for the improved estimate Proposition \ref{improvedDcontsmall}. 
\begin{lemma}[Lemma 6.6 in \cite{MR4667839}]\label{bothinPi}
Suppose that the maximal shock $(u,u_+)$ satisfies $u,u_+ \in \Pismall$. Then:
\[
D_{RH}(u,u_+, \sigma_\pm) \leq D_{cont}(u_+) \lesssim -s_0^3.
\]
\end{lemma}
\begin{proof}
The entropy dissipation is bounded by:
\begin{align*}
D_{RH}(u,u_+, \sigma_\pm) &\leq (q(u^+;u_R)-\sigma_\pm\eta(u^+|u_R))-(1+Cs_0)(q(u_+;u_L)-\sigma_\pm\eta(u_+|u_L)),
\end{align*}
by Lemma \ref{entropyformula}, as $(u,u_+)$ is $\eta$-entropic. Then, adding and subtracting $\lambda(u_+)\eta(u_+|u_R)$ and $ \lambda(u_+)(1+Cs_0)\eta(u_+|u_L)$, we obtain:
\begin{align*}
D_{RH}(u,u_+, \sigma_\pm) &\leq D_{cont}(u_+)+(\sigma_\pm-\lambda(u_+))(\eta(u_+|u_R)+(1+Cs_0)\eta(u_+|u_L)) \\
&\leq D_{cont}(u_+) \lesssim -s_0^3,
\end{align*}
where in the second to last inequality we used $u_+ \in \Pi$, and $\sigma_\pm > \lambda(u_+)$ by property \ref{enum:4}. The last inequality follows from Proposition \ref{Dcontsmall}.
\end{proof}
Now, recall the notation $\mathcal{F}(s,u)=D_{RH}(u,S_u(s), \sigma_u(s))$. By Lemma \ref{maximalshock}, it suffices to show that $\mathcal{F}(s^*,u)$ satisfies the estimate $\mathcal{F}(s^*,u) \lesssim -s_0^3$. By Lemma \ref{bothinPi}, we may assume that $u_+ \not \in \Pismall$. Then, via the fundamental theorem of calculus:
\begin{align}\label{maximaldissipationpf}
\mathcal{F}(s^*,u)&=D_{cont}(u)+\int_0^{s^*}\sigma_u'(\tau)(\tilde{\eta}(u)+\eta(u|S_u(\tau)))\diff \tau \notag \\
&\lesssim -C^{\frac{1}{3}}s_0^3-\int_0^{s^*}\sigma_u'(\tau)(\tilde{\eta}(u)+\eta(u|S_u(\tau)))\diff \tau \notag \\
&\lesssim -C^{\frac{1}{3}}s_0^3-\int_0^{s^*}\left(|u-S_u(\tau)|^2-s_0d(u,\partial \Pismall)\right) \diff \tau \notag \\
&\lesssim -C^{\frac{1}{3}}s_0^3-(s^*)^3+s_0s^*d(u,\partial \Pismall),
\end{align}
where we have used Proposition \ref{improvedDcontsmall}, $\sigma_u' \sim -1$, Lemma \ref{negativityinpi}, and Lemma \ref{quadraticentropy}. Next, since $u_+ \not \in \Pismall$, we have $d(u,\partial \Pismall) \lesssim |u-u_+| \lesssim s^*$. Further, we see:
\[
(s^*)^2 \lesssim |u-u_+|^2 \lesssim \eta(u|u_+)=-\tilde{\eta}(u)\lesssim s_0d(u,\partial \Pismall) \lesssim s_0s^*,
\]
which grants $d(u,\partial \Pismall) \lesssim s^* \lesssim s_0$. Substituting this into \eqref{maximaldissipationpf}, we see:
\[
\mathcal{F}(s^*,u) \lesssim -s_0^3,
\]
for $C$ sufficiently large. 

\subsubsection{Step 1.2: Shocks in $\Qsmall \setminus \Rsmall$}
In this section, we show that there exists a set $\Rsmall$ of order $s_0$ around $u_L$ such that $D_{max}(u) \leq 0$ except on $R_{C,s_0}$. We define a family of intervals $\Rsmall(K_h)$ as:
\begin{equation*}
\Rsmall(K_h):=\{p=u^*+b|p \in \Pismall, 0 \leq b \leq K_hs_0\}.
\end{equation*}
\begin{remark}
Note that for $C$ sufficiently large, and $K_h$ fixed (to be fixed later), $\Rsmall \subset \Qsmall$.
\end{remark}
The main proposition to be proved in this step is the following:
\begin{proposition}[Proposition 6.3 in \cite{MR4667839}]\label{improvedDcontsmallRsmall}
There exists $K_h > 0$ depending only on the system such that for $C$ sufficiently large, $s_0$ sufficiently small, and $u \not \in \Rsmall(K_h)$:
\[
D_{max}(u) \lesssim -s_0^3.
\]
\end{proposition}
Informally, in this step, we prove $D_{max}(u) \lesssim -s_0^3$ for $u$ at a sort of ``medium'' distance from $u_L$; states in $\Qsmall \setminus \Rsmall$ are at a distance more than of order $s_0$ from $u_L$, but are also far from the boundary of $\Pismall$ in some sense. A quantitative justification for this is the following lemma:
\begin{lemma}[Lemma 6.7 in \cite{MR4667839}]\label{Kstar}
There exists a universal constant $K^*$ such that for $C$ sufficiently large, $s_0$ sufficiently small, and any $u \in \Pismall$ satisfying:
\[
-\tilde{\eta}(u) \geq K^*s_0^2,
\]
$u_+ \in \Pismall$.
\end{lemma}
\begin{proof}
Let $-\tilde{u} \geq K^*s_0^2$, and assume $u_+ \not \in \Pismall$. Then, letting $K_1$ be the universal constant from Lemma \ref{negativityinpi}, we see:
\[
K^*s_0^2 \leq -\tilde{\eta}(u) \leq K_1s_0d(u,\partial \Pismall) \leq K_1s_0|u-u_+|.
\]
Letting $K_2$ be the universal constant from Lemma \ref{quadraticentropy}, Lemma \ref{maximalshock} implies: 
\[
K_2|u-u_+|^2 \leq \eta(u|u_+)=-\tilde{\eta}(u) \leq K_1s_0|u-u_+|.
\]
Combining these, we see:
\[
K^*s_0 \leq \frac{K_1^2}{K_2}s_0.
\]
Taking $K^*$ strictly larger than $\frac{K_1^2}{K_2}$ gives a contradiction, so we must have $u_+ \in \Pismall$.
\end{proof}
\begin{remark}
By Lemma \ref{bothinPi}, we then see that if $-\tilde{\eta}(u) \geq K^*s_0^2$, then $D_{max}(u) \lesssim -s_0^3$. By Lemma \ref{negativityinpi}, this roughly says that if $s_0 \lesssim d(u,\partial \Pismall)$, then $D_{max}(u) \lesssim -s_0^3$ holds. 
\end{remark}
\underline{\textbf{Proof of Proposition \ref{improvedDcontsmallRsmall}:}} Now, we prove Proposition \ref{improvedDcontsmallRsmall}. The strategy is to show that for $u \in \Qsmall$, there exists a univeral constant $K_h$ such that for $b > K_hs_0$, $C$ sufficiently large, and $s_0$ sufficiently small, one has $-\tilde{\eta}(u) \geq K^*s_0^2$, where $K^*$ is the universal constant from Lemma \ref{Kstar}. Combined with Lemma \ref{bothinPi}, this will complete the proof of Proposition \ref{improvedDcontsmallRsmall}. 
\par So, suppose $u \in \Qsmall$, with $K_hs_0 < b \leq C^{\frac{1}{6}}s_0$ (where $K_h$ is yet to be determined). Then, using the fundamental theorem of calculus, we write:
\begin{equation}\label{tildeetacalc}
-\tilde{\eta}(u)=-\tilde{\eta}(u^*)+\int_{u}^{u^*}\tilde{\eta}'(v)\diff v.
\end{equation}
Using $|u-u_L| \leq C^{\frac{1}{6}}s_0$ and Lemma \ref{pistructuresmall}, we obtain:
\begin{equation}\label{integrandcalulcation}
\tilde{\eta}'(u)=Cs_0(\eta'(u)-\eta'(u_L))+(\eta'(u_R)-\eta'(u_L))=-s_0\eta''(u_L)+\mathcal{O}(C^{\frac{7}{6}}s_0^2).
\end{equation}
So, writing $v$ as $v(t)=tu^*+(1-t)u$, we have $v'(t)=-b$. Re-parameterizing the integral in \eqref{tildeetacalc} and using \eqref{integrandcalulcation}:
\begin{align*}
-\tilde{\eta}(u)&=\int_0^1\left(-s_0\eta''(u_L)+\mathcal{O}(C^{\frac{7}{6}}s_0^2)\right)(-b)\diff s \\
&=\int_0^1\left(bs_0\eta''(u_L)+\mathcal{O}(C^{\frac{4}{3}}s_0^3)\right) \diff s \\
&=bs_0\eta''(u_L)+\mathcal{O}(C^{\frac{4}{3}}s_0^3)).
\end{align*}
Since $\eta''(u_L) \geq K_1 > 0$, and $b \geq K_hs_0$, we can write this as:
\[
-\tilde{\eta}(u) \geq K_1K_hs_0^2-K_2C^{\frac{4}{3}}s_0^3.
\]
So, taking $s_0$ sufficiently small (depending on $C$), we may obtain:
\[
-\tilde{\eta}(u) \geq \frac{K_1K_h}{2}s_0^2.
\]
Choosing $K_h=\frac{2K^*}{K_1}$ grants the result. 
\subsubsection{Step 2.1: Shocks close to $\partial \Pismall$}  In this step, we prove that there exists a constant $K_{bd} > 0$ such that if $|u-u^*| \leq K_{bd}s_0$, then $D_{max}(u) \lesssim -s_0^3$. Consequently, if we define the set:
\begin{equation*}
    \Rsmallbd:=\{u \in \Rsmall| |u-u^*| \leq K_{bd}s_0\},
\end{equation*}
then we have $D_{max} \lesssim -s_0^3$ on $\Rsmallbd$. This is contained in the following proposition:
\begin{proposition}\label{Rsmallbdprop}
There exists a universal constant $K_{bd} > 0$ such that for $C$ sufficiently large, $s_0$ sufficiently small, and any $u \in \Rsmallbd$, $D_{max}(u) \lesssim -s_0^3$.
\end{proposition}
\begin{proof}
Following the proof of Proposition \ref{OutsideQsmall}, we use the fundamental theorem of calculus to estimate:
\begin{equation}
    D_{max}(u) \lesssim D_{cont}(u)+s_0^2d(u, \partial \Pismall).
\end{equation}
The key difference is that in this region, we cannot use Lemma \ref{improvedDcontsmall} to blow up the first term (cf. \eqref{maximaldissipationpf}). So, we need a bound on the second. By definition, there are universal constants $K_{bd}, K > 0$ such that for $d(u, \partial \Pismall) \leq \frac{1}{2}K_{bd}s_0$, then $D_{max}(u) \leq -Ks_0^3$ (via Proposition \ref{Dcontsmall}). Finally, for $u \in \Rsmall$, $d(u, \partial \Pismall) \leq \frac{1}{2}K_{bd}s_0$ implies $|u-u^*| \leq K_{bd}s_0$ for $C$ sufficiently large and $s_0$ sufficiently small.
\end{proof}
\subsubsection{Step 2.2: Shocks close to $u_L$} Now, we prove that states near $u_L$ satisfy $D_{max}(u) \leq 0$. In this region, we cannot hope to obtain $D_{max}(u) \lesssim -s_0^3$, because $u_+(u_L)=u_R$, and so $D_{max}(u_L)=D_{RH}(u_L, u_R, \sigma(u_L,u_R))=0$. However, we can still show $D_{max}(u) \leq 0$ by showing that $u_L$ is a local maximum for $D_{max}$. We begin by stating some algebraic identities and estimates. The proofs of the following lemmas are straightforward and follow from taking derivatives of the defining equations for $u^+$, and using local asymptotics on $\sigma_\pm$ (cf. Lemma 3.2 in \cite{MR4667839}, and the references to \cite{MR546634} cited therein), so we simply refer to  \cite{MR4667839} for the proofs. 
\begin{lemma}[Lemma 6.8 in \cite{MR4667839}]
Let $u_+=u_+(u)=S_u(s^*)$ be such that $D_{max}(u)=D_{RH}(u,u_+, \sigma_\pm)$. Then, the following identities are valid:
\begin{align}
&(f'(u_+)-\sigma_\pm)u_+'=f'(u)-\sigma_\pm+(u_+-u)\sigma_\pm', \label{implicitderiv1} \\
&\eta'(u_R)-\eta'(u_+)+(1+Cs_0)(\eta'(u)-\eta'(u_L))=\eta''(u_+)u_+'(u-u_+), \label{implicitderiv2} \\ 
&D_{max}'(u)=(\eta'(u_+)-\eta'(u_R)-(1+Cs_0)(\eta'(u)-\eta'(u_L)))(f'(u)-\sigma_\pm). \label{implicitderiv3}
\end{align}
\end{lemma}

\begin{lemma}[Lemmas 6.9, 6.10 in \cite{MR4667839}]\label{u+estimates}
For $C$ suficiently large, $s_0$ sufficiently small, and any $u \in \Rsmall \setminus \Rsmallbd$:
\begin{align*}
    |u-u_+| &\sim s_0, \\
    \lambda(u)-\sigma_\pm &\sim s_0, \\
    \sigma_\pm-\lambda(u_+) &\sim s_0, \\
    |u_+'|+|\sigma_\pm'| &\lesssim 1, \\
    |u_+''|+|\sigma_\pm''| &\lesssim s_0^{-1}, \\
    |D_{max}'''(u)| &\lesssim 1.
\end{align*}
\end{lemma}
Now, we state and prove the main proposition of this step.
\begin{proposition}[Proposition 6.5 in \cite{MR4667839}]\label{Rsmallnotbd}
There exists a universal constant $K_0 > 0$ such that for $C$ sufficiently large $s_0$ sufficiently small, and any $u \in \Rsmallnot$, where:
\[
\Rsmallnot:=\{u \in \Rsmall||u-u_L| \leq K_0s_0\},
\]
we have:
\[
D_{max}(u) \leq 0.
\]
\end{proposition}
\begin{proof}
Firstly $u_+(u_L)=u_R$, which implies $D_{max}(u_L)=D_{RH}(u_L, u_R, \sigma(u_L,u_R))=0$. Secondly, evaluating \eqref{implicitderiv3} at $u_L$:
\[
D_{max}'(u_L)=0.
\]
So, $u_L$ is a critical point of $D_{max}$. Next, we compute $D_{max}''(u_L)$, and obtain:
\begin{equation}\label{Dmax''ul}
D_{max}''(u_L)=(\eta''(u_R)u_+'-(1+Cs_0)\eta''(u_L))(f'(u_L)-\sigma(u_L,u_R)).
\end{equation}
We show that $D_{max}''(u_L) \lesssim -s_0$. Indeed, evaluating \eqref{implicitderiv2} at $u_L$, we obtain:
\[
s_0^{-1}(u_R-u_L)\eta''(u_R)u_+'(u_L)=0,
\]
implying $u_+'(u_L)=0$. Returning to \eqref{Dmax''ul}:
\begin{equation}
D_{max}''(u_L)=-(1+Cs_0)\eta''(u_L)(f'(u_L)-\sigma(u_L,u_R))
\end{equation}
By Lemma \ref{u+estimates}, we see $D_{max}''(u_L) \lesssim -s_0$. In other words, there exists a constant $K_1 > 0$ such that for $C$ sufficiently large, and $s_0$ sufficiently small:
\[
D_{max}''(u_L) \leq -K_1s_0.
\]
Finally, we construct $\Rsmallnot$. Let $u \in \Rsmall \setminus \Rsmallbd$, and let $K_2$ be the constant from Lemma \ref{u+estimates} such that $|D_{max}'''(u)| \leq K_3$. Then, by Taylor's theorem, there exists $v$ between $u$ and $u_L$ such that:
\[
D_{max}(u)=\frac{1}{2}D_{max}''(u_L)(u-u_L)^2+\frac{1}{6}D_{max}'''(v)(u-u_L)^3.
\]
This implies:
\[
D_{max}(u) \leq \frac{-K_1}{2}s_0|u-u_L|^2+\frac{K_3}{6}|u-u_L|^3,
\]
which implies that there is a constant $K_0$ such that if $|u-u_L| < K_0$, then $D_{max}(u) \leq 0$ for $|u-u_L| \leq K_0s_0$, for $C$ sufficiently large and $s_0$ sufficiently small. 
\end{proof}

\subsubsection{Step 2.3: Shocks in $\Rsmallplus$ and $\Rsmallminus$} Finally, we show that $D_{max}(u) \leq 0$ for $u$ in the rest of $\Rsmall$. For any $u \in \Rsmall$ we can write:
\[
u=u_L-b(u).
\]
Then, $\Rsmall \setminus (\Rsmallbd \cup \Rsmallnot)$ is partitioned into two disconnected intervals, which we denote as follows:
\begin{align*}
\Rsmallplus&:=\{u \in \Rsmall \setminus (\Rsmallbd \cup \Rsmallnot)|b(u) > 0\}, \\
\Rsmallminus&:=\{u \in \Rsmall \setminus (\Rsmallbd \cup \Rsmallnot)|b(u) < 0\}. \\
\end{align*}
In both regions, the proof follows via an ODE argument. The strategy is to show that there are no critical points for $D_{max}$ in either region. We start with a characterization of critical points which will let us define the relevant ODE:
\begin{lemma}[Lemma 6.11 in \cite{MR4667839}]\label{critptcharlemma}
Let $u$ be in the interior of $\Pismall$. Then, $u$ is a critical point of $D_{max}$ if and only if:
\begin{equation}\label{critptchar}
\eta'(u_+)-\eta'(u_R)-(1+Cs_0)(\eta'(u)-\eta'(u_L))=0.
\end{equation}
\end{lemma}
\begin{proof}
We use \eqref{implicitderiv3}, which grants:
\[
D_{max}'(u)=(\eta'(u_+)-\eta'(u_R)-(1+Cs_0)(\eta'(u)-\eta'(u_L)))(f'(u)-\sigma_\pm).
\]
So, if $\eqref{critptchar}$ holds, then $D_{max}'(u)=0$. On the other hand, if $D_{max}'(u)=0$, and $u$ is in the interior of $\Pismall$, we claim $f'(u) \neq \sigma_\pm$, and the result follows. Indeed, since $u$ is in the interior of $\Pismall$, $u^+ \neq u$, and by property \ref{enum:4}, $\sigma_\pm < \lambda(u)$. 
\end{proof}
Now, we define the ODE that will track the evolution of the quantities we want to estimate:
\begin{lemma}[Lemma 6.12 in \cite{MR4667839}]\label{ODE}
Let $u \in \Rsmall \setminus \Rsmallbd$, and define $u(t)=tu+(1-t)u_L$. Define the functions $F,G,e:[0,1] \to \R$ as:
\begin{align*}
    F(t)&:=-(1+Cs_0)(\eta'(u(t))-\eta'(u_L))sgn(b), \\
    G(t)&:=-(\eta'(u_+(t)))-\eta'(u_R))sgn(b), \\
    e(t)&:=\frac{|u_+(t)-u(t)|}{b}.
\end{align*}
Then, $G$ satisfies the ODE:
\begin{align}\label{ODEeqn}
    -e(t)G'(t)+G(t)&=F(t), \\
    G(0)&=0.
\end{align}
\end{lemma}
\begin{proof}
Evaluating \eqref{implicitderiv2} at $u(t)$ and multiplying by $-sgn(b)$ gives:
\[
G(t)-F(t)=-sgn(b)(u_+(t)-u(t))\eta''(u_+(t))u_+'(t).
\]
Computing $G'(t)$, we obtain:
\[
G'(t)=-sgn(b)\eta''(u_+(t))u_+'(t)(u-u_L),
\]
because $u'(t)=u-u_L$. So, $G$ solves:
\begin{align*}
    -e(t)G'(t)+G(t)&=F(t), \\
    G(0)&=0.
\end{align*}
\end{proof}
A consequence of Lemmas \ref{critptcharlemma} and \ref{ODE} is that $u$ is a critical point of $D_{max}$ if and only if $F(1)=G(1)$. We will show this is not the case for $u$ in both remaining regions in the following analysis. 
\begin{lemma}[Lemma 6.13 in \cite{MR4667839}]\label{nocritptplus}
For $C$ sufficiently large, $s_0$ sufficiently small, and any $u \in \Rsmallplus$, $u$ is not a critical point of $D_{max}$.
\end{lemma}
\begin{proof}
Let $u \in \Rsmallplus$, and define $u,F,G$, and $e$ as in Lemma \ref{ODE}. Solving \eqref{ODEeqn}, we obtain:
\[
G(t)=\frac{-1}{E(t)}\int_0^t\frac{E(s)}{e(s)}F(s)\diff s, \text{ where } E(t)=\exp\left(\int_0^t-e(s)^{-1}\diff s\right).
\]
We will show that $F(1) \neq G(1)$. Firstly, note that $e(t), E(t), F(t) \geq 0$ for $s_0$ sufficiently small, and all $t \in [0,1]$, so $G(1) \leq 0$. So, it suffices to show that $F(1)>0$. Indeed, we have:
\[
F(1)=-(1+Cs_0)(\eta'(u)-\eta'(u_L)),
\]
in $\Rsmallplus$. Taylor expanding and using the fact that $u \in \Rsmallplus$, so $K_0s_0 \leq b \leq K_hs_0$:
\begin{align*}
F(1)&=-(1+Cs_0)(-\frac{1}{2}\eta''(u_L)(u-u_L)+\mathcal{O}((u-u_L)^2) \\
    &=\frac{(1+Cs_0)}{2}\eta''(u_L)b+\mathcal{O}(K_hs_0^2) \\
    &\geq \frac{1}{2}\eta''(u_L)K_0s_0+\mathcal{O}(K_hs_0^2),
\end{align*}
which is strictly positive for $s_0$ sufficiently small. 
\end{proof}

\begin{lemma}[Lemma 6.14 in \cite{MR4667839}]\label{nocritptminus}
For $C$ sufficiently large, $s_0$ sufficiently small, and any $u \in \Rsmallminus$, $u$ is not a critical point of $D_{max}$.
\end{lemma}
\begin{proof}
Let $u \in \Rsmallminus$, and define $u,F,G$, and $e$ as in Lemma \ref{ODE}. Solving \eqref{ODEeqn}, we obtain:
\[
G(t)=\frac{1}{E(t)}\int_0^t\frac{E(s)}{|e(s)|}F(s)\diff s, \text{ where } E(t)=\exp\left(\int_0^t|e(s)|^{-1}\diff s\right).
\]
Now, we estimate:
\[
F'(t)=(1+Cs_0)\eta''(u(t))(u-u_L)=\eta''(u_L)+\mathcal{O}(s_0),
\]
so $F$ is strictly increasing for $s_0$ sufficiently small. Next, we estimate:
\[
G(1) \leq \left(\frac{1}{E(1)}\int_0^1\frac{E(s)}{|e(s)|}\diff s\right) F(1).
\]
Using $E'(s)=\frac{E(s)}{|e(s)|}$:
\[
G(1) \leq \left(1-\frac{E(0)}{E(1)}\right)F(1).
\]
Since $E(s) \geq 0$, $E(0)=1$, and $E'(s) \sim 1$, $\left(1-\frac{E(0)}{E(1)}\right)<1$, granting the result. 
\end{proof}

\underline{\textbf{Proof of Proposition \ref{DRHsmall}}.}
Fix $C$ sufficiently large and $s_0$ sufficiently small so that Propositions \ref{OutsideQsmall}, \ref{improvedDcontsmallRsmall}, \ref{Rsmallbdprop}, \ref{Rsmallnotbd}, and Lemmas \ref{nocritptplus} and \ref{nocritptminus} are all valid. Then, we have a decomposition:
\[
\Pismall=\Rsmall \cup (\Pismall \setminus \Rsmall),
\]
where $D_{max}(u) \lesssim -s_0^3$ on $\Pismall \setminus \Rsmall$ by Proposition \ref{improvedDcontsmallRsmall}. Then, we further decompose $\Rsmall$ as:
\[
\Rsmall=\Rsmallbd \cup \Rsmallplus \cup \Rsmallnot \cup \Rsmallminus.
\]
By Propositions \ref{Rsmallbdprop} and \ref{Rsmallnotbd}, $D_{max}(u) \leq 0$ on $\Rsmallbd \cup \Rsmallnot$. Finally, $D_{max}$ has no critical points in $\Rsmallplus \cup \Rsmallminus$, so:
\[
\max_{u \in \Rsmallplus \cup \Rsmallminus}D_{max}(u)=\max_{u \in \partial \Rsmallplus \cup \partial \Rsmallminus}D_{max}(u) \leq 0,
\]
as the boundary points of $\Rsmallplus, \Rsmallminus$ are contained in either $\Rsmallbd, \Rsmallnot$, or $\Qsmall$. Finally, by the definition of $D_{max}$, for any entropic shock $(u_-, u_+)$ with $u_- \in \Pismall$ and $u_+ > \phitan(\underline{b})$:
\[
D_{RH}(u_-, u_+, \sigma_\pm) \leq D_{max}(u_-) \leq 0,
\]
with equality if and only if $u_-=u_L, u_+=u_R$. 

\section{$L^2$-stability}\label{sectionstability}
\subsection{Outline of the proof of Theorem \ref{theoremstability}}
In this section, we prove Theorem \ref{theoremstability}. The main proposition of this section is as follows:
\begin{proposition}\label{closetomaxp}
Consider \eqref{cl} with $f$ verifying \eqref{concaveconvex}. Fix $0 < \underline{b} < \overline{b}$. Then, there exists $\tilde{C},v>0$ such that the following is true: for any $m > 0, R,T > 0, u^0 \in BV(\R)$ with $Range(u^0) \subset [\underline{b}, \overline{b}]$, and any wild solution $u \in \weaktwo$, there exists $\psi:\R^+ \times \R \to \statespace$ such that for almost every $0 < s < t < T:$
\begin{align*}
    TV(\psi(t,\cdot)) &\leq TV(u^0), \\
    Range(\psi(t, \cdot)) &\subset [\underline{b}, \overline{b}], \\
    ||\psi(t,\cdot)-\psi(s,\cdot)||_{L^1} &\leq \tilde{C}|t-s|, \\
    ||\psi(t,\cdot)-u(t,\cdot)||_{L^2(-R+vt,R-vt)} &\leq C\left(||u^0-u(0,\cdot)||_{L^2(-R,R)}+\frac{1}{m}\right).
\end{align*}
Finally, the constants $C$ and $\tilde{C}$ depend only on the system and $TV(u_0)$. 
\end{proposition}
The functions $\psi$ are not solutions to \eqref{cl}, because their shocks do not verify the Rankine-Hugoniot conditions. So this proposition does not give stability between two solutions to the conservation law. In fact, what the proposition says is that, given a weak solution $u \in \weaktwo$ with initial data close to having range contained in $[\underline{b},\overline{b}]$, then in fact for all time $u$ will remain close to having range contained in $[\underline{b},\overline{b}]$ (by Chebyshev's inequality). If $u(0,\cdot)=u^0$, we will be able to show that the range of $u$ is in fact contained in $[\underline{b},\overline{b}]$ by passing into the limit via the following compactness theorem, which is a simple consequence of the Aubin-Lions lemma. For a proof, we refer to the appendix of \cite{MR4487515}.
\begin{lemma}[Lemma 2.6 in \cite{MR4487515}]\label{aubinlions}
Let $\{\psi_n\}_{n \in \N}$ be a family of piecewise constant functions uniformly bounded in $L^\infty(\R^+,BV(\R))$. Assume that there exists $\tilde{C}> 0$ such that for every $n \in \N$, $\psi_n$ verifies:
\[
||\psi_n(t,\cdot)-\psi_n(s,\cdot)||_{L^1} \leq \tilde{C}|t-s|.
\]
Then, there exists $\psi \in L^\infty(\R^+ \times \R)$ such that, up to a subsequence, $\psi_n$ converges to $\psi$ as $n \to \infty$ in $C^0([0,T];L^2(-R,R))$ for every $T,R>0$, and almost everywhere in $\R^+ \times \R$. 
\end{lemma}
From there, we can use any of the uniqueness results for convex fluxes (\cite{MR1296003}, \cite{MR2104269}, \cite{MR4623215}, \cite{MR3954680}) to conclude that $u$ is in fact the unique Kru\v{z}kov solution to the conservation law. 
\par The strategy is to construct the function $\psi$ via a modified front tracking method. We use exactly the classical front tracking method introduced by Dafermos \cite{MR303068}, but with one key difference: shocks travel at artificial velocities determined by Corollary \ref{uniformalarge} and Theorem \ref{theoremsmallshock} rather than the Rankine-Hugoniot speeds.
\par The major difficulty in this section is the construction of the weight function $a$. In order to obtain $L^2$ stability, the weight function must be reconstructed whenever any two waves collide. The estimate on $a$ in Theorem \ref{theoremsmallshock} is crucial at this step to give a global bound on the constant $C$ in Proposition \ref{closetomaxp} uniformly in $m$. 

\subsection{Proof of Theorem \ref{theoremstability}}
Let $u$ be the Kru\v{z}kov solution to \eqref{cl} with initial data $u^0$. Passing to a subsequence if necessary, we assume that $||u^0_n-u^0||_{L^2} \leq \frac{1}{n}$. By Proposition \ref{closetomaxp}, there exists a sequence $\psi_n$ uniformly bounded in $L^\infty(\R^+,BV(\R))$ such that:
\begin{align}\label{convergence}
    ||\psi_n(t,\cdot)-u_n(t,\cdot)||_{L^2} \leq \frac{2C}{n},
\end{align}
where $C$ is the constant from Proposition \ref{closetomaxp}. By Lemma \ref{aubinlions}, there exists $\psi \in L^\infty(\R^+ \times \R)$ such that for every $T,R>0$, $\psi_n$ converges in $C^0([0,T];L^2((-R,R)))$ to $\psi$. By \eqref{convergence}, $u_n$ converges in $L^\infty([0,T];L^2((-R,R)))$ to $\psi$. As the convergence is strong, and $u_n$ is a $\eta$-entropic solution to \eqref{cl}, so is $\psi$. Finally, as $\psi_n(t,\cdot)$ has range contained in $[\underline{b},\overline{b}]$ for all $t$, so does $\psi$. Finally, note that both $\psi$ and $u$ are $\eta$-entropic weak solutions for the conservation law with strictly convex flux function:
\[
\overline{f}(u)=\begin{cases}
f(u) & \underline{b} \leq u < \infty, \\
\frac{f''(\underline{b})}{2}(u-\underline{b})^2+f'(\underline{b})(u-\underline{b})+f(\underline{b}) & -\infty < u \leq \underline{b}.
\end{cases}
\]
So, by applying Corollary 2.5 in \cite{MR2104269}, we conclude that $\psi=u$.

\subsection{Relative entropy for the Riemann problem}
In this section, we give the $L^2$-stability results for shocks and rarefactions that will guide the construction of the weight function $a$. The first proposition is a synthesis of Corollary \ref{uniformalarge} and Theorem \ref{theoremsmallshock}.
\begin{proposition}\label{riemannstabilityshock}
Consider \eqref{cl} with $f$ verifying \eqref{concaveconvex}. Fix $0<\underline{b}<\overline{b}$. Then, there exist constants $C_1, C_0, \eps$, and  $\hat{\lambda}$ such that the following is true.
\par Consider any $\eta$-entropic shock $(u_L,u_R)$, any $u \in \weaktwo$, any $\overline{t} \in [0,\infty)$, and any $x_0 \in \R$. Let $\sigma$ be the strength of the shock $\sigma=|u_L-u_R|$. Then, if $\sigma \geq \frac{\eps}{2}$, for any $a_1$, $a_2$ verifying:
\[
\frac{a_2}{a_1} \leq C_1,
\]
there exists a Lipschitz path $h_1:[\overline{t},\infty) \to \R$, with $h_1(\overline{t})=x_0$, such that the following dissipation functional verifies:
\begin{equation*}
a_1(\dot{h}_1(t)\eta(u_-(t)|u_L)-q(u_-(t);u_L))-a_2(\dot{h}_1(t)\eta(u_+(t)|u_R)-q(u_+(t);u_R)) \leq 0,
\end{equation*}
for almost every $t \in [\overline{t},\infty)$. On the other hand, if $\sigma < \eps$, for any $a_1, a_2 > 0$ verifying:
\[
1-2C_0s_0 \leq \frac{a_2}{a_1} \leq 1-\frac{C_0}{2}s_0,
\]
there exists a Lipschitz path $h_2:[\overline{t},\infty) \to \R$, with $h_2(\overline{t})=x_0$, such that the following dissipation functional verifies:
\begin{equation*}
a_1(\dot{h}_2(t)\eta(u_-(t)|u_L)-q(u_-(t);u_L))-a_2(\dot{h}_2(t)\eta(u_+(t)|u_R)-q(u_+(t);u_R)) \leq 0,
\end{equation*}
for almost every $t \in [\overline{t},\infty)$. In all cases, there is a maximal shock speed: $|\dot{h}_{1,2}(t)| \leq \hat{\lambda}$ for almost every $t \in [\overline{t},\infty)$.
\end{proposition}
Notice that we have have an interval where the two regimes overlap; for shocks with $\frac{\eps}{2} \leq \sigma < \eps$, we can have the non-positivity of the dissipation functional with either choice of $\frac{a_2}{a_1}$. This will be crucial in the construction of the weight $a$ for the front tracking algorithm.
\par We need similar control on discretized rarefaction shocks that are created in the front tracking method. We begin by showing that the real rarefaction has a contraction property without the need for a weight or a shift. 
\begin{lemma}\label{rarefactionstability}
Consider \eqref{cl} with $f$ verifying \eqref{concaveconvex}. Let $\overline{u}(y)$, $v_L \leq y \leq v_R$, be a rarefaction wave for \eqref{cl}, with $\overline{u}(v_L)=u_L, \overline{u}(v_R)=u_R$. Assume that $u_L, u_R \in [\underline{b}, \overline{b}]$. Then, for any $u \in \weaktwo$, and every $t > 0$, we have:
\begin{align*}
\dv{}{t}\int_{v_Lt}^{v_Rt}\eta(u(t,x)|\overline{u}(\frac{x}{t})\diff x &\leq q(u(t,v_Lt+);\overline{u}(v_L))-q(u(t,v_Rt-);\overline{u}(v_R)) \\
&-v_L\eta(u(t,v_Lt+)|\overline{u}(v_L))+v_R\eta(u(t,v_Rt-)|\overline{u}(v_R)).
\end{align*}
\end{lemma}
\begin{proof}
Let $v(t,x)=\overline{u}(\frac{x}{t})$. Following the classical weak-strong stability proof of \cite{MR546634} (see also \cite{MR4487515}), as $\overline{u}$ is a Lipschitz solution, we have:
\[
\partial_t\eta(u|v)+\partial_xq(u;v)+\partial_x(\eta'(v))f(u|v) \leq 0.
\]
We will show that the last term is signed. Indeed as $\overline{u}$ is a rarefaction, we have:
\[
\partial_x(\eta'(v))=\eta''(v)v_x \geq 0,
\]
for any $t$. Secondly, for any $v \in [u_L, u_R]$, we see:
\[
f(u|v)=f(u)-f(v)-f'(v)(u-v),
\]
is non-negative for $u \geq \phitaninverse(v)$ by definition. As $u(t,x) \geq \phitan(\underline{b}) \geq \phitaninverse(\underline{b}) \geq \phitaninverse(v)$ for any $(t,x) \in \R^+ \times \R$ (by the maximum principle, and because $\phitaninverse$ is decreasing), we see $f(u|v) \geq 0$. Thus we see:
\[
\partial_t\eta(u|v)+\partial_xq(u;v) \leq 0.
\]
Integrating in $x$ between $v_Lt$ and $v_Rt$ and using the Strong Trace property (Definition \ref{strongtrace}) grants the result.
\end{proof}
This lets us give control on the error of the discretized rarefaction shocks in the front tracking method. The following proof is from \cite{MR4487515}:
\begin{proposition}\label{rarefactionshockerror}
There exists a constant $C>0$ such that the following is true. For any rarefaction wave $\overline{u}(y), v_L \leq y \leq v_R$, with $\overline{u}(v_L)=u_L, \overline{u}(v_R)=u_R$, $u_L, u_R \in [\underline{b}, \overline{b}]$, denote:
\[
\delta=|v_L-v_R|+\sup_{y \in [v_L,v_R]}|u_L-\overline{u}(y)|.
\]
Then, for any $u \in \weaktwo$, any $v_L \leq v \leq v_R$, and any $t > 0$, we have:
\begin{align*}
&\int_0^t\left(q(u(t,vt+);u_R)-q(u(t,vt-);u_L)-v(\eta(u(t,vt+)|u_R)-\eta(u(t,vt-)|u_L))\right)\diff t \\
&\leq C\delta |u_L-u_R|t.
\end{align*}
\begin{proof}
We define the quantity:
\[
D=\dv{}{t}\left(\int_{v_Lt}^{vt}\eta(u(t,x)|u_L)\diff x+\int_{vt}^{v_Rt}\eta(u(t,x)|u_R)\diff x -\int_{v_Lt}^{v_Rt}\eta(u(t,x)|\overline{u}\left(\frac{x}{t}\right))\diff x \right).
\]
Using \eqref{relativeentropic} twice (once with $b=u_L$, and again with $b=u_R$), and Lemma \ref{rarefactionstability}, we see:
\[
D \geq q(u(t,vt+);u_R)-q(u(t,vt-);u_L)-v\left(\eta(u(t,vt+)|u_R)-\eta(u(t,vt-)|u_L)\right).
\]
Integrating this in time, we see:
\begin{align*}
&\int_0^t\left(q(u(t,vt+);u_R)-q(u(t,vt-);u_L)-v(\eta(u(t,vt+)|u_R)-\eta(u(t,vt-)|u_L))\right)\diff t \\
&\leq \int_0^tD\diff t \\
&\leq \int_{v_Lt}^{vt}\eta(u(t,x)|u_L)\diff x+\int_{vt}^{v_Rt}\eta(u(t,x)|u_R)\diff x -\int_{v_Lt}^{v_Rt}\eta(u(t,x)|\overline{u}\left(\frac{x}{t}\right))\diff x \\ 
&\leq \int_{v_Lt}^{vt}\left(\eta(\overline{u}\left(\frac{x}{t}\right))-\eta(u_L)+\eta'(\overline{u}\left(\frac{x}{t}\right))(u-\overline{u}\left(\frac{x}{t}\right))-\eta'(u_L)(u-u_L)\right)\diff x \\
&+ \int_{vt}^{v_Rt}\left(\eta(\overline{u}\left(\frac{x}{t}\right))-\eta(u_R)+\eta'(\overline{u}\left(\frac{x}{t}\right))(u-\overline{u}\left(\frac{x}{t}\right))-\eta'(u_R)(u-u_R)\right)\diff x \\
&\leq C\delta |u_L-u_R|t.
\end{align*}
\end{proof}
\end{proposition}

\subsection{Modified front tracking algorithm}
For an introduction to the front tracking algorithm in the scalar case, we refer to Chapter 6 of Bressan's book \cite{MR1816648}, Chapter 14 of Dafermos' book \cite{MR3468916} or Chapter 4 of LeFloch's book \cite{MR1927887}. Informally, the idea is to take a piecewise-constant approximation of the initial data, solve the Riemann problems produced, and then let the solution propagate until a time when two waves interact. Then, the Riemann problems produced are again solved, and the waves are allowed to propagate again. This is repeated ad infinitum (although one can prove the process eventually terminates in a finite number of interactions). As we change the velocity of the shocks to move at a shifted speed determined by Proposition \ref{riemannstabilityshock} rather than the Rankine-Hugoniot speed, and we are in a somewhat simplified setting due to the initial data being localized to the convex region of the flux, we will give a quick introduction to the types of wave interactions we will encounter below.
\subsubsection{The Riemann solver}\label{riemannsolver}
Fix a parameter $h > 0$ with $h < \frac{\eps}{2}$, where $\eps$ is the threshold constant from Proposition \ref{riemannstabilityshock}. We distinguish between three types of waves: ``big'' shocks, ``small'' shocks, and rarefaction shocks, which we will henceforth simply refer to as rarefactions. The initial classification of waves in these three buckets will be discussed in Section \ref{algorithm}. Now, assume that at a positive time $\overline{t}$, there is an interaction at the point $\overline{x}$ between two waves: $(u_L, u_M)$ with strength $\sigma_1$ and $(u_M,u_R)$ with strength $\sigma_2$. We distinguish between so-called ``monotone'' interactions, in which $(u_M-u_L)(u_R-u_M) \geq 0$, and ``non-monotone'' interactions, in which $(u_M-u_L)(u_R-u_M) < 0$.
\newline \underline{\textbf{A: Monotone interactions.}}
\begin{enumerate}
    \item Big shock-big shock: if two big shocks interact, the solution of the Riemann problem will be the following big shock:
    \[
    v_a(t,x)=\begin{cases}
                u_L & x < h_1(t), \\
                u_R & x > h_1(t),
              \end{cases}
    \]
    where $h_1(t)$ is the shift function with initial data $h_1(\overline{t})=\overline{x}$ from Proposition \ref{riemannstabilityshock}.
    \item Rarefaction-rarefaction: as all rarefactions will travel at the characteristic speed of the right-hand state, and all the computations are localized to the region where the flux is convex, this wave interaction is not actually possible.
    \item Big shock-small shock: in this case, the solution to the Riemann problem will be a big shock:
    \[
    v_a(t,x)=\begin{cases}
                u_L & x < h_1(t), \\
                u_R & x > h_1(t).
              \end{cases}
    \]
    \item Small shock-small shock: in this case, there are two possible solutions. If $\sigma_1+\sigma_2 \geq \eps$, then the solution will be a big shock:
    \[
    v_a(t,x)=\begin{cases}
                u_L & x < h_1(t), \\
                u_R & x > h_1(t).
              \end{cases}
    \]
    On the other hand, if $\sigma_1+\sigma_2 < \eps$, the solution will be classified as a small shock, and will move with the speed $h_2$.
    \[
    v_a(t,x)=\begin{cases}
                u_L & x < h_2(t), \\
                u_R & x > h_2(t).
              \end{cases}
    \]
\end{enumerate}
\underline{\textbf{B: Non-monotone interactions.}}
\begin{enumerate}
 \item Rarefaction-big shock: in this case, there are two possible solutions. Let $\sigma_1$ be the strength of the rarefaction, and let $\sigma_2$ be the strength of the big shock. If $\sigma_2-\sigma_1 > \frac{\eps}{2}$, then the solution will be a big shock:
    \[
    v_a(t,x)=\begin{cases}
                u_L & x < h_1(t), \\
                u_R & x > h_1(t).
              \end{cases}
    \]
    On the other hand, if $\sigma_2-\sigma_1 < \frac{\eps}{2}$, the solution will be classified as a small shock, and will move with the speed $h_2$.
    \[
    v_a(t,x)=\begin{cases}
                u_L & x < h_2(t), \\
                u_R & x > h_2(t).
              \end{cases}
    \]
\item Rarefaction-small shock: in this case, there are two possible solutions. Let $\sigma_1$ be the strength of the rarefaction, and let $\sigma_2$ be the strength of the small shock. If $\sigma_1 > \sigma_2$, then the result will be a rarefaction with strength $\sigma_1-\sigma_2$:
    \[
    v_a(t,x)=\begin{cases}
                u_L & x < \overline{x}+(t-\overline{t})\lambda(u_R), \\
                u_R & x > \overline{x}+(t-\overline{t})\lambda(u_R).
              \end{cases}
    \]
    On the other hand, if $\sigma_2 > \sigma_1$, the result will be a small shock with strength $\sigma_2-\sigma_1$:
    \[
    v_a(t,x)=\begin{cases}
                u_L & x < h_2(t), \\
                u_R & x > h_2(t).
              \end{cases}
    \]
\end{enumerate}
\subsubsection{Construction of the approximate solutions}\label{algorithm}
Fix a function $u^0 \in BV(\R)$ with $Range(u^0) \subset [\underline{b},\overline{b}]$. Note that $f$ is strictly convex in this region. Then, fix a piecewise constant initial data $\psi_h(0)$ that verifies the following properties:
\begin{align*}
    &\psi_h(0,\cdot) \text{ has at most $\frac{1}{h}$ jump discontinuities}, \\
    &TV(\psi_h(0,\cdot)) \leq TV(u^0), \\
    &\inf u^0 \leq \psi_h(0,\cdot) \leq \sup u^0.
\end{align*}
Then, at each jump discontinuity of $\psi_h(0,\cdot)$, we can solve (locally in time) the Riemann problem. This works as follows. Let $u_L,u_R$ be two states on either side of a jump discontinuity at $\overline{x}$. If $u_L > u_R$, the Riemann solution is a shock. We will assign all shocks with size $|\sigma| > \eps$ to be ``big'', and all shocks with size $|\sigma| < \eps$ to be ``small''. We let big shocks propagate with speed $h_1(t)$ and small shocks propagate with speed $h_2(t)$. On the other hand, if $u_L < u_R$, the classical Riemann solution is a rarefaction fan, which we must replace with several ``rarefaction shocks'' each of strength $< h$, i.e. the Riemann solution will be:
\[
v_a(t,x)=\begin{cases}
            u_L & x< \overline{x}+t\lambda(w_1), \\
            w_j & \overline{x}+t\lambda(w_{j}) < x < \overline{x}+t\lambda(w_{j+1}) \indent (1 \leq j \leq N-1),\\
            u_R & x > \overline{x}+t\lambda(u_R), 
         \end{cases}
\]
where $N=\ceil{\frac{\sigma}{h}}$, and:
\[
w_j=u_L+\frac{j}{N}(u_R-u_L) \indent \text{ for } 1 \leq j \leq N-1.
\]

Then, we let the waves propagate. By slightly perturbing the speed of a wave if necessary, we can assume that at any time $t > 0$, there is at most one wave interaction. At the first interaction point between two waves, we solve the Riemann problem as dictated in Section \ref{riemannsolver}. Then, we continue recursively. The following lemma shows this process will produce a function $\psi_h$ with the following properties:
\begin{lemma}[Chapter 4, Theorem 2.1 in \cite{MR1927887}]\label{fronttrackingwelldefined}
The wave front tracking approximations $\psi_h$ are well-defined globally in time. In particular, the total number of waves in $\psi_h$ is uniformly bounded in $t$ (but tends to infinity when $h$ tends to $0$). Further, the approximations satisfy the uniform estimates:
\begin{align*}
    &TV(\psi_h(t,x)) \leq TV(u^0), \\
    &\inf u^0 \leq \psi_h(t,x) \leq \sup u^0, \\
    &||\psi_h(t,\cdot)-\psi_h(s,\cdot)||_{L^1(\R)} \leq TV(u^0)\hat{\lambda}|t-s|, \text{ for } 0 < s < t < T.
\end{align*}
\end{lemma}
We omit the proof, as using the shifts instead of the Rankine-Hugoniot speeds only affects the Lipschitz constant in time, where we must make sure that the maximum speed of any wave is bounded. This is guaranteed by Proposition \ref{riemannstabilityshock}, where we show that the speed of all waves is bounded by $\hat{\lambda}$. All other aspects of the proof are identical. For completeness, a proof will be given in the author's PhD thesis \cite{mythesis}.
\subsubsection{The weight function $a$}
Now, we define the function $a$ that gives stability of a weak solution $u \in \weaktwo$ and the function $\psi_h$ constructed via the front tracking algorithm:
\begin{align*}
a(t,x)=\left(C_1^{L(t)}\right)\left(\displaystyle \prod_{k_i \in K(t)}k_i\right)\left(\displaystyle \prod_{i:\text{shock}}\xi_i(t,x)\right),
\end{align*}
where the functions $L(t),\xi_i(t)$, and the set $K(t)$ are defined as follows.
\newline \underline{\textbf{L(t):}} The function $L(t)$ is a sum of ``big shock numbers'' associated to each wave present in $\psi_h(t,\cdot)$:
\[
L(t):=\sum_{i:\text{small shock or rarefaction}}\ell_i+\sum_{i:\text{big shock}}(\ell_i-1),
\]
where each wave is assigned $\ell_i$ via the following designation. Initially, all big shocks are given $\ell_i=1$, and all small shocks and rarefactions are given $\ell_i=0$. When two waves collide, the big shock number of the resulting wave will be the sum of the big shock numbers of the two incoming waves, unless the interaction is an A4-type interaction (see Section \ref{riemannsolver}) in which a big shock is produced, in which case the resulting big shock will have big shock number the sum of the incoming ones plus one.
\newline \underline{\textbf{K(t):}} The set $K(t)$ contains ``small shock numbers'' accumulated as waves interact in the following manner. Initially, $K(0+)=\emptyset$. When two waves interact in an $A3$-type interaction at a time $\overline{t}$, $K(\overline{t}+)=K(\overline{t}-) \bigcup \{1-C_0\sigma_1\}$, where $\sigma_1$ is the strength of the incoming small shock. Likewise, when a $B2$-type interaction occurs in which a rarefaction is produced at time $\overline{t}$, $K(\overline{t}+)=K(\overline{t}-) \bigcup \{1-C_0\sigma_1\}$, where $\sigma_1$ is the strength of the incoming small shock. Finally, when a $B2$-type interaction occurs in which a small shock is produced at time $\overline{t}$, $K(\overline{t}+)=K(\overline{t}-) \bigcup \{\frac{1-C_0\sigma_1}{1-C_0(\sigma_1-\sigma_2)}\}$, where $\sigma_1$ is the strength of the incoming small shock and $\sigma_2$ is the strength of the incoming rarefaction. 
\newline \underline{$\boldsymbol{\xi_i(t):}$} A shock wave $x_i(t)$ with strength $\sigma_i$ has the associated function $\xi_i(t)$:
\begin{align*}
\xi_i(t,x)=\begin{cases}
            \ind_{\{x < x_i(t)\}}(x)+(1-C_0\sigma_i)\ind_{\{x > x_i(t)\}}(x) & \text{if $x_i(t)$ is a small shock}, \\
            \ind_{\{x < x_i(t)\}}(x)+C_1\ind_{\{x > x_i(t)\}}(x) & \text{if $x_i(t)$ is a big shock}.
         \end{cases}
\end{align*}
Clearly, the component of the function $a$ involving the weights $\xi_i$ is the portion that grants the $L^2$-stability. The other two components are needed for the second of the following crucial properties.
\begin{proposition}\label{aproperties}
There exists $c=c(TV(u^0)) > 0$ such that, uniformly in $h$:
\begin{align}\label{abounded}
c \leq a(t,x) \leq 1 \text{ for any } (t,x) \in \R^+ \times \R.
\end{align}
Furthermore, for every time $t$ with a wave interaction, and almost every $x$:
\begin{align}\label{amonotone}
a(t+,x) \leq a(t-,x).
\end{align}
Moreover, for every time without a wave interaction, and for every $x$ such that a big shock is located at $x=x_i(t)$:
\begin{align}\label{abigjump}
\frac{a(t,x_i(t)+)}{a(t,x_i(t)-)}=C_1,
\end{align}
and if a small shock of strength $\sigma_i$ is located at $x=x_i(t)$:
\begin{align}\label{asmalljump}
\frac{a(t,x_i(t)+)}{a(t,x_i(t)-)}=1-C_0\sigma_i.
\end{align}
\end{proposition}
Before we begin the proof, we give a lemma that will uniformly bound diminishing contributions in $a$ coming from small shocks.
\begin{lemma}\label{productlemma}
Let $\{a_i\}_{i=1}^N$ be a finite set of real numbers with $0 < a_i < \frac{1}{2}$ for all $1 \leq i \leq N$, and $\sum_{i=1}^Na_i \leq K$ for some $K > 1$. Then:
\[
\prod_{i=1}^N(1-a_i) \geq \left(\frac{1}{2}\right)^{4K}.
\]
\end{lemma}
\begin{proof}
Firstly, we will show the following:
\begin{align}\label{intermedsum}
\sum_{i=1}^N a_i < 1 \implies \prod_{i=1}^N(1-a_i) \geq 1-\sum_{i=1}^N a_i.
\end{align}
We proceed by induction. Assume that $N=1$. Then:
\[
\prod_{i=1}^1(1-a_i)=1-a_i=1-\sum_{i=1}^1 a_i,
\]
so \eqref{intermedsum} is verified in this case. Now, assume that \eqref{intermedsum} holds for $N=k$, and let $N=k+1$. Then:
\[
\prod_{i=1}^{k+1}(1-a_i)=(1-a_{k+1})(\prod_{i=1}^k(1-a_i)) \geq (1-a_{k+1})(1-\sum_{i=1}^k(1-a_i) \geq 1-\sum_{i=1}^{k+1}(1-a_i).
\]
Now, we prove the main lemma. Firstly, assume that $\sum_{i=1}^Na_i < \frac{1}{4}$. Then:
\[
\prod_{i=1}^N(1-a_i) \geq 1-\sum_{i=1}^Na_i \geq \frac{3}{4} \geq \left(\frac{1}{2}\right)^{4K}.
\]
Otherwise, partition $\{a_i\}_{i=1}^N$ into $M$ subsets $\{a_{i,1}\}, ..., \{a_{i,<M}\}$, with $M \leq 4K$, such that:
\[
\frac{1}{4} \leq \sum_ia_{i,j} \leq \frac{1}{2} \text{ for all } 1 \leq j \leq 4K.
\]
Then, applying \eqref{intermedsum} $M$ times, we have:
\[
\prod_{i=1}^N(1-a_i)=\prod_{j=1}^{M}\prod_{i}(1-a_i) \geq \prod_{j=1}^{M}\left(\frac{1}{2}\right) \geq \left(\frac{1}{2}\right)^{M} \geq \left(\frac{1}{2}\right)^{4K}.
\]
\end{proof}
\underline{\textbf{Proof of Proposition \ref{aproperties}.}}
We start by proving \eqref{abounded}. Firstly, the upper bound is trivial, as all the components in the definition of $a$ are strictly less than $1$. So, we need only to show the lower bound. Fix an arbitrary $h$, and consider the weight $a(t,x)$ associated to $\psi_h$. It suffices to show that each of the three components of $a$ enjoys a uniform lower bound separately. Denote $V:=TV(u^0)$. We show a lower bound on $\displaystyle \prod_{i:\text{shock}}\xi_i(t,x)$ as follows. Firstly, at any time $t$, the sum of the strengths of small shocks is bounded by the total variation:
\[
\sum_{i:\text{small shock}}\sigma_i \leq TV(\psi_h(t,\cdot)) \leq TV(\psi_h(0,\cdot)) \leq V,
\]
So by applying Lemma \ref{productlemma}, we see:
\[
\prod_{i:\text{small shock}}(1-C_0\sigma_i) \geq \left(\frac{1}{2}\right)^{4C_0V}.
\] 
Secondly, the strength of any big shock must be greater than or equal to $\frac{\eps}{2}$, and the scheme is total variation decreasing, so there are at most $\ceil{\frac{2V}{\eps}}$ big shocks at any fixed time. Together, these considerations imply:
\[
\displaystyle \prod_{i:\text{shock}}\xi_i(t,x) \geq \left(\frac{1}{2}\right)^{4C_0V}\left(C_1^{\ceil{\frac{2V}{\eps}}}\right).
\]
Next, we bound $C_1^{L(t)}$. It suffices to uniformly bound in time $P(t)$, the total sum of big shock numbers, defined as :
\[
P(t)=\sum_{i:\text{wave}}\ell_i.
\]
To do so, we introduce $S(t)$, the ``small shock potential'', which is defined as follows:
\[
S(t)=\sum_{i:\text{small shock}}\sigma_i.
\]
$S(t)$ tracks the total strength of all small shocks in $\psi_h$ at time $t$. Clearly, $S(0+) \leq V$, and $S(t)$ can only increase via a $B1$-type interaction, in which a big shock is turned into a small shock. When this happens, $S(t)$ increases by the strength of the outgoing small shock, i.e. it increases by less than $\frac{\eps}{2}$. However, in order to have a big shock turn into a small shock the function $\psi_h$ must lose at least $\frac{\eps}{2}$ total variation. It follows that:
\begin{equation}\label{Saccumulationbd}
\sum_{t}\Delta_+S(t) \leq 2V,
\end{equation}
where $\Delta_+S(t)=\max(S(t+)-S(t-),0)$. As $S(t)$ is always non-negative, \eqref{Saccumulationbd} implies:
\begin{equation}\label{Sdiminishbd}
\sum_{t}\Delta_-S(t) \leq \sum_{t}\Delta_+S(t) \leq 2V,
\end{equation}
where $\Delta_-S(t)=-\min(S(t+)-S(t-),0)$. Finally, the only mechanism for $P(t)$ to increase is via an $A4$-type interaction at a time $\overline{t}$, in which case $P(t)$ goes up by $1$. However, $\Delta_-S(\overline{t}) \geq \eps$, so we see that this cannot happen more than $\ceil{\frac{2V}{\eps}}$ times. Thus:
\[
L(t) \leq P(t) \leq \ceil{\frac{2V}{\eps}}.
\]
This gives a uniform bound on $C_1^{L(t)}$. 
\par Lastly, we give a lower bound for $\displaystyle \prod_{k_i \in K(t)}k_i$. We break this part into two steps by partitioning $K(t)$ into two subsets: $K(t)=K_1(t) \bigsqcup K_2(t)$, where $K_1$ consists of numbers added to $K$ from $A3$-type interactions, while $K_2$ consists of numbers added to $K$ from $B2$-type interactions.
\newline \underline{Step 1: Lower bound over $K_1(t)$.} Note that whenever an $A3$-type interaction occurs at a time $\overline{t}$, $\Delta_-S(\overline{t})=\sigma_i$, where $\sigma_i$ is the strength of the incoming small shock. Comparing to \eqref{Sdiminishbd}, we see that:
\[
\sum_{i: (1-C_0\sigma_i) \in K_1(t)}\sigma_i \leq \sum_{t}\Delta_-S(t) \leq \sum_{t}\Delta_+S(t) \leq 2V.
\]
Combining this with Lemma \ref{productlemma}, we have:
\[
\prod_{i: k_i \in K_1(t)}k_i \geq \left(\frac{1}{2}\right)^{8C_0V}.
\]
\underline{Step 2: Lower bound over $K_2(t)$.} There are two different types of weights added to the set $K_2$, corresponding to the two different results of a $B2$-type interaction. We show that in both cases, the weight added minus one is proportional to the decrease in $S(t)$. Firstly, consider a $B2$-type interaction at time $\overline{t}$ between a shock of strength $\sigma_1$ and a rarefaction of strength $\sigma_2$ in which a rarefaction is produced. In this case:
\[
\Delta_-S(\overline{t})=\sigma_1,
\]
while the number added to $K_2(\overline{t}-)$ is $k_i=1-C_0\sigma_1$. Next, consider a $B2$-type interaction at time $\overline{t}$ between a shock of strength $\sigma_1$ and a rarefaction of strength $\sigma_2$ in which a small shock is produced. In this case:
\[
\Delta_-S(\overline{t})=\sigma_2,
\]
while the number added to $K_2(\overline{t}-)$ is $k_i=\frac{(1-C_0\sigma_1)}{1-C_0(\sigma_1-\sigma_2)}$. Then, we compute:
\[
\frac{(1-C_0\sigma_1)}{1-C_0(\sigma_1-\sigma_2)}=1-K \implies K=\frac{C_0\sigma_2}{1-C_0(\sigma_1-\sigma_2)}.
\]
Taking $\eps$ smaller if needed, we may assume that $1-C_0(\sigma_1-\sigma_2) \geq \frac{1}{2}$, in which case we see:
\[
\frac{(1-C_0\sigma_1)}{1-C_0(\sigma_1-\sigma_2)} \geq 1-2C_0\sigma_2.
\]
So, in both cases, we see $k_i \geq 1-2C_0\Delta_-S(\overline{t})$. By \eqref{Sdiminishbd} and Lemma \ref{productlemma}, this implies:
\[
\prod_{i:k_i \in K_2(t)} k_i \geq \left(\frac{1}{2}\right)^{8C_0V}.
\]
This completes the proof of \eqref{abounded}. Now, we show \eqref{amonotone}. We will go through case-by-case and show \eqref{amonotone} for each interaction type separately for an interaction at time $t$. In all cases, assume that the waves interact at $x=x_0$.
\newline \underline{$A1$-type interaction:} let the incoming big shocks have $\ell_i,\ell_j=a,b$ respectively. Then, the outgoing big shock will have $\ell_k=a+b$. Thus, we compute that for $x < x_0$, $a(t+,x)=C_1a(t-,x)$ and for $x > x_0$, $a(t+,x)=a(t-,x)$. 
\newline \underline{$A3$-type interaction:} let the incoming small shock have strength $\sigma_i$. Then, $K(t+)=K(t-) \bigcup \{1-C_0\sigma_1\}$. So, we compute that for $x < x_0$, $a(t+,x)=(1-C_0\sigma_1)a(t-,x)$ and for $x > x_0$, $a(t+,x)=a(t-,x)$.
\newline \underline{$A4$-type interaction:} firstly, assume that a small shock is produced. Then, for $x < x_0$, $a(t+,x)=a(t-,x)$, and for $x > x_0$:
\begin{align*}
a(t+,x)&=a(t+,x_0-)(1-C_0(\sigma_1+\sigma_2)) \leq a(t+,x_0-)(1-C_0\sigma_1)(1-C_0\sigma_2) \\
&=a(t-,x_0-)(1-C_0\sigma_1)(1-C_0\sigma_2)=a(t-,x),
\end{align*}
where $\sigma_1, \sigma_2$ are the strengths of the incoming shocks. On the other hand, if a big shock is produced, then for $x < x_0$, $a(t+,x)=a(t-,x)$ and for $x > x_0$:
\begin{align*}
a(t+,x)&=C_1a(t+,x_0-) \leq a(t+,x_0-)(1-C_0\sigma_1)(1-C_0\sigma_2) \\
&=a(t-,x_0-)(1-C_0\sigma_1)(1-C_0\sigma_2)=a(t-,x).
\end{align*}
\newline \underline{$B1$-type interaction:} firstly, assume that a small shock is produced. Then, for $x < x_0$, $a(t+,x)=a(t-,x)$ and for $x > x_0$, $a(t+,x)=(1-C_0\sigma_1)a(t-,x)$, where $\sigma_1$ is the strength of the small shock produced. On the other hand, assume that a big shock is produced. Then, for any $x \in \R$, $a(t+,x)=a(t-,x)$.
\newline \underline{$B2$-type interaction:} let the incoming small shock have strength $\sigma_1$ and the rarefaction have strength $\sigma_2$. Firstly, assume that a rarefaction is produced. Then, $K(t+)=K(t-)\bigcup \{1-C_0\sigma_1\}$. So, for $x < x_0$, $a(t+,x)=(1-C_0\sigma_1)a(t-,x)$ and for $x > x_0, a(t+,x)=a(t-,x)$. On the other hand, if a small shock is produced, then $K(\overline{t}+)=K(\overline{t}-) \bigcup \{\frac{1-C_0\sigma_1}{1-C_0(\sigma_1-\sigma_2)}\}$. So, for $x < x_0$, $a(t+,x)=\frac{1-C_0\sigma_1}{1-C_0(\sigma_1-\sigma_2)}a(t-,x)$ and for $x > x_0$, $a(t+,x)=a(t-x)$. This completes the proof of \eqref{amonotone} in all cases.
\par Finally, \eqref{abigjump} and \eqref{asmalljump} are verified by definition as the coefficients $C_1^{L(t)}, \displaystyle \prod_{k_i \in K}k_i$ remain constant in $x$ for a fixed time $t$. 

\subsection{Proof of Proposition \ref{closetomaxp}}
We start with a key lemma that will provide the framework to stitch together the $L^2$-stability between wave interactions. The proof exactly follows that of Lemma 2.5 in \cite{MR3954680}, so we omit it.
\begin{lemma}[Lemma 2.5 in \cite{MR3954680}, Lemma 7.1 in \cite{MR4487515}]\label{stitchinglemma}
Let $u \in L^\infty(\R^ \times \R)$ b a weak solution to \eqref{cl} with initial data $u^0$. Further, assume that $u$ is $\eta$-entropic and verifies the Strong Trace property (Definition \ref{strongtrace}). Then, for all $v \in \statespace$, and all $c,d \in \R$ with $c < d$, the approximate left- and right-hand limits:
\begin{align*}
    \aplim_{t \to t_0^\pm}\int_c^d\eta(u(t,x)|v)\diff x,
\end{align*}
exist for all $t_0 \in (0,\infty)$ and verify:
\begin{align*}
\aplim_{t \to t_0^-}\int_c^d\eta(u(t,x)|v)\diff x \geq \aplim_{t \to t_0^+}\int_c^d\eta(u(t,x)|v)\diff x.
\end{align*}
Furthermore, the approximate right-hand limit exists at $t_0=0$ and verifies:
\begin{align*}
\int_c^d\eta(u^0(x)|v)\diff x \geq \aplim_{t \to t_0^+}\int_c^d\eta(u(t,x)|v)\diff x.
\end{align*}
\end{lemma}
The next lemma will control the boundaries of the cone of information.
\begin{lemma}[Lemma 7.2 in \cite{MR4487515}]\label{qcontrolbyeta}
There exists a constant $C > 0$ such that:
\[
|q(a;b)| \leq C\eta(a|b), \text{ for any } (a,b) \in \statespace \times \statespace.
\]
\end{lemma}
\begin{proof}
We have $q(b;b)=\partial_1q(b;b)=0$ for any $b \in \statespace$, so by Lemma \ref{quadraticentropy} and $q'' \in C^0(\statespace)$, there exists a constant $C$ such that:
\[
|q(a;b)| \leq C|a-b|^2 \leq \frac{C}{c^*}\eta(a|b).
\]
\end{proof}
Now, we prove Proposition \ref{closetomaxp}. Fix $T,R > 0$, and $p \in \N$. Fix $v$ to be larger than the maximum speed $\hat{\lambda}$ from Proposition \ref{riemannstabilityshock} and the constant $C$ from Lemma \ref{qcontrolbyeta}. Take $0 < \eps < \frac{1}{2}$ sufficiently small such that Proposition \ref{riemannstabilityshock} is verified. For any initial value $u^0$ and wild solution $u \in \weaktwo$, consider the family of solutions $\psi_h$ produced by the modified front tracking algorithm. We now choose a particular value of $h$. By taking $h$ sufficiently small, we choose for the initial value to satisfy:
\begin{equation}\label{initdataapprox}
||u^0-\psi_h(0,\cdot)||_{L^2(-R,R)} \leq \frac{1}{p}.
\end{equation}
Denote $\psi_h=:\psi$ to be the front tracking function. In particular, it verifies:
\begin{align*}
    ||u(0,\cdot)-\psi(0,\cdot)||_{L^2(-R,R)} &\leq ||u^0-u(0,\cdot)||_{L^2(-R,R)}+\frac{1}{p}, \\   
    TV(\psi(t,\cdot)) &\leq TV(u^0), \\
    \inf u^0 \leq \psi(t,x) &\leq \sup u^0, \\
    ||\psi(t,\cdot)-\psi(s,\cdot)||_{L^1(-R,R)} &\leq TV(u^0)\hat{\lambda}(|t-s|) \text{ (by Lemma \ref{fronttrackingwelldefined})}.
\end{align*}
Thus, the first three properties of Proposition \ref{closetomaxp} are verified by $\psi$. We need only to show the $L^2$-stability. 
\par Consider two successive interaction times $t_j < t_{j+1}$ in the front tracking algorithm that produced $\psi$. Denote the curves of discontinuity between two times $t_j, t_{j+1}$ be $h_1, ..., h_N$, where $N \in \N$, such that:
\begin{align*}
h_1(t) < ... < h_N(t),
\end{align*}
for all $t \in (t_j, t_{j+1})$. We work only within the cone of information, so we define:
\begin{align*}
    h_0(t)&=-R+vt, \\
    h_{N+1}(t)&=R-vt.
\end{align*}
Note that there are no interactions between waves in $\psi$ and the cone of information. For any $t \in [t_j, t_{j+1}]$, note that on $Q=\{(r,x)|t_j < r < t, h_i(r) < x < h_{i+1}(r)\}$, the functions $\psi(r,x)$ and $a(r,x)$ are constant. Consequently, integrating \eqref{relativeentropic} on $Q$, and using the Strong Trace property (Definition \ref{strongtrace}), we obtain:
\begin{align*}
&\aplim_{s \to t_-}\int_{h_i(t)}^{h_{i+1}(t)}a(t-,x)\eta(u(s,x)|\psi(t,x))\diff x \\
&\leq \aplim_{s \to t_j^+}\int_{h_i(t_j)}^{h_{i+1}(t_j)}a(t_j+,x)\eta(u(s,x)|\psi(t_j,x))\diff x +\int_{t_j}^t(F_i^+(r)-F_{i+1}^-(r))\diff r,
\end{align*}
where:
\begin{align*}
&F_i^+(r)=a(r,h_i(r)+)\left(q(u(r,h_i(r)+);\psi(r,h_i(r)+))-\dot{h}_i(r)\eta(u(r,h_i(r)+)|\psi(r,h_i(r)+))\right), \\
&F_i^-(r)=a(r,h_i(r)-)\left(q(u(r,h_i(r)-);\psi(r,h_i(r)-))-\dot{h}_i(r)\eta(u(r,h_i(r)-)|\psi(r,h_i(r)-))\right).
\end{align*}
By summing in $i$, and combining terms corresponding to $i$ in one sum and terms corresponding to $i+1$ in another:
\begin{align*}
&\aplim_{s \to t_-}\int_{-R+vt}^{R-vt}a(t-,x)\eta(u(s,x)|\psi(t,x))\diff x \\
&\leq \aplim_{s \to t_j^+}\int_{R+vt_j}^{R-vt_j}a(t_j+,x)\eta(u(s,x)|\psi(t_j,x))\diff x +\sum_{i=1}^N\int_{t_j}^t(F_i^+(r)-F_{i+1}^-(r))\diff r,
\end{align*}
where we have used that $F_0^+ \leq 0$ and $F_{N+1}^- \geq 0$ due to Lemma \ref{qcontrolbyeta}, the definition of $v$, and $\dot{h}_0=-v=-\dot{h}_{N+1}$.
\par We decompose the second sum on the right-hand side into two terms, one corresponding to shocks and one corresponding to rarefactions. Due to Proposition \ref{riemannstabilityshock} and \eqref{asmalljump}, \eqref{abigjump}, for any $i$ corresponding to a shock:
\[
F_i^+(r)-F_i^-(r) \leq 0, \text{ for almost every } t_j < r < t.
\]
Denote $\mathcal{R}$ the set of $i$ corresponding to rarefactions. Then, by construction, for any $i \in \mathcal{R}$, $a(r,h_i(r)+)=a(r,h_i(r)-)$. Then, from Proposition \ref{rarefactionshockerror}, and by taking $h \leq \frac{1}{TV(u^0)pT}$:
\[
\sum_{i \in \mathcal{R}}\int_{t_j}^t(F_i^+(r)-F_i^-(r))\diff r \leq Ch(t-t_j)\sum_{i \in \mathcal{R}}\sigma_i \leq \frac{C}{pT}(t-t_j).
\]
In total, we find:
\begin{align*}
&\aplim_{s \to t_-}\int_{-R+vt}^{R-vt}a(t-,x)\eta(u(s,x)|\psi(t,x))\diff x \\
&\leq \aplim_{s \to t_j^+}\int_{R+vt_j}^{R-vt_j}a(t_j+,x)\eta(u(s,x)|\psi(t_j,x))\diff x +\frac{C(t-t_j)}{pT}.
\end{align*}
Finally, consider any $0 < t < T$, and denote $0 < t_1, ..., t_J$ the times of wave interactions before $t$, $t_0=0$, and $t_{J+1}=t$. Using Lemma \ref{stitchinglemma} and \eqref{amonotone}, we find:
\begin{align*}
&\int_{-R+vt}^{R-vt}a(t,x)\eta(u(t,x)|\psi(t,x))\diff x -\int_{-R}^Ra(0,x)\eta(u(0,x)|\psi(0,x))\diff x \\
&\leq \aplim_{s \to t^+}\int_{-R+vt}^{R-vt}a(t,x)\eta(u(t,x)|\psi(t,x))\diff x -\int_{-R}^Ra(0,x)\eta(u(0,x)|\psi(0,x))\diff x \\
& \leq \sum_{j=1}^{J+1}\Bigg(\aplim_{s \to t_j^+}\int_{-R+vt_j}^{R-vt_j}a(t_j-,x)\eta(u(t,x)|\psi(t,x))\diff x \\
&\phantom{{}\leq \sum_{j=1}^{J+1}(}-\aplim_{s \to t_{j-1}^+}\int_{-R+vt_{j-1}}^{R-vt_{j-1}}a(t_{j-1}-,x)\eta(u(0,x)|\psi(0,x))\diff x\Bigg) \\
& \leq \sum_{j=1}^{J+1}\Bigg(\aplim_{s \to t_j^-}\int_{-R+vt_j}^{R-vt_j}a(t_j-,x)\eta(u(t,x)|\psi(t,x))\diff x \\
&\phantom{{}\leq \sum_{j=1}^{J+1}(}-\aplim_{s \to t_{j-1}^+}\int_{-R+vt_{j-1}}^{R-vt_{j-1}}a(t_{j-1}-,x)\eta(u(0,x)|\psi(0,x))\diff x\Bigg) \\
& \leq \sum_{j=1}^{J+1}\Bigg(\aplim_{s \to t_j^-}\int_{-R+vt_j}^{R-vt_j}a(t_j-,x)\eta(u(t,x)|\psi(t,x))\diff x \\
&\phantom{{}\leq \sum_{j=1}^{J+1}(}-\aplim_{s \to t_{j-1}^+}\int_{-R+vt_{j-1}}^{R-vt_{j-1}}a(t_{j-1}+,x)\eta(u(0,x)|\psi(0,x))\diff x\Bigg) \\
&\leq \sum_{j=1}^{J+1}\frac{C(t_j-t_{j-1})}{pT} \leq \frac{C}{p}.
\end{align*}
Using  Lemma \ref{quadraticentropy}, \eqref{abounded}, and \eqref{initdataapprox}, we find:
\[
cc^*||\psi(t,\cdot)-u(t,\cdot)||_{L^2(-R+vt, R-vt)}^2 \leq c^{**}||u^0-u(0,\cdot)||_{L^2(-R,R)}^2+\frac{C}{p},
\]
where $c=c(TV(u^0))$ is the constant from Lemma \ref{aproperties}. Taking $p$ sufficiently large so that $\frac{C}{p} < \frac{1}{m}$ grants the result. 
\appendix
\section{One entropy is not enough for concave-convex fluxes}\label{app1}
In this section, we provide an example that shows that for concave-convex fluxes, even if one starts with Riemann initial data localized to the convex region of the flux, one entropy is not enough to grant uniqueness. 
\begin{theorem}
Let $f(u)=u^3, \eta(u)=u^2+e^u$. Then, there exists Riemann initial data with $u_L,u_R > 0$ such that the Cauchy problem for \eqref{cl} has infinitely many distinct $\eta$-entropic solutions. 
\end{theorem}
\begin{proof}
Let $u_L=1$. Then, we compute $\phidiss(u_L) \approx -1.048 < 1$. This implies that $\phi_0^\sharp(u_L) > 0$ (for a definition of $\phi_0^\sharp$, see page 37 of \cite{MR1927887}). So, by Theorem 3.5 in \cite{MR1927887}, there exists a one parameter family of $\eta$-entropic solutions to the Riemann problem for any $u_R \in (0, \phi_0^\sharp(u_L))$, composed of non-classical shocks followed by a classical shock. Of course, the Kru\v{z}kov solution consists of only a single shock.
\end{proof}
\bibliographystyle{plain}
\bibliography{refs}
\end{document}